\documentclass[a4paper,11pt,reqno]{amsart}

\usepackage[english]{babel}
\usepackage[body={16cm,23cm},centering]{geometry} 
\usepackage{enumitem}
\usepackage[colorlinks=true, allcolors=blue]{hyperref} % For clickable links
\usepackage{mathrsfs}                             % For \mathscr font
\usepackage{bm}                                     %bold math symbols

\usepackage{amssymb}
\usepackage{esint} % for the average integral \fint
\usepackage{amsmath}
\usepackage{mathtools} % For \coloneqq

\usepackage{graphicx}
\graphicspath{{figures/}} % Set the path to figures directory
\usepackage{pgfplots}
\pgfplotsset{compat=1.10}
\usepgfplotslibrary{fillbetween}
\usepgfplotslibrary{polar}

\usepackage{tikz}
\usepackage{tikz-3dplot}
\usetikzlibrary{decorations.pathreplacing}
\usetikzlibrary{patterns}

\newtheorem{theorem}{Theorem}
\newtheorem{lemma}{Lemma}
\newtheorem{proposition}{Proposition}
\newtheorem{corollary}{Corollary}
\newtheorem{example}{Example}

\theoremstyle{remark}
\newtheorem{remark}{Remark}

\numberwithin{equation}{section}

\def\R{\mathbb{R}}
\def\Ss{\mathbb{S}}
\def\e{\varepsilon}

\newcommand{\Ar}{\mathbf{A}}
\newcommand{\supp}{\operatorname{supp}}
\newcommand{\diver}{\operatorname{div}} 

\newcommand{\tr}{\operatorname{tr}} % Trace operator
\newcommand{\Hh}{\mathscr{H}} % Mean curvature
\newcommand{\mres}{\mathbin{\vrule height 1.6ex depth 0pt width 0.13ex\vrule height 0.13ex depth 0pt width 1.3ex}}   
\mathchardef\emptyset="001F
\title[]{Double-arc rearrangement and the symmetric isoperimetric problem in Gauss space}
\author[M. Barchiesi, F. Cagnetti and V. Julin]
{M. Barchiesi, F. Cagnetti and V. Julin}
\address[M. Barchiesi]{
Dipartimento di Matematica, Informatica e Geoscienze, Università degli Studi di Trieste,
Via Weiss 2 - 34128 Trieste, Italy.}
\email{barchies@gmail.com}
\address[F. Cagnetti]{Università degli Studi di Parma, Dipartimento di Scienze Matematiche, Fisiche e Informatiche,
Viale Parco Area delle Scienze 53/A
I-43124 Parma,  Italy.}
\email{filippo.cagnetti@unipr.it}
\address[V. Julin]{University of Jyv\"{a}skyl\"{a},
Department of Mathematics and Statistics,
P.O.Box 35 (MaD) FI-40014, Finland}
\email{vesa.julin@jyu.fi}
\date{October 10, 2026}
\begin{document}
\maketitle

\begin{center}
\begin{minipage}{12.3cm}
\small{
\noindent {\bf Abstract.} 
We provide a new rearrangement in Gauss space transforming symmetric sets while decreasing their Gaussian perimeter.
We use this rearrangement to analyze the minimizers of the Gaussian perimeter among symmetric sets. 

\medskip
\noindent {\bf 2020 Mathematics Subject Classification:} 
49Q20, %Variational problems in a geometric measure-theoretic setting
60E15. %Inequalities; stochastic orderings
}
\end{minipage}
\end{center}

\tableofcontents

\section{Introduction}

\noindent
The isoperimetric inequality in Gauss space asserts that among all sets with a given Gaussian measure 
the half-space has the smallest Gaussian perimeter.
Since the half-space is not symmetric, a natural question arises:
``What is the object minimizing the Gaussian perimeter among symmetric sets?''
The question was initially raised in \cite{Barthe} for \emph{coordinate-wise symmetric} sets (also called \emph{unconditional} sets), i.e., 
invariant under reflections with respect to the coordinate hyperplanes. 
We shall use the term \emph{$n$-symmetric} for such sets.

It has been conjectured for the larger class of \emph{centrally symmetric} sets (see \cite[Conjecture 1.3]{He1})
that the answer is a cylinder $B_r^k \times \R^{n-k}$, or its complement, for some $k$ depending on the volume 
and on the dimension. Here $B_r^k$ denotes the $k$-dimensional ball centered at the origin with radius $r$. 
In particular, when the volume is one half the conjecture states that the minimizer is the $n$-dimensional ball $B_r^n$, 
while when the volume is close to one the minimizer is the symmetric slab $(-r,r) \times \R^{n-1}$. 
In dimension one the solution of the conjecture is elementary (Proposition \ref{prop:line_sym}).
The part of the conjecture asserting that the symmetric slab is the minimizer for volumes close to one has been proved in \cite{BJ}.
In \cite{He2}, adapting ideas of Colding and Minicozzi from the study of the mean curvature flow, it has been 
proved that if a minimizer has the form $\Omega\times\R$, then either $\Omega$ or its complement is convex.

For the analogous problem on the hypersphere, the corresponding classification has been
established by Viana~\cite{Viana} (see also Appendix~\ref{app B}). 
The conjecture also resembles the one for the relative isoperimetric problem in the hypercube; 
this one, however, was recently shown to be false in high dimension by Milman~\cite{Milman2025}.

The difficulties in proving the entire conjecture are due to the fact that the approaches developed for the 
minimization of the Gaussian perimeter cannot be applied. For instance, Ehrhard's rearrangement, which is commonly
used in Gauss space, fails to preserve the symmetry constraint. 
This calls for a new kind of rearrangement, one that decreases the Gaussian perimeter while preserving the symmetry constraint.
A possible idea would be to rearrange an $n$-symmetric set by substituting its one-dimensional sections with the minimal
one-dimensional configuration. However, there is an obstruction: the topology of the optimal symmetric set in one dimension 
depends discontinuously on its mass, switching from two symmetric tails to a centered interval as the Gaussian measure crosses $1/2$. 
Consequently, applying this rearrangement to a set whose slice measures cross this threshold induces a topological 
discontinuity, turning the set ``inside out'' at the critical level. Introducing these additional boundaries orthogonal to the 
slicing direction strictly increases the total Gaussian perimeter (see Example~\ref{ex:counterexample}).

We provide a different kind of rearrangement, by adapting a rearrangement introduced by Bonnesen \cite{Bon} in the Euclidean 
setting for convex sets in the plane (see also \cite{CL,Gam}).

To precisely describe this rearrangement we first recall some notation.
Given a Borel set $E\subset\R^n$, $\gamma(E)$ denotes its $n$-dimensional \emph{Gaussian measure}, i.e.,
\begin{equation*}
\gamma(E)\coloneqq\frac{1}{(2\pi)^\frac{n}{2}}\int_E e^{-\frac{|x|^2}{2}}dx,
\end{equation*}
while $\mathcal{H}^k_\gamma(E)$,  $1 \leq k \leq n$, denotes its $k$-dimensional Hausdorff measure with Gaussian weight, i.e., 
\begin{equation*}
\mathcal{H}^k_\gamma(E)\coloneqq\frac{1}{(2\pi)^\frac{k}{2}}\int_{E}e^{-\frac{|x|^2}{2}}d\mathcal{H}^k(x).
\end{equation*}
$P_\gamma (E)$ denotes the \emph{Gaussian perimeter} of $E$, i.e.,
\begin{equation*}
P_\gamma (E) \coloneqq \sqrt{2 \pi} \sup \left\{
\int_E \left( \diver \varphi -  \varphi \cdot (x, y) \right) d \gamma (x, y) : \varphi \in C^1_c (\R^{n}; \R^n) \text{ and } | \varphi | \leq 1
\right\}.
\end{equation*}
If $E$ is a set of finite Gaussian perimeter, one can define the reduced boundary $\partial^* E$ of $E$, which is such that $\partial^* E \subset \partial E$ and $P_\gamma(E)=\mathcal{H}^{n-1}_\gamma(\partial^* E)$ (see \cite{AFP} for more details). 
If $E$ is an open set with Lipschitz boundary, we also have  $P_\gamma(E)=\mathcal{H}^{n-1}_\gamma(\partial E)$.

For points in $\R^n$,  $n \geq 2$, we use two notations: the usual $(x_1,\ldots,x_n)$ for a generic point in $\R^n$,
but also $(x, y)$, with $x = (x_1, x_2) \in \R^2$ and $y = (y_1, \ldots, y_{n-2}) \in \R^{n-2}$, 
in order to focus on the first two components. The latter will be our main notation.
We denote by $e_1, \ldots, e_n$ the vectors of the canonical basis of $\R^n$, 
and $| \cdot |$ stands for the norm in $\R, \R^{n-2}$, or $\R^n$, depending on the context. 
For every $r > 0$, we set \(\partial B (r) \coloneqq \{ x \in \R^2_0 : |x| = r \}\), where
$\R^2_0\coloneqq \R^2 \setminus \{ (0, 0)\}$. 

Given a Borel set $E$, for every $(r, y) \in (0, \infty) \times \R^{n-2}$ we define the 
circular slice $E_{(r, y)} $ as the subset of $\R^2_0$ given by
\begin{equation*}
E_{(r, y)} \coloneqq \{ x \in \R^2_0: x \in \partial B (r) \text{ and } (x, y) \in E\},
\end{equation*}
and define the \emph{circular (Gaussian) distribution} of $E$ as the function
\begin{equation} \label{def:mu}
\mu (r, y)\coloneqq \mathcal{H}^1_{\gamma} ( E_{(r, y)}), \qquad \text{ for every } (r, y) \in (0, \infty) \times \R^{n-2},
\end{equation}
i.e., $\mu (r, y)$ is the $1$-dimensional Gaussian length of the slice $E_{(r, y)}$.
Additionally, we set 
\begin{equation} \label{def:beta}
\xi_\mu (r, y)  \coloneqq \frac{\sqrt{2 \pi}}{e^{- \frac{r^2}{2}}} \frac{\mu (r, y)}{r}, \quad \text{ and } \quad
\beta_\mu (r, y) \coloneqq \frac{\xi_\mu (r, y)}{4}, \quad \text{ for every } (r, y) \in (0, \infty) \times \R^{n-2}.
\end{equation}
Thus, $\xi_\mu (r, y) \in [0, 2 \pi]$ represents the size of an angle corresponding to a connected arc of $\partial B (r)$ whose
$1$-dimensional Gaussian length is $\mu (r, y)$. 
The \emph{double-arc rearrangement} of $E$, denoted $F_\mu$, is defined as 
\begin{equation} \label{def:F}
F_{\mu} \coloneqq \left\{ (x, y) \in \R^2_0 \times \R^{n-2} : \arccos ( | \hat{x} \cdot e_1 | ) <  \beta_\mu (|x|, y) \right\},
\end{equation}
where $\hat{x}\coloneqq x/|x|$ is the radial versor of $x$. 
See Figures \ref{rearrangement 2D} and \ref{rearrangement 3D} for examples with $n=2$ and $n=3$.

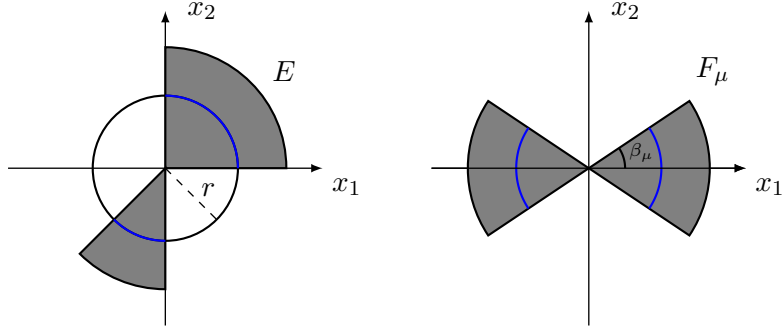
\begin{figure}[ht]
\centering
 \begin{tikzpicture}[>=latex, scale=1.6] 
    \draw[fill=gray, style=thick]
(0,0) -- +(90:1) arc(90:0:1) -- cycle;
    \draw[fill=gray, style=thick]
(0,0) -- +(225:1) arc(225:270:1) -- cycle;
    \draw(.8,.8)node[right]{$E$};
  \draw[->, line width = .5pt] (-1.3,0)--(1.3,0); 
  \draw(1.5,0)node[below]{$x_1$};
 \draw[->, line width = .5pt] (0,-1.3)--(0,1.3); 
  \draw(.1,1.3)node[right]{$x_2$};

 \draw[->, line width = .5pt] (3.5-1.3,0)--(3.5+1.3,0); 
  \draw(3.5+1.5,0)node[below]{$x_1$};
 \draw[->, line width = .5pt] (3.5,-1.3)--(3.5,1.3); 
  \draw(.1+3.5,1.3)node[right]{$x_2$};

\draw[black,thick] (0,0) circle (0.6);
\draw[black, dashed, line width = .5pt] (0,0)--(0.6/1.4,-0.6/1.4); 
\draw(0.3/1.4,-0.25/1.4)node[right]{$r$};
\draw [blue,thick,domain=0:90] plot ({0.6*cos(\x)}, {0.6*sin(\x)});
\draw [blue,thick,domain=225:270] plot ({0.6*cos(\x)}, {0.6*sin(\x)});

   \draw[fill=gray, style=thick]
(3.5,0) -- +(-33.75:1) arc(-33.75:33.75:1) -- cycle;

   \draw[fill=gray, style=thick]
(3.5,0) -- +(180-33.75:1) arc(180-33.75:180+33.75:1) -- cycle;

    \draw(4.3,.8)node[right]{$F_\mu$};
\draw [blue,thick,domain=-33.75:33.75] plot ({3.5+ 0.6*cos(\x)}, {0.6*sin(\x)});
\draw [blue,thick,domain=180-33.75:180+33.75] plot ({3.5+ 0.6*cos(\x)}, {0.6*sin(\x)});

\draw [thick,domain=0:33.75] plot ({3.5+ 0.3*cos(\x)}, {0.3*sin(\x)});
    \draw(3.5+0.25,0.13)node[right]{\tiny{$\beta_\mu$}};
 \draw[->, line width = .5pt] (3.5-1.3,0)--(3.5+1.3,0);  
\end{tikzpicture}
\caption{The double-arc rearrangement, with $n=2$.}
\label{rearrangement 2D}
\end{figure}

\begin{figure}%[hbt!]
\begin{center}
\includegraphics[width=13cm]{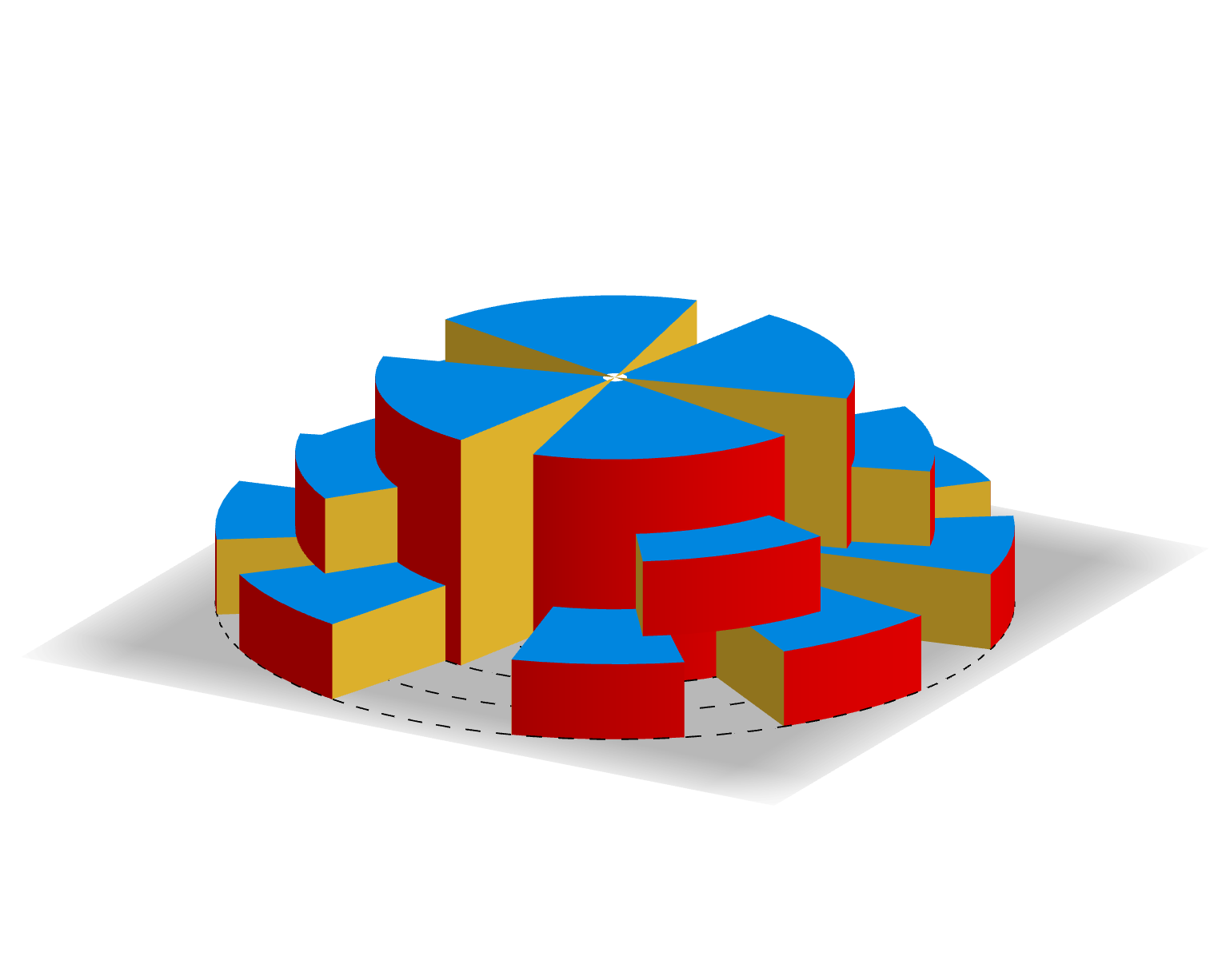}\\[-1.8cm]
\includegraphics[width=13cm]{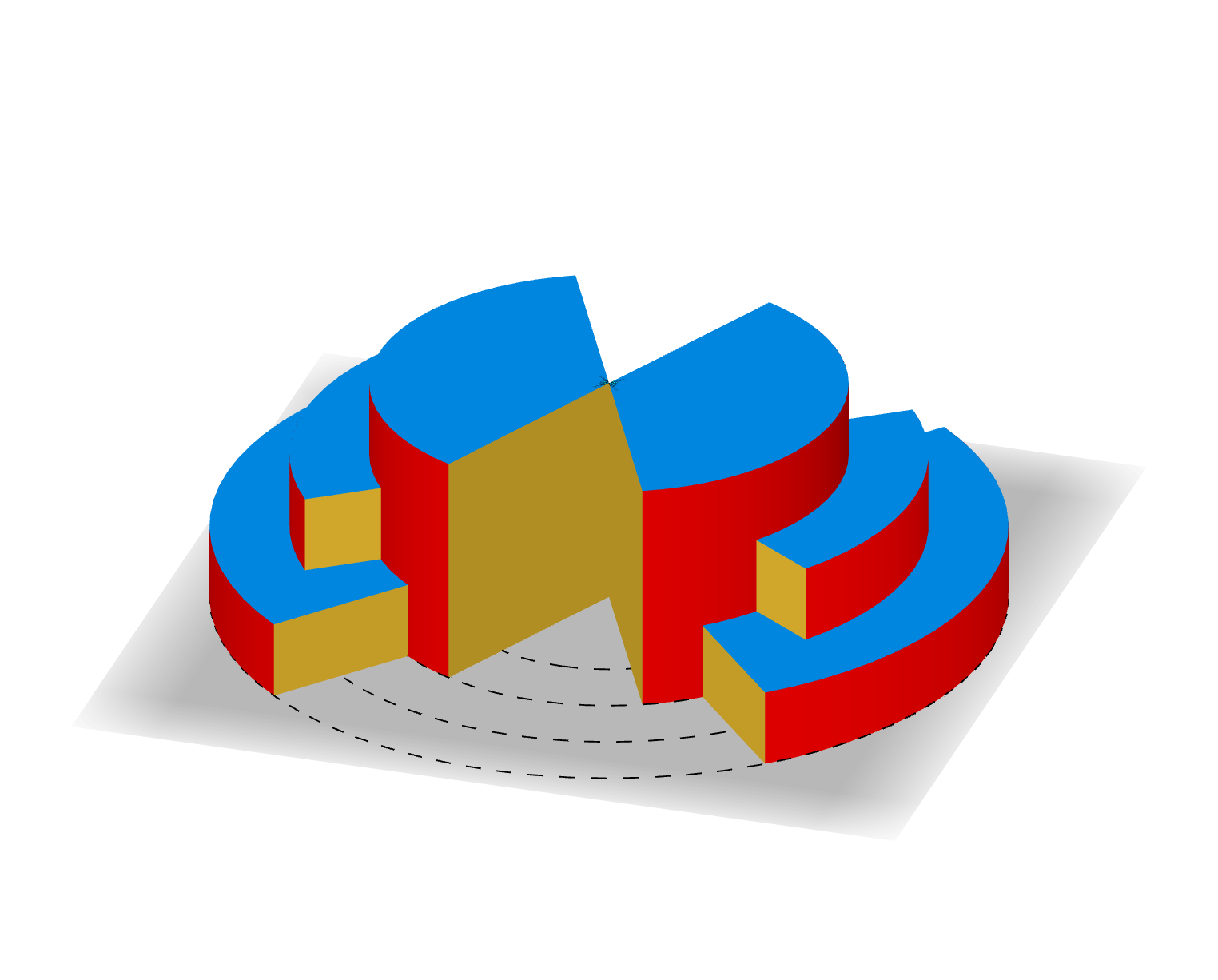}
\end{center}
\vspace*{-1cm}
\caption{A set $E$ and its rearrangement $F_\mu$, with $n=3$.}
\label{rearrangement 3D}
\end{figure}
We say that a Borel set $E \subset \R^n$ is $\mu$-distributed if \eqref{def:mu} holds. By construction, $F_\mu$ is $\mu$-distributed.
As we said, we are interested in sets that are symmetric with respect to all the coordinate hyperplanes.
More precisely, we say that $E$ is \emph{$n$-symmetric} if
\begin{equation*}
 (x, y) \in E \quad \Longleftrightarrow \quad A_j (x, y) \in E \quad \text{ for every } j =1, \ldots, n,
\end{equation*}
where $A_j$ is the $n \times n$ diagonal matrix in which the $j$-th element of the diagonal is $-1$
and all the other ones are $1$, i.e., $A_j$ is the reflection across the coordinate hyperplane 
$\{x\in\R^n \colon x_j=0\}$.

Our first result is that \emph{the double-arc rearrangement decreases the Gaussian perimeter of $n$-symmetric sets}. 
However, we are also interested in what additional properties the initial set $E$ has when $P_\gamma(E)=P_\gamma(F_\mu)$. 
We stress the fact that the double-arc rearrangement does not necessarily  decrease the Gaussian perimeter if the set is not $n$-symmetric.  

In order to state our result in full, we need some other notation. 
For every $x \in \R^2_0$, we define the tangential versor $x_{\parallel}$ as 
\begin{equation*}
x_{\parallel} = \left( - \frac{x_2}{|x|}, \frac{x_1}{|x|} \right).
\end{equation*}

For every $(x, y) \in \partial^* E$, we denote by $\nu^E (x, y)$ the measure-theoretic exterior unit normal 
to $E$ at $(x, y)$, characterized by
\begin{equation*}
D\chi_E=-\nu^E\,\mathcal{H}^{n-1}\mres\partial^*E.
\end{equation*}
We decompose $\nu^E (x, y)$ as $\nu^E (x, y) = (\nu^E_x (x, y), \nu^E_y (x, y))$, 
with $\nu^E_x (x, y) \in \R^2$ and $\nu^E_y (x, y) \in \R^{n-2}$.
If $x \neq 0$, we further decompose $\nu^E_x (x, y) \in \R^2$ as
\begin{equation*}
\nu^E_x (x, y) = \nu^E_{x \perp} (x, y) + \nu^E_{x \parallel} (x, y), 
\end{equation*}
where 
\begin{equation*}
\nu^E_{x \perp} (x, y) \coloneqq (\nu^E_x (x, y) \cdot \hat{x} ) \hat{x}, \quad \text{ and } \quad
\nu^E_{x \parallel} (x, y) \coloneqq \nu^E_x (x, y) - \nu^E_{x \perp} (x, y).
\end{equation*}

By symmetry, it follows that for every $(r, y) \in (0, \infty) \times \R^{n-2}$
the functions 
\begin{equation*}
x \mapsto \hat{x} \cdot \nu^{F_\mu}_x (x, y) \quad \text{ and } \quad 
x \mapsto | \nu_{x \|}^{F_\mu}(x, y)| 
\end{equation*}
are constant in $(\partial^* F_{\mu})_{(r, y)}$. 

Let $r > 0$, $p \in \Ss^{1}$ and $\alpha \in [0, \pi]$ be fixed, where $\Ss^1$ denotes the unit circle in $\R^2$.
The \textit{arc} of centre $r p$ and radius $\alpha$ is the subset of $\partial B (r)$ given by 
\begin{equation*}
\Ar_{\alpha} (r p) \coloneqq \{ x \in \partial B (r) : \arccos ( \hat{x} \cdot p  ) < \alpha \}.
\end{equation*}
Accordingly, if $\alpha \in (0, \pi)$, its \textit{endpoints} are 
\begin{equation*}
\mathbf{S}_{\alpha} (r p) \coloneqq \{ x \in \partial B (r) : \arccos ( \hat{x} \cdot p ) = \alpha \}.
\end{equation*} 
For what follows, it is useful to define the diffeomorphism \( \Phi: (0, \infty) \times \R^{n-2} \times \Ss^1
\to \R^{2}_0 \times \R^{n-2} \) given by
\begin{equation}\label{eq:diffeo}
\Phi (r, y, \omega) \coloneqq (r \omega, y), \quad \text{ for every } (r, y, \omega) \in (0, \infty) \times \R^{n-2} \times \Ss^1.
\end{equation}

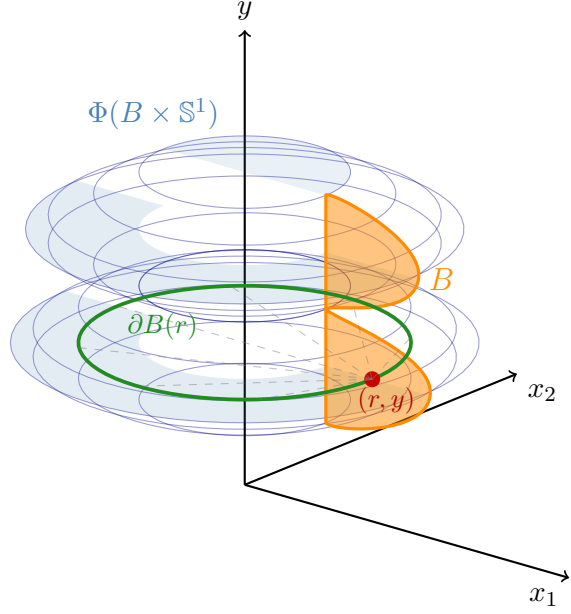
\begin{figure}[ht]
\centering
\tdplotsetmaincoords{70}{40}
\begin{tikzpicture}[tdplot_main_coords, scale=2]

% Define colors
\definecolor{surfaceblue}{RGB}{70,130,180}
\definecolor{borelcolor}{RGB}{255,140,0}
\definecolor{circlegreen}{RGB}{34,139,34}
\definecolor{darkblue}{RGB}{0,0,139}

% --- Coordinate axes ---
\draw[->, thick] (0,0,0) -- (2.8,0,0) node[anchor=north east]{$x_1$};
\draw[->, thick] (0,0,0) -- (0,2.8,0) node[anchor=north west]{$x_2$};
\draw[->, thick] (0,0,0) -- (0,0,3.2) node[anchor=south]{$y$};

% --- Draw Phi(B x S^1) as circles at various heights ---
% Bottom bump circles (y from 0.6 to 1.4)
\foreach \y/\r in {0.6/0.7, 0.75/1.1, 0.9/1.4, 1.0/1.55, 1.1/1.4, 1.25/1.1, 1.4/0.7} {
    \draw[darkblue, opacity=0.4, thin] 
        plot[domain=0:360, samples=60] ({\r*cos(\x)}, {\r*sin(\x)}, {\y});
}

% Top bump circles (y from 1.4 to 2.2)
\foreach \y/\r in {1.4/0.7, 1.55/1.0, 1.7/1.3, 1.8/1.45, 1.9/1.3, 2.05/1.0, 2.2/0.7} {
    \draw[darkblue, opacity=0.4, thin] 
        plot[domain=0:360, samples=60] ({\r*cos(\x)}, {\r*sin(\x)}, {\y});
}

% --- Draw shaded cross-sections to suggest the surface ---
% Bottom torus cross-section (back half, at max radius)
\fill[surfaceblue, opacity=0.15] 
    plot[domain=180:360, samples=40] ({1.55*cos(\x)}, {1.55*sin(\x)}, {1.0})
    -- plot[domain=360:180, samples=40] ({0.7*cos(\x)}, {0.7*sin(\x)}, {1.0})
    -- cycle;

% Top torus cross-section (back half, at max radius)
\fill[surfaceblue, opacity=0.15] 
    plot[domain=180:360, samples=40] ({1.45*cos(\x)}, {1.45*sin(\x)}, {1.8})
    -- plot[domain=360:180, samples=40] ({0.7*cos(\x)}, {0.7*sin(\x)}, {1.8})
    -- cycle;

% Vertical surface strips (to show the swept surface)
\fill[surfaceblue, opacity=0.12]
    (0.7, 0, 0.6) -- (0.7, 0, 2.2) 
    -- plot[domain=0:180, samples=30] ({0.7*cos(\x)}, {0.7*sin(\x)}, {2.2})
    -- (0.7, 0, 2.2) -- cycle;

% --- Draw the B-shaped region in the (x1, y) plane (x2 = 0) ---

% Fill the B shape
\fill[borelcolor, opacity=0.5] 
    (0.7, 0, 0.6) -- (0.7, 0, 2.2)
    % Top bump (semicircle-ish going right then back)
    .. controls (0.7, 0, 2.25) and (1.55, 0, 2.1) .. (1.5, 0, 1.8)
    .. controls (1.45, 0, 1.5) and (0.75, 0, 1.45) .. (0.7, 0, 1.4)
    % Bottom bump (semicircle-ish going right then back)
    .. controls (0.75, 0, 1.35) and (1.65, 0, 1.3) .. (1.6, 0, 1.0)
    .. controls (1.55, 0, 0.7) and (0.75, 0, 0.55) .. (0.7, 0, 0.6)
    -- cycle;

% Draw the B boundary
\draw[borelcolor, line width=1.2pt] 
    (0.7, 0, 0.6) -- (0.7, 0, 2.2)
    % Top bump
    .. controls (0.7, 0, 2.25) and (1.55, 0, 2.1) .. (1.5, 0, 1.8)
    .. controls (1.45, 0, 1.5) and (0.75, 0, 1.45) .. (0.7, 0, 1.4)
    % Bottom bump
    .. controls (0.75, 0, 1.35) and (1.65, 0, 1.3) .. (1.6, 0, 1.0)
    .. controls (1.55, 0, 0.7) and (0.75, 0, 0.55) .. (0.7, 0, 0.6);

% --- Draw a representative circle (green) ---
\draw[circlegreen, line width=1.4pt] 
    plot[domain=0:360, samples=72] ({1.1*cos(\x)}, {1.1*sin(\x)}, {1.0});

% --- Mark a point (r,y) in B ---
\fill[red!80!black] (1.1, 0, 1.0) circle (1.5pt);
\node[anchor=west, text=red!70!black, font=\small] at (0.8, 0.1, 0.75) {$(r,y)$};

% --- Dashed radial lines from point to circle ---
\foreach \angle in {0, 45, 90, 135, 180, 225, 270, 315} {
    \draw[gray, dashed, thin, opacity=0.5] 
        (1.1, 0, 1.0) -- ({1.1*cos(\angle)}, {1.1*sin(\angle)}, 1.0);
}

% --- Labels ---
\node[anchor=west, text=borelcolor, font=\bfseries\large] at (1.5, 0, 1.8) {$B$};
\node[anchor=south east, text=surfaceblue] at (-0.3, 0.2, 2.3) {$\Phi(B \times \Ss^1)$};
\node[anchor=north, text=circlegreen, font=\small] at (-1.8, 1.3, 0.5) {$\partial B(r)$};

\end{tikzpicture}
\caption{The region $\Phi(B \times \Ss^1)$ for $n=3$. The Borel set 
$B \subset (0,\infty) \times \R$ (orange) is displayed in the 
$(x_1, y)$ half-plane. Under the diffeomorphism $\Phi(r,y,\omega) = (r\omega, y)$, 
each point $(r,y) \in B$ generates a circle $\partial B(r)$ at height $y$, 
and the union of these circles   gives the region 
$\Phi(B \times \Ss^1)$ (in blue).}
%\label{fig:phi-B-S1}
\end{figure} 
The first result we prove is that the double-arc Gaussian rearrangement does not increase the Gaussian perimeter. 
Here and in the following, for $1 \leq k \leq n$ and $A, B \subset \R^n$, we write $A =_{\mathcal{H}^k} B$ whenever $\mathcal{H}^k (A \Delta B) = 0$. 
\begin{theorem}[Perimeter inequality] \label{theorem ineq}
Let $E \subset \R^n$ be an $n$-symmetric 
set of finite Gaussian perimeter, and let $\mu$ be the distribution function of $E$.
Then, $\mu \in BV_{\textnormal{loc}} ((0, \infty) \times \R^{n-2})$. 
Moreover, $F_{\mu}$ is a set of finite Gaussian perimeter, and 
\begin{equation} \label{per ineq}
P_{\gamma} (F_{\mu}; \Phi (B \times \Ss^{1})) \leq P_{\gamma} (E; \Phi  (B \times \Ss^{1})),
\end{equation}
for every Borel set $B \subset (0,\infty) \times \R^{n-2}$. 

Moreover, if equality in \eqref{per ineq} holds, then for $\mathcal{H}^{n-1}$-a.e. $(r, y) \in B$ 
the following two conditions are satisfied:
\begin{enumerate}
\item[(i)] $E_{(r, y)} =_{\mathcal{H}^1} \Ar_{\beta_\mu (r, y)} (r p (r, y)) \cup  \Ar_{\beta_\mu (r, y)} (- r p (r, y))$ for some $p(r, y) \in \{(1,0), (0,1)\}$;

\vspace{.2cm}

\item[(ii)] the functions 
\begin{equation*}
x \mapsto \hat{x} \cdot \nu^{E}_x (x, y),  \quad 
x \mapsto | \nu_{x \|}^E(x, y)|, \quad
x \mapsto  \nu_{y}^E(x, y),
\end{equation*}
are constant in $(\partial^*E )_{(r, y)}$. 
\end{enumerate}
\end{theorem} 

Condition (i) means that if the rearrangement does not strictly decrease the Gaussian perimeter,
then the initial set $E$ already has circular slices constituted by two symmetric arcs. 
This is intuitively clear: by looking at Figure \ref{rearrangement 3D}, if a circular slice consists of
more than two arcs, after the rearrangement the yellow part of the boundary is strictly reduced. 

It is important to note that even if $P_\gamma(E)=P_\gamma(F_\mu)$, the sets $E$ and $F_\mu$ are not necessarily congruent. Consider the following example: let $E$ be the union of the four cylinders 
(see Figure~\ref{no rigidity})
\begin{equation*}\begin{split}
&\{x\in\R^3 \colon x_2^2+x_3^2\leq1\}, \qquad
\{x\in\R^3 \colon x_1^2 +(x_3-3)^2\leq1\},\\
&\{x\in\R^3 \colon x_1^2 +(x_3+3)^2\leq1\}, \qquad
\{x\in\R^3 \colon x_1^2+x_2^2\leq1\}.
\end{split}\end{equation*}
The rearrangement reorients the two orange cylinders making them parallel to the blue one. 
The point is that the set 
\begin{equation} \label{def:omega-nu}
\Omega_\mu \coloneqq \{ (r, y) \in (0, \infty) \times \R^{n-2} : 0 < \beta_\mu(r, y) < \pi/2 \}
\end{equation}
is not connected (see Figure \ref{figure distribution}), and this allows us to rotate each cylinder independently, without creating an additional perimeter. 

\begin{figure}
\begin{center}
\includegraphics[width=10cm]{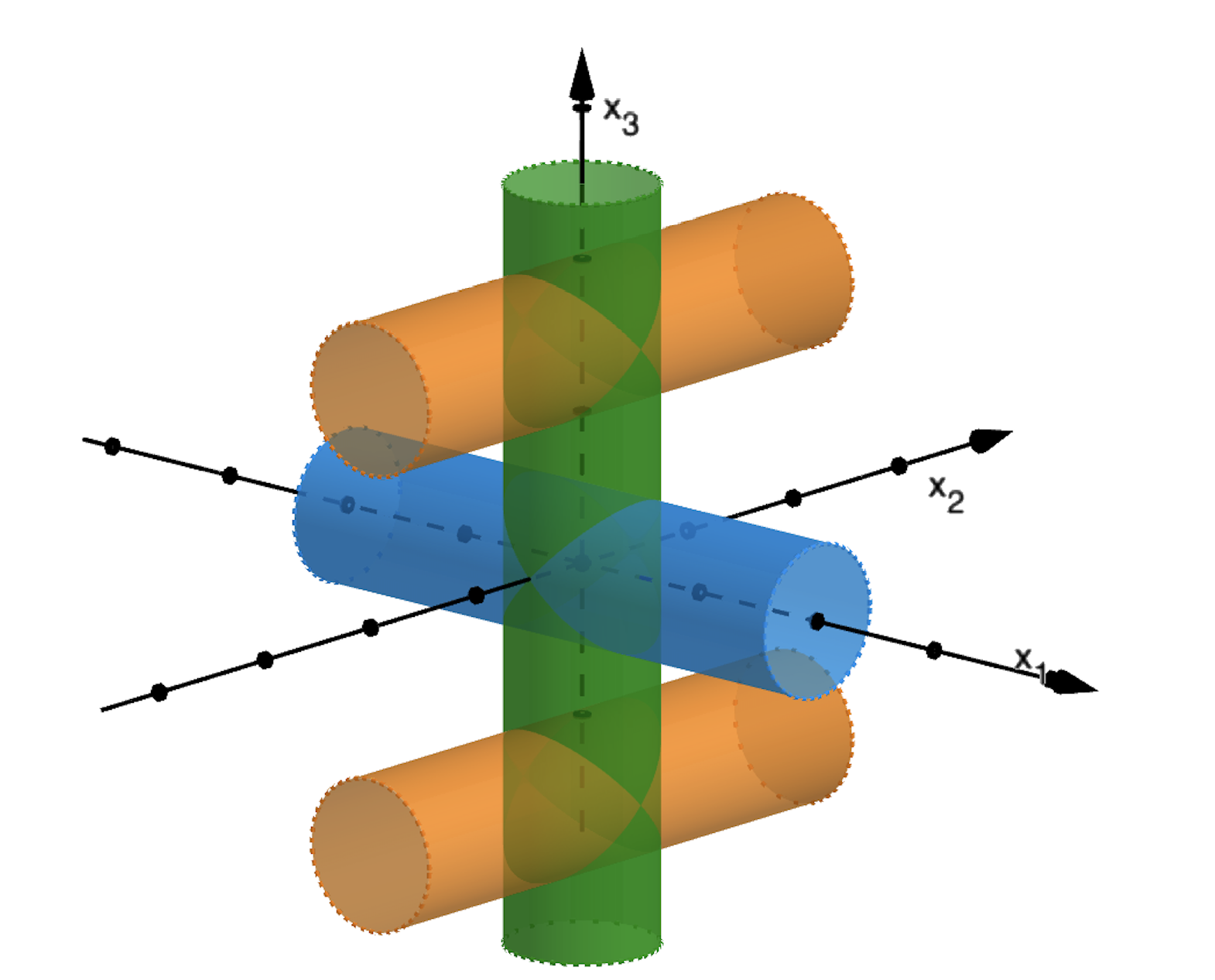}
\end{center}
\caption{The set $E$ and its rearrangement $F_\mu$ have the same perimeter, but they are not congruent}
\label{no rigidity}
\end{figure}

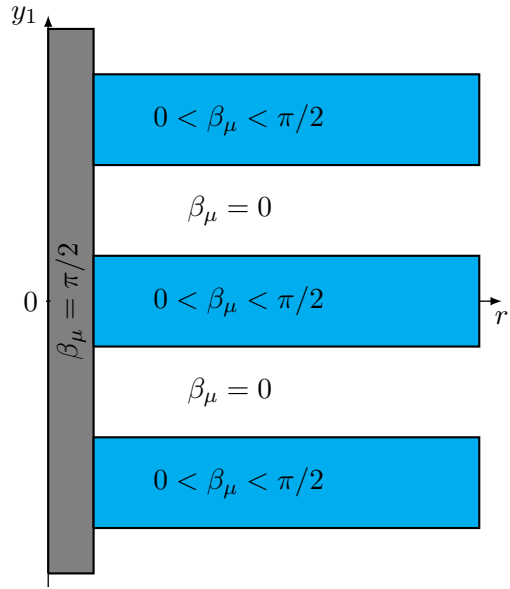
\begin{figure}[!htb]
\centering
  \begin{tikzpicture}[>=latex, scale=0.6]
 \draw[fill=gray, style=thick](0,-6)--(1,-6)--(1,6)--(0,6)--(0,-6);
 \draw[fill=cyan, style=thick](1,3)--(9.5,3)--(9.5,5)--(1,5)--(1,3);
\draw(4.2,4)node{$ 0 < \beta_{\mu} < \pi/2$};

\begin{scope}[shift={(0,-4)}]
 \draw[fill=cyan, style=thick](1,3)--(9.5,3)--(9.5,5)--(1,5)--(1,3);
\draw(4.2,4)node{$ 0 < \beta_{\mu} < \pi/2$};
\end{scope}

\begin{scope}[shift={(0,-8)}]
 \draw[fill=cyan, style=thick](1,3)--(9.5,3)--(9.5,5)--(1,5)--(1,3);
\draw(4.2,4)node{$ 0 < \beta_{\mu} < \pi/2$};
\end{scope}

\draw(0.5,0)node[rotate=90]{$ \beta_{\mu} = \pi/2$};
\draw(4,2)node{$\beta_{\mu} = 0$};
\draw(4,-2)node{$\beta_{\mu} = 0$};

   \draw[line width = .5pt] (-.05,0)--(.05,0); 
  \draw(0,0)node[left]{$0$};
  \draw[->, line width = .5pt] (9.5,0)--(10,0); 
  \draw(10,0)node[below]{$r$};
 \draw[->, line width = .5pt] (0,-6.3)--(0,6.3); 
  \draw(0,6.3)node[left]{$y_1$};
\end{tikzpicture}
\caption{Rigidity fails. In this case the set 
$\{ \beta_{\mu} = 0 \} \cup \{ \beta_{\mu} = \pi/2 \}$ disconnects the set $\{ 0 < \beta_{\mu} < \pi/2 \}$.}
       \label{figure distribution}
\end{figure}

\smallskip

We now  turn our attention to minimizers of the Gaussian isoperimetric problem for $n$-symmetric sets, and study their regularity and the behavior under the double-arc Gaussian  rearrangement. We recall that, for $t \in [0,1]$, the set $E^{(t)}$ of points of density $t$ of $E$ is given by 
\[
E^{(t)}\coloneqq  \left\{ x \in \R^n: u_E(x) \coloneqq  \lim_{r \to 0^+} \frac{|E \cap B_r(x)|}{|B_r(x)|} = t \right\}. 
\]
In what follows, we identify $E$ with its precise representative $E^{(1)}$.
We first prove the optimal regularity. 

\begin{theorem}[Regularity]\label{theorem reg}
Let $E\subset\R^n$ be a minimizer of the Gaussian isoperimetric problem among $n$-symmetric sets 
with prescribed Gaussian measure. Then $E$ is open, and $\partial E$ is a real analytic hypersurface
outside the origin, i.e., $\Sigma_E\coloneqq \partial E\setminus\partial^*E\subset\{0\}$.
\end{theorem}

When $E$ is a minimizer, the set $\Omega_\mu $ defined in \eqref{def:omega-nu}  is connected and we have the  
rigidity in the  double-arc rearrangement.

\begin{theorem}[Rigidity] \label{theorem rig}
Let $E \subset \R^n$ be a minimizer of the Gaussian isoperimetric problem among $n$-symmetric 
sets with prescribed Gaussian measure, identified with its precise representative $E^{(1)}$, 
let $\mu$ be the circular distribution of $E$ defined in \eqref{def:mu}, and set $F\coloneqq F_\mu^{(1)}$. 
Then the set $\Omega_\mu$ defined in \eqref{def:omega-nu} is open and connected (possibly empty), and 
\[
E = F \qquad\text{or}\qquad E = R(F),
\]
where $R$ is the clockwise rotation of $\pi/2$ in the plane $(x_1, x_2)$. 
Moreover, exactly one of the following alternatives holds:
\begin{enumerate}[label=(\roman*)]
\item $E$ is rotationally invariant in the $(x_1, x_2)$-plane, and $\Omega_\mu=\emptyset$;
\item $\Omega_\mu\neq\emptyset$, 
$\Pi\big(\partial E\cap(\R^2_+\times\R^{n-2})\big)=\Omega_\mu$, and $\nu^{E}_{x\parallel}(x,y)\neq 0$
for every $(x,y)\in \partial E\cap(\R^2_+\times\R^{n-2})$.
\end{enumerate}
Here $\Pi(x,y)\coloneqq(|x|,y)$ and $\R^2_+\coloneqq\{x\in\R^2\colon x_1>0,\ x_2>0\}$.
\end{theorem}

\smallskip
We believe that the equality $E = F_\mu$ (up to a rotation) should hold for general $n$-symmetric 
sets of finite Gaussian perimeter satisfying $P_\gamma(E)=P_\gamma(F_\mu)$, 
provided that $\Omega_\mu$ is ``connected''. The assumption that $E$ is a minimizer guarantees
that $\Omega_\mu$ is an open set and provides the regularity needed for a topological argument.
In the general case the connectedness assumption on $\Omega_\mu$ should be replaced by a 
suitable notion of \emph{essential connectedness}. The appropriate framework 
is developed in \cite{CCDPM, CCDPM2},  where essential connectedness is used to establish rigidity 
for Steiner and  Gaussian symmetrization. Adapting their approach to the double-arc 
rearrangement would require formulating essential connectedness for the set 
$\Omega_\mu$ in $(0,\infty)\times\R^{n-2}$ and replacing the topological 
covering argument with a measure-theoretic one. We do not pursue this 
generalization here.

\smallskip
An obstruction prevents the generalization of the double-arc rearrangement to higher dimensions, where one might 
naturally attempt to foliate $\mathbb{R}^n$ ($n \ge 3$) with spheres instead  of circles. 
Indeed, see Proposition \ref{prop:sphere_sym}, the solution to the isoperimetric problem on the sphere $\Ss^2$ among symmetric sets undergoes 
a topological phase transition depending on the mass: it switches from the union of two antipodal caps to an equatorial belt as 
the surface area crosses half the total measure. Consequently, rearranging a set by substituting its spherical slices with the optimal 
configuration could introduce topological discontinuities, increasing the global 
Gaussian perimeter rather than decreasing it (see Example \ref{ex:counterexample2}).

\medskip
The paper is organized as follows. Section \ref{sect prelimin} collects preliminary results including the coarea formula 
and properties of the distribution function. Section \ref{sect perim ineq} proves the perimeter inequality (Theorem \ref{theorem ineq}). 
Section \ref{sect reg} establishes the regularity theory (Theorem \ref{theorem reg}). 
Section~\ref{sect rig} proves the rigidity result (Theorem \ref{theorem rig}). 
Appendices treat the one-dimensional problem (\ref{app A}), the spherical obstruction (\ref{app B}), and the planar conjecture (\ref{app C}).

%%%%%%%%%%% Section 2
\section{Preliminary results}\label{sect prelimin}

To prove Theorem~\ref{theorem ineq} we need some technical results that are stated below. The first two are results of geometric measure theory.

\begin{proposition}[Coarea Formula] \label{prop:coarea}
Let $E$ be a set of finite Gaussian perimeter in $\R^n$, 
and let $g:\R^n \rightarrow [0,\infty]$ be a Borel function. Then, 
\begin{align*}
&\int_{\partial^* E} g(x, y) |\nu^E_{x \parallel}(x ,y)| d\mathcal{H}^{n-1}_{\gamma} (x, y) \\ 
&= \int_{(0, + \infty) \times \R^{n-2}} \frac{e^{- \frac{r^2 + |y|^2}{2}}}{(2 \pi)^{(n-1)/2}} \left( \int_{(\partial^* E)_{(r, y)}} g(x, y) \, d\mathcal{H}^{0}(x) \right) dr dy.
\end{align*}
\end{proposition}

\begin{proof}
We have 
\begin{align*}
&\int_{\partial^* E} g(x, y) |\nu^E_{x \parallel}(x ,y)| d\mathcal{H}^{n-1}_{\gamma} (x, y) \\
&= \int_{\partial^* E} g(x, y) \frac{e^{- \frac{|x|^2 + |y|^2}{2}}}{(2 \pi)^{(n-1)/2}} |\nu^E_{x \parallel}(x ,y)| d\mathcal{H}^{n-1} (x, y) \\
&= \int_{(0, + \infty) \times \R^{n-2}} \frac{e^{- \frac{r^2 + |y|^2}{2}}}{(2 \pi)^{(n-1)/2}} \left( \int_{(\partial^* E)_{(r, y)}} g(x, y) \, d\mathcal{H}^{0}(x) \right) dr dy,
\end{align*}
where the last equality follows by applying coarea formula \cite[Remark~2.94]{AFP} 
with $N = n-1$, $M = n$, $k = n-1$, and $f (x, y) = (|x|, y)$.
\end{proof}

\begin{remark} \label{Z has zero measure}
Let $E$ be a set of finite Gaussian perimeter in $\R^n$. Then, setting $Z = \partial^*E \cap \{\nu^E_{x \parallel} = 0\}$ and $\Pi (x, y) = (|x|, y)$, from the Coarea formula it follows that
\[
\mathcal{H}^{n-1}(\Pi (Z)) = 0.
\]
Indeed, 
\[
0 = \int_{Z} |\nu^E_{x \parallel}| \, d \mathcal{H}^{n-1}
= \int_{\Pi (Z)} \mathcal{H}^0 (Z_{(r, y)}) \, dr \, dy \geq \mathcal{H}^{n-1}(\Pi (Z)).
\]

\end{remark}

The following result will also be  useful, see \cite[Theorem~6.2]{CPS}.  \begin{proposition}[Vol'pert] \label{volpert theorem sets}
Let $E \subset \R^n$ be an $n$-symmetric 
set of finite Gaussian perimeter and let $\mu$ be the distribution function of $E$.
Then, for $\mathcal{H}^{n-1}$-a.e. $(r,y)\in(0,\infty) \times \R^{n-2}$ the following properties hold:
  \begin{enumerate}[label=(\roman*)]
  \item[(i)] $E_{(r,y)}$ is a set of finite perimeter in $\partial B (r)$ and $\partial^* \left( E_{(r,y)} \right) = ( \partial^{*} E)_{(r,y)}$;
  \item[(ii)] for every $\omega \in \Ss^{1}$ such that $r \omega \in (\partial^{*} E)_{(r,y)} \cap \partial^* (E_{(r,y)})$, 
  $$| \nu^E_{x\parallel} (r \omega , y) | > 0 \,\,\text{ and }\,\, 
  \nu^{E_{(r,y)}} (r \omega) = \frac{\nu^E_{x\parallel} (r \omega,y)}{| \nu^E_{x\parallel} (r \omega,y)|}.$$
   \end{enumerate}
Here $\nu^{E_{(r,y)}}$ denotes the measure-theoretic exterior unit normal to the slice $E_{(r,y)}$ within $\partial B(r)$.
In particular, there exists a Borel set $G_E\subset \{ 0<\beta_\mu <\pi/2 \}$ with
$\mathcal{H}^{n-1} (\{0<\beta_\mu <\pi/2\}\setminus G_{E}) = 0$ such that (i)-(ii) are fulfilled for all $(r,y)\in G_E$.
\end{proposition}

\begin{remark}
We observe that, although \cite[Theorem~6.2]{CPS} was formulated for sets of finite Euclidean perimeter, it still holds for sets of \textit{locally} finite Euclidean perimeter.  Therefore the result can be applied to sets of finite Gaussian perimeter.    
\end{remark}

\medskip
We also need some technical results about the radial and tangential components of functions and measures.
If $v: \R^n \to \R^n$, we decompose $v$ as $v(x, y) = (v_{x}(x, y), v_y(x, y))$ with $v_x(x, y) \in \R^2$ and $v_y(x, y) \in \R^{n - 2}$.
Moreover, for every $(x, y) \in \R^2_0 \times \R^{n-2}$, we decompose $v_x(x, y) \in \R^2$ as follows:
\begin{equation*}
\begin{aligned}
v_x(x, y) &= v_{x \perp}(x, y) + v_{x \parallel}(x, y), \\
\text{where } v_{x \perp}(x, y) &\coloneqq  (v_x(x, y) \cdot \hat{x})\hat{x}, \\
\text{and } v_{x \parallel}(x, y) &\coloneqq  v_x(x, y) - v_{x \perp}(x, y).
\end{aligned}
\end{equation*}

If $\varphi \in C^1_c(\R^2 \times \R^{n-2}; \R^n)$, we decompose the divergence of $\varphi$ as
\[
\diver \varphi(x, y) = \diver_x \varphi_x(x, y) + \diver_y \varphi_y(x, y), 
\]
where $\diver_x$ and $\diver_y$ denote the divergence with respect to the $x$ and $y$ variables, respectively.

If $\varphi \in C^1_c(\R^2_0 \times \R^{n-2}; \R^n)$, for every $(x, y) \in \R^2_0 \times \R^{n-2}$
we define the parallel divergence (with respect to $x$)  as:
\begin{equation}\label{eq:defin_div_parallel}
\diver_{x \parallel} \varphi_x(x, y) \coloneqq  \diver_x \varphi_x(x, y) - \big(\nabla_x \varphi_x(x, y)\hat{x}\big) \cdot \hat{x},
\end{equation}
(note that $\nabla_x \varphi_x(x, y)$ is a $2 \times 2$ matrix and $\hat{x} \in \R^2$, 
so $\nabla_x \varphi_x(x, y)\hat{x} \in \R^2$).

The next lemma is a minor variation of \cite[Lemma~3.2]{CPS}.
\begin{lemma}\label{lem:divergence_components}
Let $\varphi \in C^1_c(\R^2_0 \times \R^{n-2}; \R^n)$. Then, for every $(x, y) \in \R^2_0 \times \R^{n-2}$
\begin{align} 
(\diver_x \varphi_{x \perp})(x, y) &= \big(\nabla_x \varphi_x(x, y)\hat{x}\big) \cdot \hat{x} 
+ \frac{1}{|x|}\left(\varphi_x(x, y) \cdot \hat{x}\right), \label{eq:n-1_term} \\ 
(\diver_x \varphi_{x \parallel})(x, y) &= (\diver_{x \parallel} \varphi_{x \parallel})(x, y).  \label{eq:identity_1}
\end{align}
\end{lemma}

\begin{remark}\label{rem:divergence}
Let $\varphi \in C^1_c(\R^2_0 \times \R^{n-2}; \R^n)$. Recalling that $\varphi_x = \varphi_{x \perp} + \varphi_{x \parallel}$, combining \eqref{eq:n-1_term} and \eqref{eq:identity_1} we have that for every $(x, y) \in \R^2_0 \times \R^{n-2}$
\[
(\diver_x \varphi_x)(x, y) = \big(\nabla_x \varphi_x(x, y)\hat{x}\big) \cdot \hat{x} 
+ \frac{1}{|x|}\left(\varphi_x(x, y) \cdot \hat{x}\right) + \diver_{x \parallel} \varphi_{x \parallel}(x, y).
\]
\end{remark}

\begin{proof}[Proof of Lemma \ref{lem:divergence_components}]
First of all, note that 
\begin{equation}\label{eq:nabla_phi}
\nabla_x \left(\varphi_x(x, y) \cdot \hat{x}\right) = (\nabla_x \varphi_x(x, y))^T \hat{x} + \frac{1}{|x|}\varphi_{x \parallel}(x, y).
\end{equation}

Indeed, 
\begin{align} 
&\nabla_x \left(\varphi_x(x, y) \cdot \hat{x}\right)
= (\nabla_x \varphi_x(x, y))^T \hat{x} 
+ (\nabla_x \hat{x})^T \varphi_x(x, y) \nonumber \\
&\hspace{.2cm}= (\nabla_x \varphi_x(x, y))^T \hat{x} 
+ \frac{I_2 - \hat{x} \otimes \hat{x}}{|x|}\varphi_x(x, y)  
= (\nabla_x \varphi_x(x, y))^T \hat{x} + \frac{1}{|x|}\varphi_{x \parallel}(x, y), \label{eq:gh}
\end{align}
where $I_2$ represents the identity map in $\R^2$, and $\hat{x} \otimes \hat{x}$
is the usual tensor product of $\hat{x}$ with itself (so that $I_2 - \hat{x} \otimes \hat{x}$
is the orthogonal projection on the tangent line to $\Ss^{1}$ at $\hat{x}$).

Thanks to \eqref{eq:nabla_phi} and \eqref{eq:gh}, we have
\begin{align*}
\diver_x \varphi_{x \perp}(x, y) 
&= \diver_x \left((\varphi_x(x, y) \cdot \hat{x})\hat{x}\right) \\
&= \nabla_x \left(\varphi_x(x, y) \cdot \hat{x}\right) \cdot \hat{x} + \left(\varphi_x(x, y) \cdot \hat{x}\right) \diver_x \hat{x} \\
&= \left[(\nabla_x \varphi_x(x, y))^T \hat{x} + \frac{1}{|x|}\varphi_{x \parallel}(x, y)\right] \cdot \hat{x} + \left(\varphi_x(x, y) \cdot \hat{x}\right)\frac{1}{|x|} \\
&= \big(\nabla_x \varphi_x(x, y)\hat{x}\big) \cdot \hat{x} + \left(\varphi_x(x, y) \cdot \hat{x}\right)\frac{1}{|x|}, 
\end{align*}
which proves \eqref{eq:n-1_term}.

By \eqref{eq:defin_div_parallel}, it follows that 
\begin{align} \label{eq:intermediate}
\diver_x \varphi_x(x, y) 
= \diver_{x \parallel} \varphi_x(x, y) + \big(\nabla_x \varphi_x(x, y)\hat{x}\big) \cdot \hat{x}.
\end{align}

On the other hand, recalling that $\varphi_x(x, y) = \varphi_{x \parallel}(x, y) + \varphi_{x \perp}(x, y)$
and using \eqref{eq:n-1_term} 
\begin{align*}
\diver_x \varphi_x(x, y) &= \diver_x \varphi_{x \parallel}(x, y) + \diver_x \varphi_{x \perp}(x, y) \\
&= \diver_x \varphi_{x \parallel}(x, y) + \big(\nabla_x \varphi_x(x, y)\hat{x}\big) \cdot \hat{x} 
+ \frac{1}{|x|}\left(\varphi_x(x, y) \cdot \hat{x}\right).
\end{align*}

Comparing the last identity with \eqref{eq:intermediate} we obtain that for every $\varphi \in C^1_c(\R^2_0 \times \R^{n-2}; \R^n)$
\begin{align*}
\diver_{x \parallel} \varphi_x(x, y) + \big(\nabla_x \varphi_x(x, y)\hat{x}\big) \cdot \hat{x}
&= \diver_x \varphi_{x \parallel}(x, y) + \big(\nabla_x \varphi_x(x, y)\hat{x}\big) \cdot \hat{x} \\
&\quad + \frac{1}{|x|}\left(\varphi_x(x, y) \cdot \hat{x}\right), 
\end{align*}
from which it follows that
\begin{equation*}
\diver_{x \parallel} \varphi_x(x, y) 
= \diver_x \varphi_{x \parallel}(x, y) 
+ \frac{1}{|x|}\left(\varphi_x(x, y) \cdot \hat{x}\right).
\end{equation*}
Applying the last identity to the function $\varphi_{x \parallel}$ we obtain \eqref{eq:identity_1}. 
\end{proof}

If $\mu$ is an $\R^n$-valued Radon measure on $\R^2_0 \times \R^{n-2}$, 
we decompose $\mu$ as $\mu = (\mu_x, \mu_y)$, where 
$\mu_x$ and $\mu_y$ are $\R^{2}$-valued and $\R^{n-2}$-valued, respectively.
Moreover, we decompose $\mu_x$ as $\mu_x = \mu_{x \perp} + \mu_{x \parallel}$, where $\mu_{x \perp}$ and $\mu_{x \parallel}$
are the $\R^2$-valued Radon measures on $\R^2_0 \times \R^{n-2}$ such that
\begin{align*}
\int_{\R^2_0 \times \R^{n-2}} \varphi(x, y) \cdot d\mu_{x \perp}(x, y)
&= \int_{\R^2_0 \times \R^{n-2}} \varphi_{x \perp}(x, y) \cdot d\mu_x(x, y), \\
\int_{\R^2_0 \times \R^{n-2}} \varphi(x, y) \cdot d\mu_{x \parallel}(x, y) 
&= \int_{\R^2_0 \times \R^{n-2}} \varphi_{x \parallel}(x, y) \cdot d\mu_x(x, y),
\end{align*}
for every $\varphi \in C_c(\R^2_0 \times \R^{n-2}; \R^n)$.

Note that $\mu_{x \perp}$ and $\mu_{x \parallel}$ are well defined by Riesz representation Theorem (see, for instance, \cite[Theorem 1.54]{AFP}).
In the special case $\mu = Df$, with $f \in BV_{\textnormal{loc}}(\R^2_0 \times \R^{n-2})$, we shorten  the notation writing
$D_{x \parallel}f$ and $D_{x \perp}f$ in place of $(Df)_{x \parallel}$ and $(Df)_{x \perp}$, respectively.
In particular, if $f = \chi_E$ and $E \subset \R^n$ is a set of finite perimeter, by De Giorgi's structure theorem 
we have 
\begin{equation}\label{eq:de_giorgi_structure}
\begin{aligned}
D_{x \perp}\chi_E &= -\nu^E_{x \perp} \mathcal{H}^{n-1} \mres \partial^* E, \\
D_{x \parallel}\chi_E &= -\nu^E_{x \parallel} \mathcal{H}^{n-1} \mres \partial^* E, \\
D_y\chi_E &= -\nu^E_y \mathcal{H}^{n-1} \mres \partial^* E.
\end{aligned}
\end{equation}

The next lemma is a variant of \cite[Lemma~3.4]{CPS}.
\begin{lemma}
Let $f \in BV_{\textnormal{loc}}(\R^2_0 \times \R^{n-2})$. Then, for every $\varphi \in C^1_c(\R^2_0 \times \R^{n-2}; \R^n)$
\begin{align}
\int_{\R^2_0 \times \R^{n-2}} \varphi(x, y) \cdot dD_{x \parallel}f 
&= -\int_{\R^2_0 \times \R^{n-2}} f(x, y) \, \diver_{x \parallel}\varphi_{x \parallel}(x, y) \, dx \, dy, \label{eq:4_3dominik} 
\end{align}
and 
\begin{align}
\int_{\R^2_0 \times \R^{n-2}} \varphi(x, y) \cdot dD_{x \perp}f 
&= -\int_{\R^2_0 \times \R^{n-2}} f(x, y) \big(\nabla_x\varphi_x(x, y)\hat{x}\big) \cdot \hat{x} \, dx \,dy \nonumber \\
&\quad -\int_{\R^2_0 \times \R^{n-2}} f(x, y) \frac{1}{|x|}\left(\varphi_x(x, y) \cdot \hat{x}\right) \, dx \, dy. \label{eq:4_4dominik}
\end{align}
\end{lemma}

\begin{proof}
Let $\varphi \in C^1_c(\R^2_0 \times \R^{n-2}; \R^n)$.
By definition of $D_{x \parallel}f$ and thanks to \eqref{eq:identity_1} we have 
\begin{align*}
\int_{\R^2_0 \times \R^{n-2}} \varphi(x, y) \cdot dD_{x \parallel}f(x, y)
&= \int_{\R^2_0 \times \R^{n-2}} \varphi_{x \parallel}(x, y) \cdot dD_x f(x, y)\\
&= -\int_{\R^2_0 \times \R^{n-2}} \diver_x\varphi_{x \parallel}(x, y)f(x, y) \, dx \, dy\\
&= -\int_{\R^2_0 \times \R^{n-2}} \diver_{x \parallel}\varphi_{x \parallel}(x, y)f(x, y) \, dx \, dy, 
\end{align*}
and this shows \eqref{eq:4_3dominik}.

Similarly, by definition of $D_{x \perp}f$
\begin{align*}
\int_{\R^2_0 \times \R^{n-2}} \varphi(x, y) \cdot dD_{x \perp}f(x, y)
&= \int_{\R^2_0 \times \R^{n-2}} \varphi_{x \perp}(x, y) \cdot dD_x f(x, y) \\
&= -\int_{\R^2_0 \times \R^{n-2}} \diver_x\varphi_{x \perp}(x, y)f(x, y) \, dx \, dy.
\end{align*}
Thanks to \eqref{eq:n-1_term}, identity \eqref{eq:4_4dominik} follows.
\end{proof}

An immediate consequence of identity \eqref{eq:4_3dominik} is the following.
\begin{corollary}\label{cor:parallel_total_variation}
Let $f \in BV_{\textnormal{loc}}(\R^2_0 \times \R^{n-2})$ and let $\Omega \Subset \R^2_0 \times \R^{n-2}$ be open and bounded.
Then,
\begin{equation*}
\left|D_{x \parallel}f \right|(\Omega)= \sup \left\{\int_{\R^2_0 \times \R^{n-2}}f(x, y) \, \diver_{x \parallel}\varphi_{x \parallel}(x, y) dx \, dy:
\; \varphi \in C^1_{c}(\Omega;\R^n),\, \|\varphi \|_{L^\infty(\Omega; \R^n)} \leq 1\right\}.
\end{equation*}
\end{corollary}

\begin{lemma}\label{lem:mu BV}
Let $E \subset \R^n$ be an $n$-symmetric 
set of finite Gaussian perimeter, let $\mu$ be the distribution function of $E$, 
and let $\xi_\mu$ be defined by \eqref{def:beta}.
Then:
\begin{enumerate}[label=(\roman*)]
\item $\mu, \xi_\mu \in BV_{\textnormal{loc}}((0, \infty) \times \R^{n-2})$;

\vspace{.1cm}

\item For bounded Borel functions with compact support, the following holds:
\begin{align} 
\int_{(0, \infty)\times \R^{n-2}} \psi(r, y) \,
r \, dD_r\xi_\mu(r, y)
&= \int_{\R^2_0 \times \R^{n-2}} \psi(|x|, y) \, \hat{x} \cdot dD_{x \perp}\chi_E(x, y), \label{eq:formula_for_D_xi} \\
\int_{(0, \infty)\times \R^{n-2}} r \, \phi(r, y) \,
\cdot dD_y\xi_\mu(r, y)
&= \int_{\R^2_0 \times \R^{n-2}} \phi(|x|, y) \cdot dD_{y}\chi_E(x, y),  \label{eq:formula_2_for_D_xi}
\end{align}
where $\psi: (0, \infty) \times \R^{n-2} \to \R$ and $\phi: (0, \infty) \times \R^{n-2} \to \R^{n-2}$.
\item As a consequence:
\begin{equation}\label{eq:bound_for_D_xi}
|r e^{-\frac{r^2}{2}} e^{-\frac{|y|^2}{2}} \, D\xi_\mu|(B) 
\leq |e^{-\frac{|x|^2}{2}} e^{-\frac{|y|^2}{2}} D\chi_E|(\Phi(B \times \Ss^{1})), 
\end{equation}
for every Borel set $B \Subset (0, \infty) \times \R^{n-2}$. 
\end{enumerate}
\end{lemma}

\begin{proof}
We divide the proof into several steps.

\medskip
\noindent
\textbf{Step 1:} We show that $\mu \in BV_{\textnormal{loc}}((0, \infty) \times \R^{n-2})$.

Let $A \Subset (0, \infty) \times \R^{n-2}$ be an open set, and let us show that 
$\mu \in BV(A)$.
Let $r_A, R_A > 0$ be such that 
\begin{equation*}
r \geq r_A \, \text{ and } \,
r^2 + |y|^2 \leq R^2_A
\end{equation*}
for every $(r, y) \in A$.

\medskip
\noindent
\textbf{Step 1a:} We show that $\mu, \xi_\mu \in L^1(A)$.

We have
\begin{align*}
1 &\geq \gamma(E \cap \Phi(A \times \Ss^{1}))
= \int_{\Phi(A \times \Ss^{1})} \chi_E(x, y)\, d\gamma(x, y) \\
&= \frac{1}{(2 \pi)^{\frac{n}{2}}} \int_{\Phi(A \times \Ss^{1})} e^{-\frac{|x|^2+|y|^2}{2}} \chi_E(x, y) \, dx dy \\
&= \frac{1}{(2 \pi)^{\frac{n-1}{2}}} \int_A  e^{-\frac{|y|^2}{2}} \left[ \frac{e^{-\frac{r^2}{2}}}{\sqrt{2 \pi}} \int_{\partial B(r)} \chi_E(x, y)\, d\mathcal{H}^1(x) \right] \, dr dy \\
&= \frac{1}{(2 \pi)^{\frac{n-1}{2}}} \int_A  e^{-\frac{|y|^2}{2}} \mu(r, y) \, dr dy  
\geq \frac{e^{-\frac{R_A^2}{2}}}{(2 \pi)^{\frac{n-1}{2}}} \|\mu\|_{L^1(A)},  
\end{align*}
so that $\|\mu\|_{L^1(A)} < \infty$. Moreover, 
\begin{equation*}
\|\xi_\mu\|_{L^1(A)}
= \sqrt{2 \pi} \int_A \frac{e^{\frac{r^2}{2}}}{r} \mu(r, y) \, dr dy 
\leq \sqrt{2 \pi} \frac{e^{\frac{R_A^2}{2}}}{r_A} \|\mu\|_{L^1(A)} < \infty,
\end{equation*}
which shows that $\|\xi_\mu\|_{L^1(A)} < \infty$.

\medskip
\noindent
\textbf{Step 1b:} We show that $|D_r\mu|(A) < \infty$.
Let $\psi \in C^1_c (A)$ with $|\psi| \leq 1$. 
Applying formula \eqref{eq:n-1_term} to the function $\psi(|x|, y) \hat{x}$
and observing that $(\psi(|x|, y) \hat{x})_{x \perp} = \psi(|x|, y) \hat{x}$,
we obtain that for every $(x, y) \in \R^2_0 \times \R^{n-2}$ 
\begin{align*}
\diver_x \left(\psi(|x|, y) \hat{x}\right)
&= \left[\nabla_x \left(\psi(|x|, y) \hat{x}\right) \hat{x}\right] \cdot \hat{x} 
+  \frac{1}{|x|} \left[\psi(|x|, y) \hat{x} \cdot \hat{x}\right]  \nonumber \\
&= \left[\left(\frac{\partial\psi}{\partial r}(|x|, y) \, \hat{x} \otimes \hat{x} 
+ \psi(|x|, y) \frac{I_2 - \hat{x} \otimes \hat{x}}{|x|}\right) \hat{x}\right] \cdot \hat{x}  
+ \frac{1}{|x|} \psi(|x|, y)  \nonumber \\
&= \frac{\partial\psi}{\partial r}(|x|, y) + \frac{1}{|x|} \psi(|x|, y), 
\end{align*}
where $I_2$ denotes the identity matrix in $\R^2$.

Thus, 
\begin{align*}
\diver_x \Big(e^{-\frac{|x|^2}{2}} \psi(|x|, y) \hat{x}\Big) 
&= \nabla_x \Big(e^{-\frac{|x|^2}{2}}\Big) \cdot \psi(|x|, y) \hat{x}  
+ e^{-\frac{|x|^2}{2}} \diver_x \left(\psi(|x|, y) \hat{x}\right) \\
&= e^{-\frac{|x|^2}{2}} \left[
\frac{\partial\psi}{\partial r}(|x|, y) + \left(\frac{1-|x|^2}{|x|}\right) \psi(|x|, y)
\right]. 
\end{align*}

Therefore, 
\begin{align} 
\frac{1}{(2 \pi)^{\frac{n-1}{2}}} \int_{(0, \infty)\times \R^{n-2}} 
&\psi(r, y) \, dD_r\mu(r, y) \nonumber \\
&= \frac{1}{(2 \pi)^{\frac{n}{2}}} \int_{\R^2_0 \times \R^{n-2}} 
\left(\frac{1-|x|^2}{|x|}\right)
e^{-\frac{|x|^2}{2}} \psi(|x|, y) \chi_E(x, y) \, dx \, dy \nonumber \\
&- \frac{1}{(2 \pi)^{\frac{n}{2}}} \int_{\partial^* E} e^{-\frac{|x|^2}{2}} \psi(|x|, y) \, \hat{x} \cdot 
\nu^E_{x \perp} d\mathcal{H}^{n-1}(x, y)
\label{eq:distr_derivative_useful}  
\\
&\leq \int_{\R^2_0 \times \R^{n-2}} e^{\frac{|y|^2}{2}}
\frac{|1-|x|^2|}{|x|} \, |\psi(|x|, y)| \, \chi_E(x, y) \, d\gamma(x, y) \nonumber \\
&+ \frac{1}{\sqrt{2 \pi}} \int_{\partial^* E} e^{\frac{|y|^2}{2}} \, |\psi(|x|, y)| \, d\mathcal{H}^{n-1}_\gamma(x, y) \nonumber \\
&\leq e^{\frac{R_A^2}{2}} \left[M_A + \frac{1}{\sqrt{2 \pi}} P_\gamma(E)\right] < \infty,  \nonumber 
\end{align}
where we set 
\begin{equation*}
M_A = \max_{s \in [r_A, R_A]} \frac{|1-s^2|}{s}.
\end{equation*}

Taking the supremum over $\psi$ we conclude. 

\medskip
\noindent
\textbf{Step 1c:} We conclude the proof that $\mu \in BV(A)$, showing that
\begin{equation*}
|D_y\mu|(A) < \infty.
\end{equation*}

Let $i \in \{1, \ldots, n-2\}$ and let $\psi \in C^1_c (A)$ with $|\psi| \leq 1$. We have
\begin{align}
\frac{1}{(2 \pi)^{\frac{n-1}{2}}} \int_{(0, \infty)\times \R^{n-2}} 
&\psi(r, y)\, dD_{y_i}\mu(r, y) \nonumber \\
&= -\frac{1}{(2 \pi)^{\frac{n-1}{2}}} \int_{(0, \infty)\times \R^{n-2}} 
\frac{\partial\psi}{\partial y_i}(r, y) \, \mu(r, y) \, dr \, dy \nonumber \\
&= -\frac{1}{(2 \pi)^{\frac{n-1}{2}}} \int_{(0, \infty)\times \R^{n-2}} 
\frac{\partial\psi}{\partial y_i}(r, y) \, \frac{e^{-\frac{r^2}{2}}}{\sqrt{2 \pi}} 
\left(\int_{\partial B(r)} \chi_E(x, y) \, d\mathcal{H}^{1}(x)\right) \, dr \, dy \nonumber \\
&= -\frac{1}{(2 \pi)^{\frac{n}{2}}} \int_{\R^2_0 \times \R^{n-2}}  
e^{-\frac{|x|^2}{2}} \frac{\partial\psi}{\partial y_i}(|x|, y) \chi_E(x, y) \, dx \, dy \nonumber \\
&= \frac{1}{(2 \pi)^{\frac{n}{2}}} \int_{\R^2_0 \times \R^{n-2}}  
e^{-\frac{|x|^2}{2}} \psi(|x|, y) \, dD_{y_i}\chi_E(x, y) \label{eq:pol} \\
&= -\frac{1}{(2 \pi)^{\frac{n}{2}}} \int_{\partial^* E}  
e^{-\frac{|x|^2}{2}} \psi(|x|, y) \, \nu^E_{y_i}(x, y) \, d\mathcal{H}^{n-1}(x, y) \nonumber\\
&= -\frac{1}{\sqrt{2 \pi}} \int_{\partial^* E}  
e^{\frac{|y|^2}{2}} \psi(|x|, y) \, \nu^E_{y_i}(x, y) \, d\mathcal{H}^{n-1}_{\gamma}(x, y),\nonumber
\end{align}
and therefore
\begin{align*}
\frac{1}{(2 \pi)^{\frac{n-1}{2}}} \Bigg|\int_{(0, \infty)\times \R^{n-2}} 
&\psi(r, y)\, dD_{y_i}\mu(r, y)\Bigg|\nonumber \\
&=
\frac{1}{\sqrt{2\pi}}
\left|
\int_{\partial^*E}
e^{\frac{|y|^2}{2}}
\psi(|x|,y)\,\nu^E_{y_i}(x,y)\,
d\mathcal H^{n-1}_\gamma(x,y)
\right|
\nonumber\\
&\le
\frac{1}{\sqrt{2\pi}}
\int_{\partial^*E}
e^{\frac{|y|^2}{2}}
|\psi(|x|,y)|\,|\nu^E_{y_i}(x,y)|\,
d\mathcal H^{n-1}_\gamma(x,y)
\nonumber\\
&\le
\frac{e^{\frac{R_A^2}{2}}}{\sqrt{2\pi}}
P_\gamma(E)
<\infty .
\end{align*}

Taking the supremum over $\psi$ we conclude. 

\medskip
\noindent
\textbf{Step 2:} We show that $\xi_\mu \in BV_{\textnormal{loc}}((0, \infty) \times \R^{n-2})$.

We already proved that $\xi_\mu \in L^1_{\textnormal{loc}}((0, \infty) \times \R^{n-2})$.
Let us now show that $D\xi_\mu$ is a locally bounded Radon measure in $(0, \infty) \times \R^{n-2}$.
We recall that 
\begin{equation*}
\mu(r, y) = \frac{r e^{-\frac{r^2}{2}}}{\sqrt{2 \pi}} \xi_\mu(r, y).
\end{equation*}

Since the function $(r, y) \mapsto 1/(r e^{-\frac{r^2}{2}})$ is smooth and locally bounded in $(0, \infty) \times \R^{n-2}$,
by the chain rule in $BV$ (see \cite[Example~3.97]{AFP})
we have that $\xi_\mu \in BV_{\text{loc}}((0, \infty)\times \R^{n-2})$, with 
\begin{align} 
D_r\mu &= \frac{e^{-\frac{r^2}{2}}}{\sqrt{2 \pi}}
\left[(1-r^2)\xi_\mu \, dr \, dy + rD_r\xi_\mu\right] \nonumber \\
&= \frac{1-r^2}{r}\mu(r, y) \, dr \, dy + \frac{e^{-\frac{r^2}{2}}}{\sqrt{2 \pi}} rD_r\xi_\mu, \label{eq:Dv} \\
D_y\mu &= \frac{r e^{-\frac{r^2}{2}}}{\sqrt{2 \pi}} D_y\xi_\mu(r, y). \label{eq:Dy_mu}
\end{align} 

\medskip
\noindent
\textbf{Step 3:} We show formulas \eqref{eq:formula_for_D_xi} and \eqref{eq:formula_2_for_D_xi}.

Let $\psi \in C^1_c((0, \infty) \times \R^{n-2})$. From \eqref{eq:Dv} it follows that 
\begin{align} 
\int_{(0, \infty)\times \R^{n-2}} 
\psi(r, y) \, dD_r\mu(r, y)   
&= \int_{(0, \infty)\times \R^{n-2}} \psi(r, y)
\frac{1-r^2}{r}\mu(r, y) \, dr \, dy \nonumber \\
&+ \int_{(0, \infty)\times \R^{n-2}} \psi(r, y)
\frac{e^{-\frac{r^2}{2}}}{\sqrt{2 \pi}} \,
r \, dD_r\xi_\mu(r, y). \label{eq:this_one}
\end{align}

On the other hand, from \eqref{eq:distr_derivative_useful} we have
\begin{align*} 
 \int_{(0, \infty)\times \R^{n-2}} 
&\psi(r, y) \, dD_r\mu(r, y)    
= \frac{1}{\sqrt{2 \pi}} \int_{\R^2_0 \times \R^{n-2}} 
\Big(\frac{1-|x|^2}{|x|}\Big)
e^{-\frac{|x|^2}{2}} \psi(|x|, y) \chi_E(x, y) \, dx \, dy  \\
&+ \frac{1}{\sqrt{2 \pi}} \int_{\R^2_0 \times \R^{n-2}} e^{-\frac{|x|^2}{2}} \psi(|x|, y) \, \hat{x} \cdot dD_{x \perp}\chi_E(x, y) \\
=&  \int_{(0, \infty) \times \R^{n-2}} 
\psi(r, y) \frac{1-r^2}{r} 
\frac{e^{-\frac{r^2}{2}}}{\sqrt{2 \pi}} \int_{\partial B(r)} \chi_E(x, y) d\mathcal{H}^1(x)\, dr \, dy  \\
&+ \frac{1}{\sqrt{2 \pi}} \int_{\R^2_0 \times \R^{n-2}} e^{-\frac{|x|^2}{2}} \psi(|x|, y) \, \hat{x} \cdot dD_{x \perp}\chi_E(x, y) \\
=&  \int_{(0, \infty) \times \R^{n-2}} 
\psi(r, y) \frac{1-r^2}{r}\mu(r, y) \, dr \, dy  \\
&+ \frac{1}{\sqrt{2 \pi}} \int_{\R^2_0 \times \R^{n-2}} e^{-\frac{|x|^2}{2}} \psi(|x|, y) \, \hat{x} \cdot dD_{x \perp}\chi_E(x, y).
\end{align*}

Combining the previous identity and \eqref{eq:this_one},
\begin{align*}
\int_{(0, \infty)\times \R^{n-2}} e^{-\frac{r^2}{2}} \psi(r, y)
\,
r \, dD_r\xi_\mu(r, y)
= \int_{\R^2_0 \times \R^{n-2}} e^{-\frac{|x|^2}{2}} \psi(|x|, y) \, \hat{x} \cdot dD_{x \perp}\chi_E(x, y),
\end{align*}
for every $\psi \in C^1_c((0, \infty) \times \R^{n-2})$.
Applying the identity above to the function $e^{\frac{r^2}{2}}\psi(r, y)$ we obtain that
\begin{align*}
\int_{(0, \infty)\times \R^{n-2}} \psi(r, y) \,
r \, dD_r\xi_\mu(r, y)
= \int_{\R^2_0 \times \R^{n-2}} \psi(|x|, y) \, \hat{x} \cdot dD_{x \perp}\chi_E(x, y),
\end{align*}
for every $\psi \in C^1_c((0, \infty) \times \R^{n-2})$.

By approximation, the identity above is true also when $\psi$ is a bounded Borel function, 
and this gives \eqref{eq:formula_for_D_xi}. 

Let us now show \eqref{eq:formula_2_for_D_xi}.
Let $\phi \in C^1_c((0, \infty) \times \R^{n-2}, \R^{n-2})$.
Thanks to \eqref{eq:pol} and \eqref{eq:Dy_mu}, 
\begin{align*}
\int_{(0, \infty)\times \R^{n-2}} r \, \phi(r, y) \cdot dD_y\xi_\mu(r, y) 
&= \sqrt{2 \pi} \int_{(0, \infty)\times \R^{n-2}} e^{\frac{r^2}{2}} \phi(r, y) \,
\cdot dD_y\mu(r, y)\\
&= \int_{\R^2_0 \times \R^{n-2}}  
\phi(|x|, y) \cdot dD_{y}\chi_E(x, y).
\end{align*}
By approximation, the formula also holds whenever $\phi$ is a bounded Borel function. 
\end{proof}

The next result follows by combining Lemma~\ref{lem:mu BV}, Coarea Formula, and Remark~\ref{Z has zero measure}.
\begin{lemma}\label{lem:nabla_xi}
Let $E \subset \R^n$ be an $n$-symmetric 
set of finite Gaussian perimeter, let $\mu$ be the distribution function of $E$, 
and let $\xi_\mu$ be defined by \eqref{def:beta}.
Then
\begin{align*}
(r \, D_r\xi_\mu)(B) 
&=-\int_{\partial^* E \cap \Phi(B \times \Ss^{1})\cap\{\nu^{E}_{x \|}=0 \}}
\hat{x} \cdot \nu^E_x(x, y) \, d\mathcal{H}^{n-1}(x, y)  \\
&\quad-\int_{B} \int_{(\partial^* E)_{(r, y)} \cap \{\nu_{x \|}^E \neq 0 \}} 
\frac{\hat{x} \cdot \nu^E_x(x, y)}{|\nu_{x \|}^E(x, y)|}d\mathcal{H}^{0}(x) \, dr \, dy,
\end{align*}
for every Borel set $B \Subset (0, \infty)\times \R^{n-2}$. 
Moreover, 
\begin{equation*}
r \partial_r \xi_\mu(r, y) = -\int_{(\partial^* E)_{(r, y)} \cap \{\nu^{E}_{x \|} \neq 0 \}} 
\frac{\hat{x} \cdot \nu^E_x(x, y)}{|\nu_{x \|}^E(x, y)|}d\mathcal{H}^{0}(x),
\end{equation*}
for $\mathcal{H}^{n-1}$-a.e. $(r, y) \in (0, \infty)\times \R^{n-2}$, 
where $\partial_r \xi_\mu$ denotes the approximate differential 
of $\xi_\mu$ with respect to $r$. Similarly, 
\begin{align*}
(rD_{y}\xi_\mu)(B) 
&= -\int_{\partial^* E \cap \Phi(B \times \Ss^{1})\cap\{\nu^{E}_{x \|}=0 \}}
\nu^E_y(x, y) \, d\mathcal{H}^{n-1}(x, y) 
\\
&\quad-\int_{B} \int_{(\partial^* E)_{(r, y)} \cap \{\nu^{E}_{x \|} \neq 0 \}} 
\frac{\nu^E_y(x, y)}{|\nu_{x \|}^E(x, y)|}d\mathcal{H}^{0}(x) \, dr \, dy,
\end{align*}
for every Borel set $B \Subset (0, \infty)\times \R^{n-2}$, and 
\begin{equation*}
r \nabla_{y}\xi_\mu(r, y) = -\int_{(\partial^* E)_{(r, y)} \cap \{\nu^{E}_{x \|} \neq 0 \}} 
\frac{\nu^E_y(x, y)}{|\nu_{x \|}^E(x, y)|}d\mathcal{H}^{0}(x),
\end{equation*}
for $\mathcal{H}^{n-1}$-a.e. $(r, y) \in (0, \infty)\times \R^{n-2}$, 
where $\nabla_{y}\xi_\mu$ denotes the approximate gradient 
of $\xi_\mu$ with respect to $y$.
\end{lemma}

\begin{remark}\label{rem:constant}
Applying Lemma~\ref{lem:nabla_xi} to $F_\mu$ it follows that for every $(r, y) \in G_{F_\mu}$
\begin{align*}
r \partial_r\xi_\mu(r, y) &= -4\frac{\hat{x} \cdot \nu^{F_\mu}_x(x, y)}{|\nu_{x \|}^{F_{\mu}}(x, y)|} \\
r \nabla_{y}\xi_\mu(r, y) &= -4\frac{\nu^{F_\mu}_y(x, y)}{|\nu_{x \|}^{F_{\mu}}(x, y)|}
\end{align*}
for every $x$ such that $(x, y) \in \partial^* F_\mu$, 
where we have used the fact that $\mathcal{H}^{0}((\partial^* F_\mu)_{(r, y)}) = 4$.
\end{remark}

\begin{proposition}\label{prop:bound P F_mu}
Let $E \subset \R^n$ be an $n$-symmetric 
set of finite Gaussian perimeter, and let $\mu$ be the distribution function of $E$.
Then, $F_\mu$ is a set of finite Gaussian perimeter in $\R^n$. 
Moreover, for every Borel set $B \Subset (0, \infty) \times \R^{n-2}$ 
\begin{equation}\label{eq:smhts}
(2 \pi)^{\frac{n-1}{2}}P_\gamma(F_\mu;\Phi(B\times \Ss^{1}))
\leq \left|r e^{-\frac{r^2}{2}} e^{-\frac{|y|^2}{2}} D\xi_{\mu}\right|(B) 
+ \left|e^{-\frac{|x|^2}{2}} e^{-\frac{|y|^2}{2}} D_{x \parallel}\chi_{F_\mu}\right|
(\Phi(B\times \Ss^{1})).
\end{equation}
\end{proposition}

\begin{proof}
The proof is based on the arguments of \cite[Lemma 3.5]{CCF} and \cite[Lemma 3.3]{BCF}.
Thanks to Lemma~\ref{lem:mu BV}, $\mu, \xi_\mu \in BV_{\text{loc}}((0, \infty) \times \R^{n-2})$.
Let $\{\xi_j\}_{j \in \mathbb{N}} \subset C^1_c((0, \infty) \times \R^{n-2})$ be a sequence of non-negative functions 
with $0 \leq \xi_j(r, y) \leq 2 \pi$
such that $\xi_j \to \xi_\mu$ $\mathcal{H}^{n-1}$-a.e. in $(0, \infty) \times \R^{n-2}$ 
and $|D\xi_j| \overset{*}{\rightharpoonup} |D\xi_\mu|$.
For every $j \in \mathbb{N}$, set
\[
\mu_j (r, y) \coloneqq \frac{r e^{- \frac{r^2}{2}}}{\sqrt{2 \pi}} \xi_j (r, y), \quad \text{ and } \quad
\beta_j (r, y) \coloneqq \frac{\xi_j (r, y)}{4}, \quad \text{ for every } (r, y) \in (0, \infty) \times \R^{n-2},
\]
and let us denote by $F_{\mu_j} \subset \R^n$
the set defined by \eqref{def:F}, with $\mu_j$ in place of $\mu$. Note that, according to the notation given in \eqref{def:beta}, for every $j \in \mathbb{N}$ we have $\xi_j = \xi_{\mu_j}$ and $\beta_j = \beta_{\mu_j}$.

Let $A \Subset (0,\infty) \times \R^{n-2}$ be open, 
and let $\varphi \in C^{1}_{c}(\Phi(A\times \Ss^{1});\R^n)$ with 
$\|\varphi\|_{L^{\infty}(\Phi(A\times \Ss^{1});\R^n)} \leq 1$. 
We have
\begin{align}
\int_{F_{\mu_j}} &\left((\diver \varphi)(x, y) - \varphi(x, y) \cdot (x, y) \right) e^{-\frac{|x|^2}{2}} e^{-\frac{|y|^2}{2}} \, dx \, dy \nonumber \\
&= \int_{\Phi(A\times \Ss^{1})} \chi_{F_{\mu_j}}(x, y)
e^{-\frac{|x|^2}{2}} e^{-\frac{|y|^2}{2}} \, \Big(\diver \varphi(x, y) - x \cdot \varphi_x(x, y) - y \cdot \varphi_y(x, y)\Big) 
dx \, dy \nonumber \\
&=  \int_{\Phi(A\times \Ss^{1})} \chi_{F_{\mu_j}}(x, y) 
e^{-\frac{|x|^2}{2}} e^{-\frac{|y|^2}{2}} \, \Big(\diver_y \varphi_y(x, y) - y \cdot \varphi_y(x, y)\Big)  
dx \, dy \nonumber \\
&+  \int_{\Phi(A\times \Ss^{1})} \chi_{F_{\mu_j}}(x, y)
e^{-\frac{|x|^2}{2}} e^{-\frac{|y|^2}{2}} \, \, \Big(\diver_x \varphi_x(x, y) - x \cdot \varphi_x(x, y)\Big) 
dx \, dy. \nonumber 
\end{align} 

Let us now focus on the last integral.
Thanks to Remark~\ref{rem:divergence}, this can be written as
\begin{align}
& \int_{\Phi(A\times \Ss^{1})} \chi_{F_{\mu_j}}(x, y)
e^{-\frac{|x|^2}{2}} e^{-\frac{|y|^2}{2}} \, \, \Big(\diver_x \varphi_x(x, y) - x \cdot \varphi_x(x, y)\Big) 
dx \, dy \nonumber \\
&= \int_{\Phi(A\times \Ss^{1})} \chi_{F_{\mu_j}}(x, y)
e^{-\frac{|x|^2}{2}} e^{-\frac{|y|^2}{2}} \, \diver_{x \parallel} \varphi_{x \parallel}(x, y) 
dx \, dy \nonumber \\
&+  \int_{\Phi(A\times \Ss^{1})} \chi_{F_{\mu_j}}(x, y)
e^{-\frac{|x|^2}{2}} e^{-\frac{|y|^2}{2}} \, \big(\nabla_x \varphi_x(x, y)\hat{x}\big) \cdot \hat{x} \, 
dx \, dy \label{eq:equality} \\
&+ \int_{\Phi(A\times \Ss^{1})} \chi_{F_{\mu_j}}(x, y)
e^{-\frac{|x|^2}{2}} e^{-\frac{|y|^2}{2}} \, \varphi_x(x, y) \cdot \left(\frac{\hat{x}}{|x|} - x\right) dx \, dy. \nonumber
\end{align} 

In the following, it is convenient to introduce the function $V_j : A \to \R^{n-1}$
given by
\begin{align*}
V_j(r, y) &\coloneqq  e^{-\frac{r^2}{2}} e^{-\frac{|y|^2}{2}} \int_{\Ar_{\beta_{\mu_j}(r, y)}(r e_1) \cup \Ar_{\beta_{\mu_j}(r, y)}(-r e_1)} 
\big(\varphi_x(x, y) \cdot \hat{x}, \varphi_y(x, y)\big) \, d\mathcal{H}^{1}(x) \\
&= r e^{-\frac{r^2}{2}} e^{-\frac{|y|^2}{2}} \int_{\Ar_{\beta_{\mu_j}(r, y)}(e_1) \cup \Ar_{\beta_{\mu_j}(r, y)}(-e_1)} 
\big(\varphi_x(r \omega, y) \cdot \omega, \varphi_y(r \omega, y)\big) \, d\mathcal{H}^{1}(\omega).
\end{align*}

We denote the components of $V_j$ by 
\begin{align*}
(V_j)_r(r, y) &= r e^{-\frac{r^2}{2}} e^{-\frac{|y|^2}{2}} \int_{\Ar_{\beta_{\mu_j}(r, y)}(e_1) \cup \Ar_{\beta_{\mu_j}(r, y)}(-e_1)} 
\varphi_x(r \omega, y) \cdot \omega \, d\mathcal{H}^{1}(\omega) \\
(V_j)_y(r, y) &= r e^{-\frac{r^2}{2}} e^{-\frac{|y|^2}{2}} \int_{\Ar_{\beta_{\mu_j}(r, y)}(e_1) \cup \Ar_{\beta_{\mu_j}(r, y)}(-e_1)} 
\varphi_y(r \omega, y) \, d\mathcal{H}^{1}(\omega).
\end{align*}

We divide the proof into several steps.

\medskip
\noindent
\textbf{Step 1:} We show that $V_j$ is Lipschitz continuous with compact support in $A$.

The fact that $V_j$ has compact support in $A$ follows from the inclusion $\text{supp} \, V_j \subset \Lambda(\supp \, \varphi)$, where we set
\begin{equation*}
\Lambda(\supp \, \varphi) 
\coloneqq\Pi(\supp\varphi)=\{(|x|,y):(x,y)\in\supp\varphi\}.
\end{equation*}

Moreover, for every $(r_1, y_1), (r_2, y_2) \in \Lambda(\supp \, \varphi)$ 
(here we set $\Ar_i = \Ar_{\beta_{\mu_j}(r_i, y_i)}(e_1) \cup \Ar_{\beta_{\mu_j}(r_i, y_i)}(-e_1)$
for $i = 1, 2$)
\begin{align*}
&|V_j(r_1, y_1) - V_j(r_2, y_2)| \\
&= \Big|r_1 e^{-\frac{r_1^2}{2}} e^{-\frac{|y_1|^2}{2}} 
\int_{\Ar_1} 
\big(\varphi_x(r_1 \omega, y_1) \cdot \omega, \varphi_y(r_1 \omega, y_1)\big) \, d\mathcal{H}^{1}(\omega) \\
&- r_2 e^{-\frac{r_2^2}{2}} e^{-\frac{|y_2|^2}{2}} 
\int_{\Ar_2} 
\big(\varphi_x(r_2 \omega, y_2) \cdot \omega, \varphi_y(r_2 \omega, y_2)\big) \, d\mathcal{H}^{1}(\omega) \Big| \\
&\leq \int_{\Ar_1} 
\Big|r_1 e^{-\frac{r_1^2}{2}} e^{-\frac{|y_1|^2}{2}} \big(\varphi_x(r_1 \omega, y_1) \cdot \omega, \varphi_y(r_1 \omega, y_1)\big) \\
&- r_2 e^{-\frac{r_2^2}{2}} e^{-\frac{|y_2|^2}{2}} \big(\varphi_x(r_2 \omega, y_2) \cdot \omega, \varphi_y(r_2 \omega, y_2)\big) 
\Big| \, d\mathcal{H}^{1}(\omega) \\
&+ r_2 e^{-\frac{r_2^2}{2}} e^{-\frac{|y_2|^2}{2}} 
\int_{\Ar_1 \Delta \Ar_2} 
|\big(\varphi_x(r_2 \omega, y_2) \cdot \omega, \varphi_y(r_2 \omega, y_2)\big)| \, d\mathcal{H}^{1}(\omega) \\
&\leq C|(r_1, y_1) - (r_2, y_2)| + r_2 e^{-\frac{r_2^2}{2}} e^{-\frac{|y_2|^2}{2}} \mathcal{H}^1(\Ar_1 \Delta \Ar_2) \\
&= C|(r_1, y_1) - (r_2, y_2)| +  r_2 e^{-\frac{r_2^2}{2}} e^{-\frac{|y_2|^2}{2}} |\xi_{\mu_j}(r_1, y_1) - \xi_{\mu_j}(r_2, y_2)| \\
&\leq C'|(r_1, y_1) - (r_2, y_2)|,
\end{align*}
for some $C, C' > 0$, where we used the fact that the function 
\begin{equation*}
(r, y) \mapsto r e^{-\frac{r^2}{2}} e^{-\frac{|y|^2}{2}} \big(\varphi_x(r \omega, y) \cdot \omega, \varphi_y(r \omega, y)\big)
\end{equation*}
is locally Lipschitz continuous.

\medskip
\noindent
\textbf{Step 2:} We show that $\beta_{\mu_j}$ is $\mathcal{H}^{n-1}$-a.e. differentiable and 
that 
\begin{align}
\partial_r(V_j)_r(r, y) &= 
e^{-\frac{|y|^2}{2}} e^{-\frac{r^2}{2}}(1 - r^2) \int_{\Ar_{\beta_{\mu_j}(r, y)}(e_1) \cup \Ar_{\beta_{\mu_j}(r, y)}(-e_1)} 
\varphi_x(r \omega, y) \cdot \omega \, d\mathcal{H}^{1}(\omega) \nonumber \\
&+ r e^{-\frac{r^2}{2}} e^{-\frac{|y|^2}{2}} \frac{\partial \beta_{\mu_j}}{\partial r}(r, y) 
\bigg(\int_{\mathbf{S}_{\beta_{\mu_j}(r, y)}(e_1) \cup \mathbf{S}_{\beta_{\mu_j}(r, y)}(-e_1)} 
\varphi_x(r \omega, y) \cdot \omega \, d\mathcal{H}^{0}(\omega) \bigg)  \label{eq:formula_V_prime} \\
&+ r e^{-\frac{r^2}{2}} e^{-\frac{|y|^2}{2}} 
\int_{\Ar_{\beta_{\mu_j}(r, y)}(e_1) \cup \Ar_{\beta_{\mu_j}(r, y)}(-e_1)} 
\Big(\left(\nabla_x \varphi_x(r \omega, y)\right) \, \omega \Big) 
\cdot \omega \, d\mathcal{H}^{1}(\omega),\nonumber
\end{align}
for $\mathcal{H}^{n-1}$-a.e. $(r, y) \in A$.

Let us set $A_j\coloneqq  \{0 < \beta_{\mu_j} < \pi/2\}$.
Since $\mu_j \in C^1_c((0, \infty) \times \R^{n-2})$, we have that $\beta_{\mu_j} \in C^1(A_j)$. Moreover, for every $(r, y) \in A_j$
\begin{align*}
\frac{\partial (V_j)_r}{\partial r}(r, y)
&= e^{-\frac{|y|^2}{2}} \frac{\partial}{\partial r} \bigg(r e^{-\frac{r^2}{2}} \int_{0}^{
\beta_{\mu_j}(r, y)}d \alpha 
\int_{\mathbf{S}_\alpha(e_1) \cup \mathbf{S}_\alpha(-e_1)} \varphi_x(r \omega, y) \cdot \omega \, \, d\mathcal{H}^{0}(\omega) \bigg) \\
&= e^{-\frac{|y|^2}{2}} e^{-\frac{r^2}{2}}(1 - r^2) \int_{\Ar_{\beta_{\mu_j}(r, y)}(e_1) \cup \Ar_{\beta_{\mu_j}(r, y)}(-e_1)} 
\varphi_x(r \omega, y) \cdot \omega \, d\mathcal{H}^{1}(\omega) \\
&+ r e^{-\frac{r^2}{2}} e^{-\frac{|y|^2}{2}} \frac{\partial \beta_{\mu_j}}{\partial r}(r, y) 
\bigg(\int_{\mathbf{S}_{\beta_{\mu_j}(r, y)}(e_1) \cup \mathbf{S}_{\beta_{\mu_j}(r, y)}(-e_1)} 
\varphi_x(r \omega, y) \cdot \omega \, d\mathcal{H}^{0}(\omega) \bigg) \\
&+ r e^{-\frac{r^2}{2}} e^{-\frac{|y|^2}{2}} \int_{0}^{
\beta_{\mu_j}(r, y)}d \alpha 
\int_{\mathbf{S}_\alpha(e_1) \cup \mathbf{S}_\alpha(-e_1)} \Big(\left(\nabla_x \varphi_x(r \omega, y)\right) \, \omega \Big) 
\cdot \omega \, d\mathcal{H}^{0}(\omega) \\
&= e^{-\frac{|y|^2}{2}} e^{-\frac{r^2}{2}}(1 - r^2) \int_{\Ar_{\beta_{\mu_j}(r, y)}(e_1) \cup \Ar_{\beta_{\mu_j}(r, y)}(-e_1)}  
\varphi_x(r \omega, y) \cdot \omega \, d\mathcal{H}^{1}(\omega) \\
&+ r e^{-\frac{r^2}{2}} e^{-\frac{|y|^2}{2}} \frac{\partial \beta_{\mu_j}}{\partial r}(r, y) 
\bigg(\int_{\mathbf{S}_{\beta_{\mu_j}(r, y)}(e_1) \cup \mathbf{S}_{\beta_{\mu_j}(r, y)}(-e_1)} 
\varphi_x(r \omega, y) \cdot \omega \, d\mathcal{H}^{0}(\omega) \bigg) \\
&+ r e^{-\frac{r^2}{2}} e^{-\frac{|y|^2}{2}} 
\int_{\Ar_{\beta_{\mu_j}(r, y)}(e_1) \cup \Ar_{\beta_{\mu_j}(r, y)}(-e_1)} 
\Big(\left(\nabla_x \varphi_x(r \omega, y)\right) \, \omega \Big) 
\cdot \omega \, d\mathcal{H}^{1}(\omega).
\end{align*} 

This shows \eqref{eq:formula_V_prime} whenever $(r, y) \in A_j$.
Note now that 
\begin{align*}
(V_j)_r(r, y) &= 0 \quad \text{ for every } (r,y) \in \text{Int}(\{\beta_{\mu_j} = 0\}), \\
(V_j)_r(r, y) &=
r e^{-\frac{r^2}{2}} e^{-\frac{|y|^2}{2}} \int_{\Ss^1} 
\varphi_x(r \omega, y) \cdot \omega \, d\mathcal{H}^{1}(\omega)
\quad \text{for every } (r,y) \in \text{Int}(\{\beta_{\mu_j} = \pi/2\}),
\end{align*}
where $\text{Int}(\cdot)$ stands for the interior of a set.
Since 
\begin{equation*}
\partial_r \beta_{\mu_j}(r, y)= 0 \text{ for every } (r, y) \in \{\beta_{\mu_j} = 0\} \cup \{\beta_{\mu_j} = \pi/2\},
\end{equation*}
using the identities above one can see that \eqref{eq:formula_V_prime} holds true 
for $\mathcal{H}^{n-1}$-a.e. $(r, y) \in (0, \infty) \times \R^{n-2}$.

\medskip
\noindent
\textbf{Step 3:} We show that  
\begin{align*}
\big(\diver_y (V_j)_y\big)(r, y) 
&= e^{-\frac{|y|^2}{2}} e^{-\frac{r^2}{2}} 
\int_{\Ar_{\beta_{\mu_j}(r, y)}(r e_1) \cup \Ar_{\beta_{\mu_j}(r, y)}(-r e_1)} 
\, \Big(\diver_y \varphi_y(x, y) - y \cdot \varphi_y(x, y)\Big) \, d\mathcal{H}^{1}(x) \\
&+ r e^{-\frac{r^2}{2}} e^{-\frac{|y|^2}{2}}  
\bigg(\int_{\mathbf{S}_{\beta_{\mu_j}(r, y)}(e_1) \cup \mathbf{S}_{\beta_{\mu_j}(r, y)}(-e_1)} 
\nabla_y \beta_{\mu_j}(r, y) \cdot \varphi_y(r \omega, y) \, d\mathcal{H}^{0}(\omega) \bigg),
\end{align*}
for $\mathcal{H}^{n-1}$-a.e. $(r, y) \in A$.

The proof can be done by adapting the arguments used in the previous step. 

\medskip
\noindent
\textbf{Step 4:} We show that
\begin{align*} 
&\int_{\Phi(A\times \Ss^{1})} \chi_{F_{\mu_j}}(x, y)
e^{-\frac{|x|^2}{2}} e^{-\frac{|y|^2}{2}} \, \big(\nabla_x \varphi_x(x, y)\hat{x}\big) \cdot \hat{x} \, 
dx \, dy \\
&+ \int_{\Phi(A\times \Ss^{1})} \chi_{F_{\mu_j}}(x, y)
e^{-\frac{|x|^2}{2}} e^{-\frac{|y|^2}{2}} \, \varphi_x(x, y) \cdot \left(\frac{\hat{x}}{|x|} - x\right) dx \, dy \\
&+ \int_{\Phi(A\times \Ss^{1})} \chi_{F_{\mu_j}}(x, y)
e^{-\frac{|x|^2}{2}} e^{-\frac{|y|^2}{2}} \, \Big(\diver_y \varphi_y(x, y) - y \cdot \varphi_y(x, y)\Big) dx \, dy \\
&= -\int_A r e^{-\frac{r^2}{2}} e^{-\frac{|y|^2}{2}}  
\bigg(\int_{\mathbf{S}_{\beta_{\mu_j}(r, y)}(e_1) \cup \mathbf{S}_{\beta_{\mu_j}(r, y)}(-e_1)} 
\nabla_{(r, y)} \beta_{\mu_j}(r, y) \cdot \big(\varphi_x(r \omega, y) \cdot \omega, \varphi_y(r \omega, y)\big) \, d\mathcal{H}^{0}(\omega) \bigg) \, dr \, dy.
\end{align*}
 
Integrating \eqref{eq:formula_V_prime}, thanks to the classical divergence theorem applied in $A$, 
and recalling that $V_j$ has compact support in $A$, we obtain 
\begin{align*}
0 &= \int_A \diver_{(r, y)} V_j(r, y) \, dr \, dy \\
&= \int_A \left(\partial_r(V_j)_r(r, y) + \diver_{y}(V_j)_y(r, y)\right) \, dr \, dy  \\
&= \int_A \left(e^{-\frac{|y|^2}{2}} e^{-\frac{r^2}{2}}(1 - r^2) 
\int_{\Ar_{\beta_{\mu_j}(r, y)}(e_1) \cup \Ar_{\beta_{\mu_j}(r, y)}(-e_1)}  
\varphi_x(r \omega, y) \cdot \omega \, d\mathcal{H}^{1}(\omega) \right) \, dr \, dy \\
&+ \int_A r e^{-\frac{r^2}{2}} e^{-\frac{|y|^2}{2}}  
\bigg(\int_{\mathbf{S}_{\beta_{\mu_j}(r, y)}(e_1) \cup \mathbf{S}_{\beta_{\mu_j}(r, y)}(-e_1)} 
\nabla_{(r, y)} \beta_{\mu_j}(r, y) \cdot \big(\varphi_x(r \omega, y) \cdot \omega, \varphi_y(r \omega, y)\big) \, d\mathcal{H}^{0}(\omega) \bigg) \, dr \, dy \\
&+ \int_A r e^{-\frac{r^2}{2}} e^{-\frac{|y|^2}{2}} \bigg(
\int_{\Ar_{\beta_{\mu_j}(r, y)}(e_1) \cup \Ar_{\beta_{\mu_j}(r, y)}(-e_1)} 
\left(\nabla_x \varphi_x(r \omega, y) \, \omega\right) 
\cdot \omega \, d\mathcal{H}^{1}(\omega) \bigg) \, dr \, dy \\
&+ \int_A e^{-\frac{|y|^2}{2}} e^{-\frac{r^2}{2}} 
\int_{\Ar_{\beta_{\mu_j}(r, y)}(r e_1) \cup \Ar_{\beta_{\mu_j}(r, y)}(-r e_1)} 
\, \Big(\diver_y \varphi_y(x, y) - y \cdot \varphi_y(x, y)\Big) \, d\mathcal{H}^{1}(x) \, dr \, dy \\
&= \int_{\Phi(A\times \Ss^{1})} \chi_{F_{\mu_j}}(x, y)
e^{-\frac{|x|^2}{2}} e^{-\frac{|y|^2}{2}} \, \varphi_x(x, y) \cdot \left(\frac{\hat{x}}{|x|} - x\right) dx \, dy \\
&+ \int_A r e^{-\frac{r^2}{2}} e^{-\frac{|y|^2}{2}}  
\bigg(\int_{\mathbf{S}_{\beta_{\mu_j}(r, y)}(e_1) \cup \mathbf{S}_{\beta_{\mu_j}(r, y)}(-e_1)} 
\nabla_{(r, y)} \beta_{\mu_j}(r, y) \cdot \big(\varphi_x(r \omega, y) \cdot \omega, \varphi_y(r \omega, y)\big) \, d\mathcal{H}^{0}(\omega) \bigg) \, dr \, dy \\
&+ \int_{\Phi(A\times \Ss^{1})} \chi_{F_{\mu_j}}(x, y)
e^{-\frac{|x|^2}{2}} e^{-\frac{|y|^2}{2}} \, \big(\nabla_x \varphi_x(x, y)\hat{x}\big) \cdot \hat{x} \, 
dx \, dy \\
&+ \int_{\Phi(A\times \Ss^{1})} \chi_{F_{\mu_j}}(x, y)
e^{-\frac{|x|^2}{2}} e^{-\frac{|y|^2}{2}} \, \Big(\diver_y \varphi_y(x, y) - y \cdot \varphi_y(x, y)\Big) dx \, dy, 
\end{align*}
which gives the claim.

\medskip
\noindent
\textbf{Step 5:} We prove that
\begin{align}
 &\int_{\Phi(A\times \Ss^{1})} \chi_{F_{\mu_j}}(x, y)((\diver \varphi)(x, y) 
-  \varphi(x, y) \cdot (x, y)) e^{-\frac{|x|^2}{2}} e^{-\frac{|y|^2}{2}} \, dx \, dy \nonumber \\
&\leq \left|r e^{-\frac{r^2}{2}} e^{-\frac{|y|^2}{2}} D\xi_{\mu_j}\right|(\Lambda(\supp \, \varphi)) \label{eq:qwedfg33}  \\
&+ \int_A e^{-\frac{r^2}{2}} e^{-\frac{|y|^2}{2}} 
\mathcal{H}^{0}(\mathbf{S}_{\beta_{\mu_j}(r, y)}(r e_1) \cup \mathbf{S}_{\beta_{\mu_j}(r, y)}(-r e_1)) \, dr \, dy, \nonumber
\end{align}
where $\Lambda(\supp \, \varphi) \subset (0, \infty) \times \R^{n-2}$ is the compact set defined in Step 1.

Thanks to \eqref{eq:equality} and Step~4
\begin{equation}\label{eq:2_4} 
\begin{split} 
& \int_{\Phi(A\times \Ss^{1})} \chi_{F_{\mu_j}}(x, y)((\diver \varphi)(x, y) 
-  \varphi(x, y) \cdot (x, y)) e^{-\frac{|x|^2}{2}} e^{-\frac{|y|^2}{2}} \, dx \, dy \\
&=  \int_{\Phi(A\times \Ss^{1})} \chi_{F_{\mu_j}}(x, y)
e^{-\frac{|x|^2}{2}} e^{-\frac{|y|^2}{2}} \, \diver_{x \parallel} \varphi_{x \parallel}(x, y) 
dx \, dy \\
&- \int_A r e^{-\frac{r^2}{2}} e^{-\frac{|y|^2}{2}}
\bigg(\int_{\mathbf{S}_{\beta_{\mu_j}(r, y)}(e_1) \cup \mathbf{S}_{\beta_{\mu_j}(r, y)}(-e_1)} 
\nabla_{(r, y)} \beta_{\mu_j}(r, y) \cdot \big(\varphi_x(r \omega, y) \cdot \omega, \varphi_y(r \omega, y)\big) \, d\mathcal{H}^{0}(\omega) \bigg) \, dr \, dy. 
\end{split}
\end{equation}

We now estimate the right hand side of the expression above.
Thanks to \eqref{def:beta} and arguing as in Step 2 we have that 
\begin{equation*}
\nabla_{(r, y)} \xi_{\mu_j}(r, y) 
= 4 \nabla_{(r, y)} \beta_{\mu_j}(r, y)
= \nabla_{(r, y)} \beta_{\mu_j}(r, y) \mathcal{H}^{0}(\mathbf{S}_{\beta_{\mu_j}(r,y)}(e_1)
\cup \mathbf{S}_{\beta_{\mu_j}(r,y)}(-e_1)),
\end{equation*}
for $\mathcal{H}^{n-1}$-a.e. $(r, y) \in (0, \infty) \times \R^{n-2}$.

Therefore, 
\begin{align}
-&\int_A r e^{-\frac{r^2}{2}} e^{-\frac{|y|^2}{2}} 
\bigg(\int_{\mathbf{S}_{\beta_{\mu_j}(r, y)}(e_1) \cup \mathbf{S}_{\beta_{\mu_j}(r, y)}(-e_1)} 
\nabla_{(r, y)} \beta_{\mu_j}(r, y) \cdot \big(\varphi_x(r \omega, y) \cdot \omega, \varphi_y(r \omega, y)\big) \, d\mathcal{H}^{0}(\omega) \bigg) \, dr \, dy \nonumber \\
&\leq  \int_{\Lambda(\text{supp} \, \varphi)} r e^{-\frac{r^2}{2}} e^{-\frac{|y|^2}{2}} 
|\nabla_{(r, y)} \beta_{\mu_j}(r, y)|
\mathcal{H}^0\left(\mathbf{S}_{\beta_{\mu_j}(r, y)}(e_1) \cup \mathbf{S}_{\beta_{\mu_j}(r, y)}(-e_1)\right)
\, dr \, dy  \nonumber \\
&=  \int_{\Lambda(\text{supp} \, \varphi)} r e^{-\frac{r^2}{2}} e^{-\frac{|y|^2}{2}} 
|\nabla_{(r, y)} \xi_{\mu_j}(r, y)| \, dr \, dy  \nonumber \\
&=  \left|r e^{-\frac{r^2}{2}} e^{-\frac{|y|^2}{2}} D\xi_{\mu_j}\right|(\Lambda(\supp \, \varphi)). 
\label{eq:first_piece}
\end{align}

Let us now focus on the term in \eqref{eq:2_4}.
Applying the divergence theorem in the circle
and denoting by $\nu_*(x, y)$ the exterior unit conormal in $\partial B(r)$ to
$\Ar_{\beta_{\mu_j}(r, y)}(r e_1) \cup \Ar_{\beta_{\mu_j}(r, y)}(-r e_1)$ at its endpoints, we have 

\begin{align}
&\int_{\Phi(A\times \Ss^{1})} \chi_{F_{\mu_j}}(x, y)
e^{-\frac{|x|^2}{2}} e^{-\frac{|y|^2}{2}} \, \diver_{x \parallel} \varphi_{x \parallel}(x, y) 
dx \, dy \nonumber \\
&=  \int_A e^{-\frac{r^2}{2}} e^{-\frac{|y|^2}{2}} 
\left[\int_{\Ar_{\beta_{\mu_j}(r, y)}(r e_1) \cup \Ar_{\beta_{\mu_j}(r, y)}(-r e_1)} 
\diver_{x \parallel} \varphi_{x \parallel}(x, y) \, d\mathcal{H}^1(x) \right] \, dr \, dy \nonumber \\
&=  \int_A e^{-\frac{r^2}{2}} e^{-\frac{|y|^2}{2}} 
\left[\int_{\mathbf{S}_{\beta_{\mu_j}(r, y)}(r e_1) \cup \mathbf{S}_{\beta_{\mu_j}(r, y)}(-r e_1)} 
\varphi_{x \parallel}(x, y) \cdot \nu_{*}(x, y) \, d\mathcal{H}^0(x) \right] \, dr \, dy \nonumber \\
&\leq  \int_A e^{-\frac{r^2}{2}} e^{-\frac{|y|^2}{2}} 
\mathcal{H}^{0}(\mathbf{S}_{\beta_{\mu_j}(r, y)}(r e_1) \cup \mathbf{S}_{\beta_{\mu_j}(r, y)}(-r e_1)) \, dr \, dy. \label{eq:second_piece}
\end{align}

Combining \eqref{eq:2_4}, \eqref{eq:first_piece}, and \eqref{eq:second_piece}, we obtain \eqref{eq:qwedfg33}.

\medskip
\noindent
\textbf{Step 6:} We show that $F_\mu$ is a set of finite Gaussian perimeter.

Note that $\chi_{F_{\mu_j}} \to \chi_{F_{\mu}}$ $\mathcal{H}^n$-a.e. in $\R^n$, and 
$\beta_{\mu_j} \to \beta_{\mu}$ $\mathcal{H}^{n-1}$-a.e. in $(0, \infty) \times \R^{n-2}$. 
Note also that, from our choice of the sequence $\{\xi_{\mu_j}\}_{j \in \mathbb{N}}$, and since the function 
\[ (r, y) \mapsto r e^{-\frac{r^2}{2}} e^{-\frac{|y|^2}{2}}
\]
is continuous and bounded, we have 
\begin{equation*}
\left|r e^{-\frac{r^2}{2}} e^{-\frac{|y|^2}{2}} D\xi_{\mu_j}\right|
\overset{*}{\rightharpoonup} \left|r e^{-\frac{r^2}{2}} e^{-\frac{|y|^2}{2}} D\xi_{\mu}\right| \quad \text{as } j \to \infty.
\end{equation*}

Therefore, taking the limsup as $j \to \infty$ in \eqref{eq:qwedfg33}, and using the fact that 
$\Lambda(\supp \, \varphi)$ is compact,
\begin{align*}
& \int_{\Phi(A\times \Ss^{1})} \chi_{F_{\mu}} (x, y)
\left( (\diver \varphi) (x, y) -  \varphi (x, y) \cdot (x, y) \right) e^{- \frac{|x|^2}{2}} e^{- \frac{|y|^2}{2}} 
dx \, dy  \\
&= \limsup_{j \to \infty}
\int_{\Phi(A\times \Ss^{1})} \chi_{F_{\mu_j}} (x, y)
\left( (\diver \varphi) (x, y) -  \varphi (x, y) \cdot (x, y) \right) e^{- \frac{|x|^2}{2}} e^{- \frac{|y|^2}{2}} 
dx \, dy  \\
&\leq \limsup_{j \to \infty}
 \left| r e^{- \frac{r^2}{2}} e^{- \frac{|y|^2}{2}} D \xi_{\mu_j} \right|(\Lambda(\supp \, \varphi))   
 + 4 
 \int_A e^{- \frac{r^2}{2}} e^{- \frac{|y|^2}{2}}   \, dr \, dy \\
&\leq   \left| r e^{- \frac{r^2}{2}} e^{- \frac{|y|^2}{2}} D \xi_{\mu} \right|(\Lambda(\supp \, \varphi)) 
+ 2 (2 \pi)^{\frac{n-1}{2}} \\
&\leq  \left| r e^{- \frac{r^2}{2}} e^{- \frac{|y|^2}{2}} D \xi_{\mu} \right|(A)  
+  2 (2 \pi)^{\frac{n-1}{2}}.
\end{align*}
Taking the supremum of the above inequality 
over all functions $ \varphi \in C^{1}_{c} (\Phi(A\times \Ss^{1});\R^n)$ with 
$\| \varphi \|_{L^{\infty} (\Phi(A\times \Ss^{1});\R^n)} \leq 1$, we obtain 
$$
P_{\gamma} (F_\mu;\Phi(A\times \Ss^{1}))\leq \frac{1}{(2 \pi)^{\frac{n-1}{2}}} \left| r e^{- \frac{r^2}{2}} e^{- \frac{|y|^2}{2}} D \xi_{\mu} \right|(A)  
 + 2.
$$
Thanks to \eqref{eq:bound_for_D_xi}, using the fact that $E$ is a set of finite Gaussian perimeter, we have
\begin{equation} \label{bound on Pfmu}
P_{\gamma} (F_\mu;\Phi(A\times \Ss^{1})) 
\leq  P_{\gamma}(E) + 2 < \infty.
\end{equation}
Since $A$ was arbitrary, this shows that $F_\mu$ is a set of finite Gaussian perimeter in $\R^2_0 \times \R^{n-2}$. 
 
Let us show that, in fact, $F_\mu$ is a set of finite Gaussian perimeter in $\R^{n}$.  
Suppose that $A' \Subset [0, \infty) \times \R^{n-2}$, and let $\psi \in C^{1}_{c}(\Phi(A'\times \Ss^{1});\R^n)$ with 
$\|\psi\|_{L^{\infty}(\Phi(A'\times \Ss^{1});\R^n)} \leq 1$. 
Let now $\varepsilon > 0$, and let $h_{\varepsilon} : [ 0, \infty) \to [0, 1]$ be a $C^1$ function such that $h_{\varepsilon} (s) = 0$ for $0\leq s\leq \varepsilon$ and $h_{\varepsilon} (s) =1$ for $s \geq 2 \varepsilon$, with $h'_\varepsilon \leq \frac{C}{\varepsilon}$ for some $C > 0$. Setting $\psi_{\varepsilon}\coloneqq h_{\varepsilon} \psi$, we have
\begin{align*}
& \int_{\Phi(A'\times \Ss^{1})} h_\varepsilon (|x|) \chi_{F_{\mu}} (x, y)
\left( \diver \psi (x, y) -  \psi (x, y) \cdot (x, y) \right) e^{- \frac{|x|^2}{2}} e^{- \frac{|y|^2}{2}} 
dx \, dy  \\
&= \int_{\Phi(A'\times \Ss^{1})} \chi_{F_{\mu}} (x, y)
\left( \diver \psi_\varepsilon (x, y) -  \psi_\varepsilon (x, y) \cdot (x, y) \right) e^{- \frac{|x|^2}{2}} e^{- \frac{|y|^2}{2}} 
dx \, dy  
+ I_\varepsilon \\
&\leq P_{\gamma}(E) + 2 
+ I_\varepsilon,  
\end{align*}
where we used \eqref{bound on Pfmu}, and we set
\[
I_\varepsilon
=  \int_{\Phi(A'\times \Ss^{1})}  \chi_{F_{\mu}} (x, y)
\left( h'_\varepsilon (|x|) \frac{x}{|x|} \cdot \psi_x (x, y)  \right) e^{- \frac{|x|^2}{2}} e^{- \frac{|y|^2}{2}} 
dx \, dy.
\]
Note now that 
\[
|I_\varepsilon |
\leq C_{A'} \varepsilon, 
\]
for some $C_{A'} >0$ depending on $A$.
Therefore, passing first to the limit as $\varepsilon \to 0$, and then taking the supremum over $\psi$ we obtain that 
\[
P_{\gamma} (F_\mu;\Phi(A'\times \Ss^{1})) 
\leq  P_{\gamma}(E) + 2 < \infty.
\]
This shows that $F_\mu$ is a set of finite Gaussian perimeter in $\R^n$.

\vspace{.2cm}

\noindent
\textbf{Step 7:} We conclude. Let $A \Subset (0,\infty) \times \R^{n-2}$ be open,
and let $\varphi \in C^{1}_{c} (\Phi(A\times \Ss^{1});\R^n)$ with 
$\| \varphi \|_{L^{\infty} (\Phi(A\times \Ss^{1});\R^n)} \leq 1$.
Combining \eqref{eq:2_4} and \eqref{eq:first_piece}, 
we have that for every $j \in \mathbb{N}$
\begin{align*}
&\int_{\Phi(A\times \Ss^{1})} \chi_{F_{\mu_j}} (x, y)
e^{- \frac{|x|^2}{2}} e^{- \frac{|y|^2}{2}}  \, \, \Big( \diver \varphi (x, y) -  (x, y) \cdot \varphi (x, y)  \Big) 
dx \, dy  \\
&\leq  \left| r e^{- \frac{r^2}{2}} e^{- \frac{|y|^2}{2}} D \xi_{\mu_j} \right|(\Lambda(\supp \, \varphi)) \\
&+  \int_{\Phi(A\times \Ss^{1})} \chi_{F_{\mu_j}} (x, y)
e^{- \frac{|x|^2}{2}} e^{- \frac{|y|^2}{2}}  \, \diver_{x \parallel} \varphi_{x \parallel} (x, y) 
dx \, dy.
\end{align*}
Taking the limsup as $j \to \infty$, since $\Lambda(\supp \, \varphi)$ is compact, and thanks to Corollary~\ref{cor:parallel_total_variation}, 
we have
\begin{align*}
&\int_{\Phi(A\times \Ss^{1})} \chi_{F_{\mu}} (x, y)
e^{- \frac{|x|^2}{2}} e^{- \frac{|y|^2}{2}}  \, \, \Big( \diver \varphi (x, y) -  (x, y) \cdot \varphi (x, y)  \Big) 
dx \, dy  \\
&\leq  \left| r e^{- \frac{r^2}{2}} e^{- \frac{|y|^2}{2}} D \xi_{\mu} \right|(\Lambda(\supp \, \varphi)) \\
&+  \int_{\Phi(A\times \Ss^{1})} \chi_{F_{\mu}} (x, y)
e^{- \frac{|x|^2}{2}} e^{- \frac{|y|^2}{2}}  \, \diver_{x \parallel} \varphi_{x \parallel} (x, y) 
dx \, dy \\
&\leq  \left| r e^{- \frac{r^2}{2}} e^{- \frac{|y|^2}{2}} D \xi_{\mu} \right|(\Lambda(\supp \, \varphi)) 
+  | D_{x \parallel} (e^{- \frac{|x|^2}{2}} e^{- \frac{|y|^2}{2}}  \chi_{F_\mu}) | (\Phi(A\times \Ss^{1})) \\
&=  \left| r e^{- \frac{r^2}{2}} e^{- \frac{|y|^2}{2}} D \xi_{\mu} \right|(\Lambda(\supp \, \varphi)) 
+  |  e^{- \frac{|x|^2}{2}} e^{- \frac{|y|^2}{2}} D_{x \parallel} \chi_{F_\mu} | (\Phi(A\times \Ss^{1})),
\end{align*}
where we also used the fact that $D_{x \parallel} (e^{- \frac{|x|^2}{2}} e^{- \frac{|y|^2}{2}}  \chi_{F_\mu}) 
= e^{- \frac{|x|^2}{2}} e^{- \frac{|y|^2}{2}}  D_{x \parallel}  \chi_{F_\mu}$.
Taking the supremum over all $\varphi \in C^{1}_{c} (\Phi(A\times \Ss^{1});\R^n)$ with 
$\| \varphi \|_{L^{\infty} (\Phi(A\times \Ss^{1});\R^n)} \leq 1$, 
\[
(2 \pi)^{\frac{n-1}{2}} P_\gamma(F_\mu;\Phi(A\times \Ss^{1}))
\leq  \left| r e^{- \frac{r^2}{2}} e^{- \frac{|y|^2}{2}} D \xi_{\mu} \right| (A) 
+ \left| e^{- \frac{|x|^2}{2}} e^{- \frac{|y|^2}{2}} D_{x \parallel} \chi_{F_\mu} \right|
 (\Phi(A\times \Ss^{1})),
\]
which shows \eqref{eq:smhts} when $B \Subset (0, \infty) \times \R^{n-2}$ is open.
The general case with $B$ a Borel set follows by approximation. 
 \end{proof}

%%%%%%%%%%%%%% Section 3
\section{Perimeter inequality}\label{sect perim ineq}

The following easy result is the key in the proof of Theorem~\ref{theorem ineq}.

\begin{lemma}[Isoperimetric inequality for $2$-symmetric sets in the circle] \label{lem:isop 2 symm}
Let $r > 0$ be fixed, let $0 < \alpha < \pi/2$, and let $D \subset \partial B (r)$ be a $2$-symmetric set of finite perimeter in $\partial B (r)$ with $\mathcal{H}^1 (D) = 4 \alpha r$. Then,
\[
4 = \mathcal{H}^0 \left( \partial^* \Big( \Ar_{\alpha} (r e_1) \cup  \Ar_{\alpha} (- r e_1) \Big) \right)
\leq \mathcal{H}^0 (\partial^* D).
\]
Moreover, equality holds if and only if
\[
D =_{\mathcal{H}^1} \Ar_{\alpha} (r p) \cup  \Ar_{\alpha} (- r p),
\]
for some $p \in \{(1,0), (0,1)\}$.
\end{lemma}

\begin{proof}
Since $0 < \alpha < \pi/2$, $D$ cannot be empty, nor the full circle. By $2$-symmetry, $D$ must have at least one portion in each quadrant.
In each quadrant, there is at least one point of the boundary (indeed, two, unless $D$ crosses the axes). 
A minimizer is a union of two arcs centered on an axis (crossing the same coordinate axis at two points), yielding exactly $4$ boundary points.
 \end{proof}

\begin{proof}[Proof of Theorem~\ref{theorem ineq}]
We adapt the arguments of the proof of \cite[Theorem 1.1]{BCF}.
We can assume that $B \Subset (0,\infty) \times \R^{n-2}$, since 
the general case $B \subset (0,\infty) \times \R^{n-2}$ can be obtained by approximation. 
Let $G_E$ and $G_{F_\mu}$ be the sets given by Proposition~\ref{volpert theorem sets}.

We start by proving \eqref{per ineq}. 
We first prove the inequality when $B \subset ((0,\infty) \times \R^{n-2}) \setminus (G_{F_\mu} \cap G_E)$, and then 
in the case $B \subset G_{F_\mu} \cap G_E$. 
The case of a general Borel set $B \Subset (0,\infty) \times \R^{n-2}$  
then follows by decomposing $B$ as $B = (B \setminus (G_{F_\mu} \cap G_E)) \cup (B \cap G_{F_\mu} \cap G_E)$.

\vspace{.2cm}

\noindent
\textbf{Step 1:} We show \eqref{per ineq} when $B \subset ((0,\infty) \times \R^{n-2}) \setminus (G_{F_\mu} \cap G_E)$.

\vspace{.2cm}

The set $(0,\infty) \times \R^{n-2} \setminus (G_{F_\mu} \cap G_E)$ can be decomposed as
\begin{equation} \label{decomposition of bad set}
\{ \beta_\mu = 0 \} \cup \{ \beta_\mu = \pi/2 \} \cup 
\underbrace{\left( \{ 0 < \beta_\mu < \pi/2 \} \setminus (G_{F_\mu} \cap G_E) \right)}_{=: \mathcal{N}},
\end{equation}
where $\mathcal{H}^{n-1}(\mathcal{N}) = 0$ by the definition of $G_{F_\mu}$ and $G_E$ 
in Proposition~\ref{volpert theorem sets}.
We treat each piece separately, establishing the following auxiliary results first.

\vspace{.2cm}

\noindent
\textbf{Step 1a:} We show that for every $\mu$-distributed set $W \subset \R^{n}$ of finite Gaussian perimeter we have 
\[
\mathcal{H}^{n-1} \left(  \partial^* W \cap \Phi (\mathcal{N} \times \Ss^{1}) \cap \{ \nu^W_{x \parallel} \neq 0\}  \right) = 0
\]
for every Borel set $\mathcal{N}\Subset (0,\infty) \times \R^{n-2}$ such that $\mathcal{H}^{n-1} (\mathcal{N}) = 0$.

Indeed, by Proposition~\ref{prop:coarea}, 
\begin{align*}
&\mathcal{H}^{n-1} \left(  \partial^* W \cap \Phi (\mathcal{N} \times \Ss^{1}) \cap \{ \nu^W_{x \parallel} \neq 0\} \right) \\
&= \int_{\mathcal{N}}  
\int_{( \partial^{*} W)_{(r,y)}}
  \frac{\chi_{\{ \nu^W_{x \parallel}  \neq 0 \}} (x, y)}{ |  \nu^W_{x \parallel} | }   \, d \mathcal{H}^0 (x)  \, dr \, dy  = 0,
\end{align*}
where in the last equality we used the fact that $\mathcal{H}^{n-1} (\mathcal{N}) = 0$.

\vspace{.2cm}

\noindent
\textbf{Step 1b:} We show that for every $\mu$-distributed set $W \subset \R^{n}$ of finite Gaussian perimeter  we have 
\[
\mathcal{H}^{n-1} \left(  \partial^* W \cap \Phi (\{ \beta_\mu = \pi/2 \} \times \Ss^{1}) \cap \{ \nu^W_{x \parallel} \neq 0\}  \right) = 0.
\]
Indeed, by Proposition~\ref{prop:coarea}, 
\begin{align*}
&\mathcal{H}^{n-1} \left(  \partial^* W \cap \Phi (\{ \beta_\mu = \pi/2 \} \times \Ss^{1}) \cap \{ \nu^W_{x \parallel} \neq 0\}  \right) \\
&= \int_{\{ \beta_\mu = \pi/2 \}}  
\int_{( \partial^{*} W)_{(r,y)}}
  \frac{\chi_{\{ \nu^W_{x \parallel} \neq 0 \}} (x, y)}{ | \nu^W_{x \parallel}  (x, y)|}  \, d \mathcal{H}^0 (x)  \, dr \, dy \\
 &= \int_{\{ \beta_\mu = \pi/2 \}}  
\int_{ \partial^* ( W_{(r,y)} )}
  \frac{\chi_{\{ \nu^W_{x \parallel} \neq 0 \}} (x, y)}{ | \nu^W_{x \parallel}  (x, y)|}  \, d \mathcal{H}^0 (x)  \, dr \, dy  = 0,
 \end{align*}
where in the last equality we used the fact that, 
thanks to Vol'pert Theorem, for $\mathcal{H}^{n-1}$-a.e. $(r, y) \in \{ \beta_{\mu} = \pi/2 \}$ we have 
$( \partial^{*} W)_{(r,y)} = \partial^* ( W_{(r,y)} ) = \partial^* (\partial B (r)) = \emptyset$.
Similarly,
\begin{equation*}
\mathcal{H}^{n-1} \left(  \partial^* W \cap \Phi (\{ \beta_\mu = 0 \} \times \Ss^{1}) \cap \{ \nu^W_{x \parallel} \neq 0\}  \right) = 0.
\end{equation*}

\vspace{.2cm}

\noindent
\textbf{Step 1c:} Using the decomposition \eqref{decomposition of bad set}, we can write 
$B = (B \cap \{ \beta_\mu = 0 \}) \cup (B \cap \{ \beta_\mu = \pi/2 \}) \cup (B \cap \mathcal{N})$
and treat each piece separately. Therefore, by applying \textbf{Step 1a} and \textbf{Step 1b} to $W = F_\mu$, we have
\[
P_\gamma(F_\mu; \Phi(B \times \Ss^1) \cap \{ \nu^{F_\mu}_{x \parallel} \neq 0 \}) = 0.
\]
In particular, by \eqref{eq:de_giorgi_structure},
\begin{equation*}
\left| e^{- \frac{|x|^2}{2}} e^{- \frac{|y|^2}{2}} D_{x \parallel} \chi_{F_\mu} \right|(\Phi(B \times \Ss^1))
= \int_{\partial^* F_\mu \cap \Phi(B \times \Ss^1)} 
e^{- \frac{|x|^2}{2}} e^{- \frac{|y|^2}{2}} |\nu^{F_\mu}_{x \parallel}(x)| \, d\mathcal{H}^{n-1}(x)=0.
\end{equation*}
Then, by Proposition~\ref{prop:bound P F_mu},
\begin{align*}
(2\pi)^{\frac{n-1}{2}} P_\gamma(F_\mu; \Phi(B \times \Ss^1)) 
&\leq \left| r e^{- \frac{r^2}{2}} e^{- \frac{|y|^2}{2}}  D\xi_\mu \right|(B)
+ \left| e^{- \frac{|x|^2}{2}} e^{- \frac{|y|^2}{2}}  D_{x \parallel} \chi_{F_\mu} \right|
(\Phi(B \times \Ss^1)) \\
&= \left| r e^{- \frac{r^2}{2}} e^{- \frac{|y|^2}{2}}  D\xi_\mu \right|(B)
\leq (2\pi)^{\frac{n-1}{2}}P_\gamma(E; \Phi(B \times \Ss^1)),
\end{align*}
where in the last inequality we used \eqref{eq:bound_for_D_xi}.
Notice moreover that for the set $F_\mu$ and a Borel set 
$B \subset ((0,\infty) \times \R^{n-2}) \setminus G_{F_\mu}$, inequality \eqref{eq:bound_for_D_xi} holds as an equality.

%%%%%%%%%%%%%%%%%%%%%%%%%%%%%%%%%%%%

\vspace{.2cm}

\noindent
\textbf{Step 2:} We prove inequality \eqref{per ineq} when $B \subset G_{F_\mu} \cap G_E$. 
In this step, we set
\[
p_{E} (r, y) = \mathcal{H}^0 (\partial^* (E_{(r, y)})), 
\]
and we use a similar notation for the set $F_\mu$. 

First of all note that, thanks to Proposition~\ref{volpert theorem sets}, 
we have $\nu^{F_\mu}_{x \parallel} (x, y) \neq 0$ whenever
$(x, y) \in \partial^* F_\mu$ and $(|x|, y) \in G_{F_{\mu}}$.
Therefore, 
\begin{align*}
&P_\gamma(F_{\mu};\Phi(B\times \Ss^{1})) 
= P_\gamma(F_{\mu};\Phi(B\times \Ss^{1})\cap \{\nu^{F_\mu}_{x \parallel} \neq 0  \}) \\
&= \int_{\partial^* F_\mu \cap \Phi(B\times \Ss^{1})\cap \{\nu^{F_\mu}_{x \parallel}\neq0  \}}d\mathcal{H}^{n-1}_{\gamma}(x, y) \\
&= \int_{B}  \frac{e^{- \frac{r^2 + |y|^2}{2}}}{(2 \pi)^{(n-1)/2}}
\int_{(\partial^*F_\mu)_{(r, y)}} \frac{1}{|\nu^{F_\mu}_{x \parallel}(x, y)|}d\mathcal{H}^{0}(x) \, dr \, dy. 
\end{align*}
Now, using the fact that 
$$
1 = | \nu^{F_\mu}_{x \perp}(x, y)|^2 + | \nu^{F_\mu}_{x \parallel} (x, y) |^2 + | \nu_{y}^{F_{\mu}}(x, y) |^2
= ( \hat{x} \cdot \nu^{F_{\mu}}_x  (x, y) )^2 + | \nu^{F_\mu}_{x \parallel}(x, y) |^2 + | \nu^{F_\mu}_{y}(x, y) |^2,
$$
thanks to Remark~\ref{rem:constant}, for $\mathcal{H}^{n-1}$-a.e. $(r, y) \in B$ we have 
\begin{equation}\begin{split}\label{ineq isop 2 dim}
&\int_{(\partial^*F_\mu)_{(r, y)}} \frac{1}{|\nu^{F_\mu}_{x \parallel}(x, y)|}d\mathcal{H}^{0}(x) \\
&=   \int_{(\partial^*F_\mu)_{(r, y)}}
\sqrt{1+ \frac{(\hat{x} \cdot \nu^{F_{\mu}}_x (x, y))^2}{| \nu^{F_\mu}_{x \parallel}(x, y)|^2} 
+ \frac{| \nu_{y}^{F_{\mu}} (x, y) |^2} {| \nu^{F_\mu}_{x \parallel}(x, y)|^2} } d\mathcal{H}^{0}(x) \\
&= p_{F_\mu} (r, y) \, 
\sqrt{1+  \frac{|r \nabla_{(r, y)} \xi_\mu (r, y)|^2}{p^2_{F_\mu} (r, y)} }   
=   \sqrt{p^2_{F_\mu} (r, y) + |r \nabla_{(r, y)} \xi_\mu (r, y)|^2}  \\
&\leq  
\sqrt{p^2_{E} (r, y)+|r \nabla_{(r, y)} \xi_\mu (r, y)|^2}  
= p_{E} (r, y) \, 
\sqrt{1+  \frac{|r \nabla_{(r, y)} \xi_\mu (r, y)|^2}{p^2_{E} (r, y)} }    \\
&= p_{E} (r, y) \, \sqrt{1 + \left(  \fint_{(\partial^*E)_{(r, y)}}  \frac{\hat{x} \cdot \nu^{E}_x(x, y)}{| \nu^E_{x \parallel}(x, y)|} \, d\mathcal{H}^{0}(x) \right)^2 
+  \left| \fint_{(\partial^*E)_{(r, y)}} \frac{ \nu_{y}^{E} (x, y) } {| \nu^E_{x \parallel}(x, y)|} \, d\mathcal{H}^{0}(x)  \right|^2},    
\end{split}\end{equation}
thanks to Lemma~\ref{lem:isop 2 symm} and Lemma~\ref{lem:nabla_xi}.

Observe now that the function $f: \R \times \R^{n-2} \to [0, \infty)$ given by
$$ 
f(t, z)\coloneqq  \sqrt{1 + t^{2} + |z|^2},
$$
is strictly convex.
Therefore, for $\mathcal{H}^{n-1}$-a.e. $(r, y) \in B$ we can apply Jensen's inequality to obtain
\begin{equation} \label{jensen}
\begin{split}
&\int_{(\partial^*F_\mu)_{(r, y)}} \frac{1}{|\nu^{F_\mu}_{x \parallel}(x, y)|}d\mathcal{H}^{0}(x)  \\
&\leq \int_{(\partial^*E)_{(r, y)}}
\sqrt{1+ \frac{(\hat{x} \cdot \nu^{E}_x (x, y))^2}{| \nu^E_{x \parallel}(x, y)|^2} 
+ \frac{| \nu_{y}^{E} (x, y) |^2} {| \nu^E_{x \parallel}(x, y)|^2} } d\mathcal{H}^{0}(x)  \\
&= \int_{(\partial^*E)_{(r, y)}} \frac{1}{|\nu^E_{x \parallel}(x, y)|}d\mathcal{H}^{0}(x). 
\end{split}
\end{equation}
From this it follows that
\begin{align*}
&(2 \pi)^{(n-1)/2} P_\gamma(F_{\mu};\Phi(B\times \Ss^{1})) 
= \int_{B}  e^{- \frac{r^2 + |y|^2}{2}}
\int_{(\partial^*F_\mu)_{(r, y)}} \frac{1}{|\nu^{F_\mu}_{x \parallel}(x, y)|}d\mathcal{H}^{0}(x) \, dr \, dy \\
&\leq \int_{B} e^{- \frac{r^2 + |y|^2}{2}} 
\int_{(\partial^*E)_{(r, y)}} \frac{1}{|\nu^E_{x \parallel}(x, y)|}d\mathcal{H}^{0}(x) \, dr \, dy 
= (2 \pi)^{(n-1)/2} P_\gamma(E;\Phi(B\times \Ss^{1})).
\end{align*}

\vspace{.2cm}

\noindent
\textbf{Step 3:} We conclude the proof of the theorem.
Suppose equality holds in \eqref{per ineq}. Then, equality holds in \eqref{ineq isop 2 dim} and in \eqref{jensen} for $\mathcal{H}^{n-1}$-a.e. $(r, y) \in B$.
From the equality in \eqref{ineq isop 2 dim}, thanks to Lemma~\ref{lem:isop 2 symm}, we have (i).
Instead, the equality in \eqref{jensen} gives (ii).
\end{proof}

%%%%%%%%% Section 4
\section{Regularity}\label{sect reg}

In this section we prove Theorem~\ref{theorem reg}. We will proceed in two steps, where we first prove 
the regularity up to a small singular set (Proposition \ref{prop:partial-reg}). The singular set is small enough for us to use the 
second variation and to deduce that the boundary  $\partial E \cap \{x_i >0\}$ for every $i = 1, \dots, n$ is a graph of a BV-function. 
We then apply an argument similar to the one used in \cite{BDGM},  to deduce that the singular set is empty outside the origin.

The strategy for the first step is to remove the symmetry constraint at small
scales and reduce the problem to the standard regularity theory for almost-minimizers of the
perimeter. More precisely, starting from penalized minimality with respect to $n$-symmetric
competitors, Lemma~\ref{lem:quasi_control} yields an unconstrained almost-minimality
property by replicating a local perturbation along the orbit of the reflection group $\mathcal{G}$.
This argument is close in spirit to standard symmetry-reduction principles for invariant
variational problems; for related ideas see 
\cite{Morgan,HPRR}. Since we did not find the precise local estimate needed here in the
literature, we include a self-contained proof.

Once the symmetry constraint has been removed, the
classical theory of almost-minimizers yields the $C^{1,\alpha}$ regularity of $\partial^* E$
and the estimate $\dim_{\mathcal H}(\Sigma_E)\le n-8$, where $\dim_{\mathcal H}$ denotes the Hausdorff dimension. We  then use symmetric criticality (see \cite{Palais}) 
to derive the Euler--Lagrange equation away from the coordinate hyperplanes, and finally extend
the analyticity across the hyperplanes by a reflection argument.

\smallskip
The first lemma deals with the symmetry of the problem.  
\begin{lemma}\label{lem:quasi_control}
Let $E\subset\R^n$ be an $n$-symmetric set of finite Gaussian perimeter.
Let $p \in \partial E$. Define the index set $I \coloneqq \{ i \in \{1, \ldots, n\} : p \in \Pi_i \}$,
where $\Pi_i \coloneqq \{x \in \R^n : x_i = 0\}$ are the coordinate hyperplanes. Let $\mathcal{G}$ be the reflection group 
of order $2^n$ generated by the reflections $\{A_1, \ldots, A_n\}$ across the coordinate hyperplanes.
Let $r> 0$ be small enough such that  $B_r(p) \cap \Pi_j = \emptyset$ for all $j \notin I$.

Assume that there exists \(\Lambda>0\) such that, for every \(n\)-symmetric set \(W\) satisfying $W \triangle E \Subset \bigcup_{g\in \mathcal{G}} g(B_r(p))$ the following inequality holds: 
\begin{equation} \label{eq:quasi_control sym}
P_{\gamma}(E;B_{r}(p))\le P_{\gamma}(W;B_{r}(p))+\Lambda\gamma(W\triangle E).
\end{equation}
Then, for every perturbation $F$ of $E$ in $B_r(p)$, i.e., another set such that $F \triangle E \Subset B_r(p)$, one has
\begin{equation} \label{eq:quasi_control}
P_\gamma(E;B_r(p)) \leq P_\gamma(F;B_r(p)) + 2^n\Lambda \, \gamma(F \triangle E).
\end{equation}
\end{lemma}
\begin{proof}
Let $B=B_r(p)$ and let $F$ be any set of finite Gaussian  perimeter with $F \triangle E \Subset B$. 
We have to construct a globally $n$-symmetric variant $W$ of $F$, keeping a certain control inside $B$. 

If $I = \emptyset$, i.e., if $p$ does not belong to any coordinate hyperplane, $B$ is disjoint from its reflections.
The modification can be easily done by substituting $F$ with its symmetral inside the reflections of $B$.

We assume henceforth that $I \neq \emptyset$. We follow an orbit averaging strategy.
We decompose $\mathcal{G}$ based on the local geometry at $p$:
\begin{enumerate}[label=(\roman*)]
\item \emph{The Stabilizer:} $\mathcal{G}_{\mathrm{stab}} \coloneqq \{g \in \mathcal{G} : g(p) = p\}$. 
This subgroup is generated by $\{A_i\}_{i \in I}$ and has order $|\mathcal{G}_{\mathrm{stab}}| = 2^{|I|}$.
\item \emph{The Orbit Representatives:} Let $\{g_1, \ldots, g_N\}$ be a set of representatives 
for the cosets $\mathcal{G}/\mathcal{G}_{\mathrm{stab}}$, where $N = 2^{n - |I|}$.
\end{enumerate}
Note that the images $\{g_m(B)\}_{m=1}^{N}$ are pairwise disjoint, 
as they lie in distinct orthants not separated by any $\Pi_j$ with $j \in I$.

\medskip
\noindent\textbf{Step 1: Local Symmetrization.}
The hyperplanes $\{\Pi_i\}_{i \in I}$ partition the ball $B$ into $|\mathcal{G}_{\mathrm{stab}}|$ 
disjoint sectors. Let $\mathcal{S}$ denote this collection of sectors. 
For each sector $S \in \mathcal{S}$, we construct  a locally symmetric candidate $F_S$ on $B$ 
by reflecting the restriction of $F$ to $S$:
\[
F_S \coloneqq \bigcup_{h \in \mathcal{G}_{\mathrm{stab}}} h(F \cap S).
\]
By construction, $F_S$ is symmetric with respect to the hyperplanes $\{\Pi_i\}_{i \in I}$.
Using the $\mathcal{G}_{\mathrm{stab}}$-invariance of the Gaussian weight and the sub-additivity of the
perimeter, we  obtain:
\begin{equation} \label{eq:local_sym_ineq}
\sum_{S \in \mathcal{S}} P_\gamma(F_S; B) \leq |\mathcal{G}_{\mathrm{stab}}| \, P_\gamma(F; B).
\end{equation}

For each sector $S$, we define a set $W_S$ by replacing $E$ with $F_S$ inside $B$, 
and replicating this modification to all other balls in the orbit:
\[
W_S \coloneqq \left( E \setminus \bigcup_{m=1}^{N} g_m(B) \right) \cup \bigcup_{m=1}^{N} g_m(F_S).
\]
Since $E$ is globally $n$-symmetric and $F_S$ is symmetric under $\mathcal{G}_{\mathrm{stab}}$, the set $W_S$ 
is invariant under the full group $\mathcal{G}$, i.e., it is an $n$-symmetric set.

\smallskip
\textbf{Step 2: Perimeter control.} 
By the hypothesis \eqref{eq:quasi_control sym}, applied to the competitor $W_S$, we have
\begin{equation}\label{eq:local_S_est}
P_{\gamma}(E;B)\le P_{\gamma}(W_{S};B)+\Lambda\gamma(W_{S}\triangle E).
\end{equation}
Note that $P_{\gamma}(W_{S};B) = P_{\gamma}(F_{S};B)$. We now  explicitly compute the symmetric difference term. Since the modification is replicated $N$ times on pairwise disjoint balls and $E$ is $\mathcal{G}$-invariant,
$$
\gamma(W_{S}\triangle E) = 
\sum_{m=1}^N
\gamma\bigl(g_m(F_S)\triangle(E\cap g_m(B))\bigr) = N \gamma(F_S \triangle (E \cap B)).
$$
Inside $B$, $F_S$ consists of $|\mathcal{G}_{\mathrm{stab}}|$ copies of $F \cap S$ reflected by elements of $\mathcal{G}_{\mathrm{stab}}$. Since $E \cap B$ is also symmetric with respect to $\mathcal{G}_{\mathrm{stab}}$, the symmetric difference on each reflected sector is identical to the difference on the generating sector $S$. Thus:
$$
\gamma(F_S \triangle (E \cap B)) = |\mathcal{G}_{\mathrm{stab}}| \gamma((F \triangle E) \cap S).
$$
Inserting this into \eqref{eq:local_S_est}, we obtain the local estimate for the sector $S$:
$$
P_{\gamma}(E;B)\le P_{\gamma}(F_{S};B) + N\Lambda|\mathcal{G}_{\mathrm{stab}}|\gamma((F \triangle E) \cap S).
$$
Summing this inequality over all sectors $S\in\mathcal{S}$:
\begin{align*}
|\mathcal{G}_{\mathrm{stab}}| \, P_\gamma(E; B) 
&\leq \sum_{S \in \mathcal{S}} P_\gamma(F_S; B) + N\Lambda \sum_{S \in \mathcal{S}} |\mathcal{G}_{\mathrm{stab}}| \, \gamma\big((F \triangle E) \cap S\big) \\
&\leq |\mathcal{G}_{\mathrm{stab}}| \, P_\gamma(F; B) + |\mathcal{G}_{\mathrm{stab}}| \, N\Lambda \, \gamma(F \triangle E),
\end{align*}
where we used \eqref{eq:local_sym_ineq} and the fact that 
$\sum_{S \in \mathcal{S}} \gamma((F \triangle E) \cap S) = \gamma(F \triangle E)$.
Dividing by $|\mathcal{G}_{\mathrm{stab}}|$ we get \eqref{eq:quasi_control}.
\end{proof}

\medskip

We may proceed to prove the partial regularity. 
\begin{proposition}
\label{prop:partial-reg}
Let $E$ be a minimizer of the Gaussian  isoperimetric problem  among $n$-symmetric sets 
with prescribed Gaussian measure. Then the reduced boundary $\partial^* E$ is a relatively open, real analytic hypersurface, and the singular set  
$\Sigma_E = \partial  E \setminus \partial^* E$  satisfies $\dim_{\mathcal H}(\Sigma_E) \leq n-8$. 
\end{proposition}

\begin{proof}
We divide the proof into four main steps, and  we will use the notation introduced in Lemma~\ref{lem:quasi_control} (and in its proof) for the reflection group $\mathcal{G}$. 

\smallskip
\noindent \textbf{Step 1: $C^{1,\alpha}$ regularity and dimension of the singular set.}
Let $p \in \partial E$ and let $I \coloneqq  \{i \in \{1, \dots, n\} : p \in \Pi_i\}$. Set $r_p > 0$ as the distance between $p$ and $\bigcup_{j \notin I} \Pi_j$.

Since $E$ is a minimizer of the Gaussian perimeter among $n$-symmetric sets subject to a volume constraint, we first observe that $E$ satisfies 
a penalized minimality condition within the class of symmetric sets. By performing symmetric volume-fixing variations supported in 
$\bigcup_{g \in \mathcal{G}} g(B_r(p))$ (adapting standard arguments, see e.g. \cite[Part III]{Maggi}), we deduce that there exists a constant 
$\Lambda > 0$ such that for every $r < r_p$ and for every $n$-symmetric competitor $W$ with $W \triangle E \Subset \bigcup_{g \in \mathcal{G}} g(B_r(p))$,
\[
P_\gamma(E; B_r(p)) \le P_\gamma(W; B_r(p)) + \Lambda\gamma(W \triangle E).
\]
Applying Lemma~\ref{lem:quasi_control}, we remove the symmetry constraint: for any set $F$ (not necessarily symmetric) such that $F \Delta E \Subset B_r(p)$,
\begin{equation} \label{eq:penalized_local}
P_\gamma(E; B_r(p)) \le P_\gamma(F; B_r(p)) + 2^n \Lambda \gamma(F \triangle E).
\end{equation}
Following \cite[Proposition 2]{BBJ}, from \eqref{eq:penalized_local} we get that $E$ is an almost-minimizer of the \textit{Euclidean} perimeter:
there exists $\tilde{\Lambda} > 0$, depending on $|p|$, such that
\[
P(E; B_r(p)) \le P(F; B_r(p)) + \tilde{\Lambda} r^n.
\]
We can now apply the classical regularity theory for almost-minimizers (see \cite[Theorem 21.8 and Theorem 28.1]{Maggi}).
Let us define, for every $p 
\in \R^n$, 
\[
\overline{u}_E (p) \coloneqq  \limsup_{r \to 0} \frac{|E \cap B_r (p)|}{|B_r (p)|}, \quad \text{ and }
\underline{u}_E (p) \coloneqq  \liminf_{r \to 0} \frac{|E \cap B_r (p)|}{|B_r (p)|},
\]
and let us use the notation $u_E (p)$ if the liminf and the limsup above coincide.
We assume that $E$ coincides with the precise representative 
\begin{equation}\label{eq: precise representative}
\{p\in\R^n \colon u_E(p)
= 1\}.
\end{equation}
Then, thanks to the density estimate on $\partial E$ for almost-minimizers (see \cite[Theorem~21.11]{Maggi}), we have that the precise representative $E=\{u_E  =1\}$ and its essential complement $\{ u_E =0\}$ are both open, and $\partial E \subset \R^n \setminus (\{u_E=1\} \cup \{u_E=0\})$. 
The \emph{topological} boundary $\partial E$ coincides with the \emph{essential} boundary $\{p\in\R^n \colon u_E(p)\notin\{0,1\}\}$ 
and it is partitioned in two parts:
\begin{enumerate}[label=(\roman*)]
\item the \emph{reduced} boundary $\partial^* E$, that is a subset of $\{p\in\R^n \colon u_E(p)=1/2\}$, is  a $C^{1,\alpha}$ hypersurface  for every $\alpha<1/2$, and is relatively open in $\partial E$.
\item the singular set $\Sigma_E=\partial E \setminus \partial^* E$;  it satisfies $\dim_{\mathcal{H}}(\Sigma_E) \leq n - 8$.
\end{enumerate}

\smallskip
\noindent \textbf{Step 2: Lagrange multiplier for symmetric variations.}
Let $X\in C_c^\infty(\R^n;\R^n)$ be a vector field. Let $0< \delta \ll 1$, and let  $\Phi : \R^n \times(-\delta,\delta) \to \R^n$ be the flow associated with $X$, i.e.,   
\begin{equation*}
\frac{\partial}{\partial t} \Phi(x,t) = X( \Phi(x,t) ), \qquad \Phi(x,0) = x.
\end{equation*}
We perturb $E$ through the flow $\Phi$ by defining $E_t \coloneqq  \Phi(E,t)$ for $t\in(-\delta,\delta)$.  
The first variations of the perimeter and the volume are given by
\begin{equation*}\begin{split}
\frac{d}{dt}P_\gamma(E_t)\Big|_{t=0}&=\int_{\partial^* E} (\Hh_E - x \cdot \nu^E)\, (X \cdot \nu^E) \, d\mathcal{H}^{n-1}_\gamma\\
\frac{d}{dt}\gamma(E_t)\Big|_{t=0}&= \frac{1}{\sqrt{2\pi}}\int_{\partial^* E} (X \cdot \nu^E) \, d\mathcal{H}^{n-1}_\gamma.
\end{split}\end{equation*}

If $\Phi$ preserves $n$-symmetry, since $E$ minimizes the Gaussian perimeter among $n$-symmetric sets of finite perimeter 
subject to the volume constraint $\gamma(E) = m$, by the method of Lagrange multipliers there exists $\lambda \in \R$ such that
\begin{equation} \label{reg:lagrange symmetric}
\frac{d}{dt}P_\gamma(E_t)\Big|_{t=0} = \lambda\sqrt{2\pi}\, \frac{d}{dt}\gamma(E_t)\Big|_{t=0}.
\end{equation}

\smallskip
\noindent \textbf{Step 3: Extension to general variations via symmetric criticality.}
Recall that $|\mathcal{G}| = 2^n$. 
Both the Gaussian perimeter and the Gaussian volume are $\mathcal{G}$-invariant: for every $g \in \mathcal{G}$ and every set $F$ of finite perimeter,
\begin{equation*}
P_\gamma(g \cdot F) = P_\gamma(F), \qquad \gamma(g \cdot F) = \gamma(F).
\end{equation*}
Since $E$ and the Gaussian weight $e^{-|x|^2/2}$ are $\mathcal{G}$-invariant, a change of variables gives that the first
variation is invariant with respect to $\mathcal{G}$.

Now let $X \in C^\infty_c(\R^n; \R^n)$ be an arbitrary vector field supported away 
from the coordinate hyperplanes. Define the symmetrized field
\begin{equation*}
\hat{X}(x) \coloneqq \frac{1}{|\mathcal{G}|} \sum_{g \in \mathcal{G}} g^{-1} \cdot X(g \cdot x).
\end{equation*}
By construction, $\hat{X}$ is $\mathcal{G}$-equivariant and generates a flow preserving $n$-symmetry.
Let $\hat{E}_t$ be the perturbation of $E$ through the flow associated to $\hat{X}$.

Using the fact that $X$ is supported away from the coordinate hyperplanes  and the $\mathcal{G}$-invariance, 
\begin{equation*}
\frac{d}{dt}P_\gamma(\hat{E}_t)\Big|_{t=0}=\frac{d}{dt}P_\gamma(E_t)\Big|_{t=0},
\end{equation*}
and similarly for the volume term.
Since $\hat{X}$ preserves $n$-symmetry, equation \eqref{reg:lagrange symmetric} holds for $\hat{X}$ 
and therefore also for $X$. Since this is true for any $X$ supported away from the coordinate hyperplanes,
we conclude that pointwise 
\begin{equation} \label{eq:CGMC} 
\Hh_E - \nu^E \cdot x  = \lambda \qquad \text{on } \partial^* E \setminus \bigcup_{i=1}^n \Pi_i,
\end{equation}
where $\Hh_E$ denotes the mean curvature. 

\smallskip
\noindent \textbf{Step 4: Upgrade to analyticity.}
Away from the coordinate hyperplanes, since $\partial^* E$ satisfies the Euler--Lagrange equation \eqref{eq:CGMC}, 
standard elliptic regularity (see \cite[Chapter 6]{Morrey}) upgrades $C^{1,\alpha}$ regularity to real analyticity, i.e.,  
$\partial^* E \setminus \bigcup_{i=1}^n \Pi_i$ is a real analytic hypersurface.

We now  extend the analyticity across the coordinate hyperplanes by a standard
multi-fold reflection argument. Let $p \in \partial^* E \cap \bigcup_{i=1}^n \Pi_i$
and let $I = \{i : p \in \Pi_i\}$, $k \coloneqq |I|$. By $n$-symmetry, $\nu^E(p) \cdot e_i = 0$
for every $i \in I$, so $k \leq n-1$ and, after relabeling, rotating and translating  within
$\text{span}(e_{k+1},\ldots,e_n)$, we may write $\partial^* E$ near $p$ as the graph
$z_n = u(z')$ with $u \in C^{1,\alpha}(B'_\rho(0))$ even in each of $z_1,\ldots,z_k$, where $\rho$ is sufficiently small and $B'_\rho(0)$ denotes the open ball of $\R^{n-1}$ centered in $0$ with radius $\rho$.  
The even extension $u(z') \coloneqq u(|z_1|,\ldots,|z_k|,z_{k+1},\ldots,z_{n-1})$
belongs to $C^{1,\alpha}(B'_\rho(0))$, because evenness forces homogeneous Neumann
conditions $\partial_{z_i}u|_{z_i=0}=0$ for $i \leq k$. Therefore we deduce from \eqref{eq:CGMC}  that the function $u$ is a  $C^{1,\alpha}$-regular weak solution of the equation
\[
\diver\left(\frac{\nabla u}{\sqrt{1 + |\nabla u|^2}}\right) 
- \frac{(z', u) \cdot (\nabla u, -1)}{\sqrt{1 + |\nabla u|^2}} = \lambda
\]
in $B'_\rho(0)$ in the sense of integration by parts against a smooth test function. By standard elliptic regularity 
results  (see e.g. \cite[Theorem~6.7.6]{Morrey}) it is analytic. 
\end{proof}

\begin{remark} \label{E is open let's face it}
From now on, we will always identify a minimizer $E$ with the open set $\{u_E = 1 \}$. Therefore, we have that $\{u_E = 0 \} = \R^n \setminus \overline{E}$ is also open and 
\[
\partial E = \partial^* E \cup \Sigma_E,
\]
where $\partial^* E$ is (analytic and) relatively open in $\partial E$, while $\Sigma_E$ is relatively closed in $\partial E$, and thus also closed in $\R^n$. 
\end{remark}

\medskip

We introduce some notation for calculus on smooth hypersurfaces. 
Let $M$ be a $C^\infty$ hypersurface. We may extend a given  vector field $X \in C^\infty(M; \R^n)$ to $\R^n$ and denote by $D_M X$ the tangential differential of $X$ on $M$ and for a real valued function  $u \in C^\infty(M)$ we denote the 
tangential gradient as $\nabla_M u$.
The tangential divergence of $X$ on $M$ is defined by $\diver_M X=\tr(D_M X)$ and 
the Laplace--Beltrami operator of $u$ on $M$ is
\[
\Delta_M u \coloneqq  \diver_M(D_M u). 
\]
When $M$ is the boundary of a set $E$, the exterior unit normal $\nu^E$ fixed above is a smooth vector 
field in $C^\infty(M; \R^n)$. Then the
second fundamental form is  $B_E = D_{\partial E} \nu^E$ and the  mean curvature $\Hh_E$   is 
\[
\Hh_E = \diver_{\partial E}(\nu^E).
\]

Finally, the divergence theorem on hypersurfaces  in the Gaussian setting states that for every Lipschitz continuous 
vector field $X:\partial E \to \R^n$ with compact support, using the notation  $X_{\partial E} = X - (\nu^E \otimes \nu^E) X$,  it holds
\[
\int_{\partial E} (\diver_{\partial E} X- X_{\partial E} \cdot x)\, d \mathcal{H}^{n-1}_\gamma(x) 
= \int_{\partial E} \Hh_E X\cdot\nu^E \, d \mathcal{H}^{n-1}_\gamma(x).
\]

The drift Laplacian (or the surface Ornstein--Uhlenbeck operator) on $M$ associated with the Gaussian measure is defined as
\begin{equation*} 
 \Delta_M\varphi - \nabla_M \varphi \cdot x , \quad \varphi\in C^2(M).
\end{equation*} 
Integration by parts with respect to the Gaussian measure $\mathcal{H}^{n-1}_\gamma$ yields
\begin{equation}\label{eq:L_OU}
\int_M |D_M \varphi|^2 \,d\mathcal{H}^{n-1}_\gamma 
= - \int_M \varphi (  \Delta_M\varphi - \nabla_M \varphi \cdot x) \,d\mathcal{H}^{n-1}_\gamma, \quad \varphi\in C^2_c(M).
\end{equation}

\medskip

The following is the standard stability inequality for the Gaussian isoperimetric 
problem (see \cite[Proposition~3]{BBJ}), adapted to the  $n$-symmetric class via the principle of symmetric criticality.

\begin{proposition}\label{prop:second_var_geom}
Let $E\subset\R^n$ be a minimizer of the Gaussian 
isoperimetric problem  among $n$-symmetric sets with prescribed Gaussian measure. 
Set $M=\partial^*E$.
Define the quadratic form $Q:C^\infty_c(M)\to\R$ by
\begin{equation*}
Q(\varphi)\coloneqq\int_M\!\Big(|\nabla_M\varphi|^2
\;-\;\big(|B_E|^2+1\big)\varphi^2\Big)\,d\mathcal{H}^{n-1}_\gamma.
\end{equation*}
Then
\begin{equation}\label{eq:stability}
Q(\varphi)\ge 0
\end{equation}
for every  $n$-symmetric $\varphi\in C^\infty_c(M)$ 
satisfying $\int_M\varphi\,d\mathcal{H}^{n-1}_\gamma=0$. 
\end{proposition}

The following result is a direct adaptation of
\cite[Lemma~1]{BBJ}. Indeed, the cutoff construction used there relies
only on the finiteness of \(P_\gamma(E)\), the local perimeter-growth
estimate, and
\[
\mathcal H^{n-3}(\partial E\setminus\partial^*E)=0,
\]
all of which hold in the present setting. Averaging over the finite reflection group \(\mathcal{G}\), 
in order to make the cutoffs \(n\)-symmetric, yields the following lemma.

\begin{lemma}\label{lem:gaussian_cutoff}
Let \(E\) be as in Proposition~\ref{prop:partial-reg}, and set
\(M=\partial^*E\). There exists a sequence of \(n\)-symmetric functions
\[
\zeta_h\in C_c^\infty(M),\qquad 0\leq\zeta_h\leq1,
\]
such that
\[
\int_M\left((1-\zeta_h)^2+
|\nabla_M\zeta_h|^2\right)\,
d\mathcal H^{n-1}_\gamma\longrightarrow0.
\]
\end{lemma}

\begin{remark}\label{rem:less regularity}
Let $H^1_\gamma(M)$ be the Sobolev space of functions which are square-integrable with respect to the measure 
$\mathcal{H}^{n-1}_\gamma$ in $M$, together with their weak first derivatives.  
Lemma~\ref{lem:gaussian_cutoff}, combined with truncation and local smooth approximation, implies that
\[
H_\gamma^1(M)
=
H_{\gamma,0}^1(M)
\coloneqq 
\overline{C_c^\infty(M)}^{\,H_\gamma^1(M)}.
\]
Moreover, Proposition \ref{prop:second_var_geom} can be extended to all $n$-symmetric $H^1_\gamma(M)$ functions with zero average.
\end{remark}

\medskip
We may use the second variation condition to deduce some information about the  global structure of the minimizers. 
The first result deals with the  connected components of $\partial^*E$.

\begin{proposition}\label{prop:components_structure}
Let $E \subset \R^n$ be a minimizer of the Gaussian 
isoperimetric problem  among $n$-symmetric sets with prescribed Gaussian measure. 
Then one of the following holds:
\begin{enumerate}[label=(\roman*)]
    \item $\partial^*E$ is connected.
    \item $\partial^*E$ consists of exactly two parallel hyperplanes, one being the reflection of the other one  across a coordinate hyperplane (in which case $E$ is a symmetric slab or its complement).
\end{enumerate}
\end{proposition}

\begin{proof}
Let $\mathcal{G}$ be the reflection group of order $2^n$ generated by the reflections $\{A_1, \ldots, A_n\}$ across the coordinate hyperplanes.
We first show that $\mathcal{G}$ acts transitively on the set of connected components $\mathscr{C}$ of $\partial^* E$.

Assume, by contradiction, that \(\mathcal{G}\) does not act transitively on the
set \(\mathscr C\) of connected components of \(M=\partial^*E\).
Then \(\mathscr C\) can be partitioned into two nonempty
\(\mathcal{G}\)-invariant subcollections, whose unions we denote by
\(\Gamma_A\) and \(\Gamma_B\). Define the test function
\[
\varphi\coloneqq \mathcal H^{n-1}_\gamma(\Gamma_B)\chi_{\Gamma_A}-\mathcal H^{n-1}_\gamma(\Gamma_A)\chi_{\Gamma_B}.
\]
We have $\varphi\in H^1_\gamma(M)$ and $\nabla_M\varphi=0$ a.e. on $M$.
Moreover, \(\varphi\) is \(n\)-symmetric and $\int_M\varphi\,d\mathcal H^{n-1}_\gamma=0$.
Thus Remark~\ref{rem:less regularity} allows us to use \(\varphi\) in
the stability inequality. We obtain
\[
\int_M (|B_E|^2+1)\varphi^2\,
d\mathcal H^{n-1}_\gamma
\leq
\int_M|\nabla_M\varphi|^2\,
d\mathcal H^{n-1}_\gamma
=0.
\]
On the other hand,
\[
\int_M (|B_E|^2+1)\varphi^2\,
d\mathcal H^{n-1}_\gamma
\geq
\int_M\varphi^2\,d\mathcal H^{n-1}_\gamma
>0,
\]
which is a contradiction. Hence \(\mathcal{G}\) acts transitively on \(\mathscr C\).

Let $N$ be the number of connected components. By transitivity, $N = 2^k$ for some $k \in \{0, \dots, n\}$. If $N=1$, we are in case (i).

Assume $N \ge 2$. The stabilizer $\mathcal{G}_\Gamma$ of any component $\Gamma \in \mathscr{C}$ is a proper subgroup of $\mathcal{G}$. Thus, there exists an index $j \in \{1, \dots, n\}$ such that the reflection $A_j \notin \mathcal{G}_\Gamma$. We can assume without loss of generality that $j=1$. This implies $A_1$ moves every component to a disjoint component, so no component can intersect its fixed-point space $\{x_1 = 0\}$. Consequently, $\partial^*E \cap \{x_1 = 0\} = \emptyset$. 
By Remark~\ref{E is open let's face it}, $E$ is open. Therefore,
recalling that $\partial E = \partial^* E \cup \Sigma_E$, we have
\[
\{x_1=0\} \setminus \Sigma_E \subset E \cup (\R^n \setminus \overline{E}).
\]
Since $\dim_{\mathcal{H}}(\Sigma_E) \le n-8$, the punctured hyperplane $\{x_1=0\} \setminus \Sigma_E$ is connected. 
Because $E$ and $\R^n \setminus \overline{E}$ are open in $\R^n$, their intersections with $\{x_1=0\}$ form a partition of a connected topological space by two disjoint, relatively open sets. Thus, 
\[
\text{either} \quad \{x_1=0\} \setminus \Sigma_E \subset E \quad \text{or} \quad \{x_1=0\} \setminus \Sigma_E \subset \R^n \setminus \overline{E}.
\]

Assume $\{x_1=0\} \setminus \Sigma_E \subset \R^n \setminus \overline{E}$ (the other case follows from similar arguments). Then $S \coloneqq E \cap \{x_1 > 0\}$ is a set of finite perimeter with no reduced boundary on $\{x_1=0\}$. By the $n$-symmetry of $E$ and the fact that $\partial^* E \cap \{x_1=0\} = \emptyset$, both the Gaussian measure and perimeter split evenly: $\gamma(S) = \frac{1}{2}\gamma(E)$ and $P_\gamma(S) = \frac{1}{2}P_\gamma(E)$.

Let $I(m)$ be the Gaussian perimeter of a half-space of Gaussian mass $m$.
By the standard Gaussian isoperimetric inequality, $P_\gamma(S) \ge I(\gamma(S))$, with equality if and only if $S$ is a half-space. 

We claim this must hold with equality. Suppose, for contradiction, that the inequality is strict: $P_\gamma(S) > I(\gamma(S))$.
Consider the half-space $H_a \coloneqq \{x \in \R^n \colon x_1 > a\}$, with $a>0$ chosen so that $\gamma(H_a) = \gamma(S)$. 
Define the competitor set $W \coloneqq H_a \cup A_1(H_a)$, which is the complement of a symmetric slab.
Since $a > 0$,  the half-spaces $H_a$ and $A_1(H_a)$ are essentially disjoint. Therefore, $W$ preserves the mass of $E$:
\[
\gamma(W) = 2\gamma(H_a) = 2\gamma(S) = \gamma(E).
\]
Since $W$ is $n$-symmetric, it is an admissible competitor. 
This contradicts the assumption that $E$ is a global minimizer of the Gaussian perimeter among $n$-symmetric sets, since
\[
P_\gamma(W) = 2P_\gamma(H_a) = 2I(\gamma(S)) < 2P_\gamma(S) = P_\gamma(E).
\]

Therefore, $P_\gamma(S) = I(\gamma(S))$, and $S$ must coincide  with $H_a$. By symmetry, $E$ coincides with $W$, 
and its reduced boundary consists of exactly two hyperplanes $\{x_1 = \pm a\}$, corresponding to case (ii).
\end{proof}

\medskip

The second result on the global structure of the minimizers is a dichotomy related to the components of 
the normal $\nu^E$.
For the rest of this section $x=(x_1,\ldots,x_n)$ is a generic point in $\R^n$ (we pause the notation $(x,y)$
with $x\in\R^2$ and $y\in\R^{n-2}$).

\begin{lemma}\label{lem:graph}
Let $E\subset\R^n$ be a minimizer of the Gaussian isoperimetric problem among 
$n$-symmetric sets with prescribed Gaussian measure. 
Given $i\in\{1,\dots,n\}$, set $\nu_i\coloneqq \nu^E\cdot e_i$, and $M^+_i \coloneqq \partial^*E\cap\{x\in\R^n \colon x_i >0\}$. 
Then   one of the following  holds:
\begin{enumerate}[label=(\roman*)]
\item $\nu_i\equiv 0$ on $\partial^*E$. In this case $E$ is a cylinder in the direction $e_i$:
there exists a Borel set $E'\subset\R^{n-1}$
such that $E=\{x\in\R^n \colon (x_1,\ldots,x_{i-1},x_{i+1},\ldots,x_n)\in E'\}$.
\item $\nu_i$ does not change sign on $M^+_i$ (either $\nu_i (x)> 0$ for every $x \in M^+_i$, or $\nu_i (x) < 0$ 
for every $x \in M^+_i$).  
\end{enumerate}
\end{lemma}

\begin{proof}
Without loss of generality, let us assume $i=n$. Set $M\coloneqq \partial^*E$ and $M^+ \coloneqq M^+_n$.

\smallskip
\noindent\textbf{Step 1: Slab case.}
By Proposition~\ref{prop:components_structure}, either $M$ is
connected or $E$ is a symmetric slab (or its complement). In the slab
case, $E=\{|x_j|<a\}$ for some $j\in\{1,\dots,n\}$ and some $a>0$. 
If $j\neq n$ one has the alternative (i), while if $j=n$ one has the alternative (ii). 
Henceforth for the remainder of the proof we can assume that $M$ is connected.

\smallskip
\noindent\textbf{Step 2: $M^+$ is connected.}
By Proposition~\ref{prop:partial-reg},  $M$ is real analytic and, by symmetry, it   meets
each coordinate hyperplane $\Pi_k$ orthogonally; hence
$M\cap\{x\in\R^n \colon x_n\geq 0\}$ is a topological manifold with corners.
The continuous folding map
$x\mapsto(x_1,\dots,x_{n-1},|x_n|)$ from $M$ onto
$M\cap\{x_n\geq 0\}$ is surjective, so the latter is
connected. 
The interior of a connected manifold
with corners is connected: every point has a corner chart whose interior is connected and non-
empty, so it is adjacent to exactly one component of the interior, and the closures of two distinct
components would disconnect the manifold. Hence $M^+= M \cap \{x_n > 0\}$ is connected.

\smallskip
\noindent\textbf{Step 3: \(\nu_n\) does not change sign on \(M^+\).}
Suppose, for contradiction, that both sets
\[
\{\nu_n>0\}\cap M^+,
\qquad
\{\nu_n<0\}\cap M^+
\]
are nonempty. Define on \(M^+\)
\[
u_+\coloneqq (\nu_n)^+,
\qquad
u_-\coloneqq (\nu_n)^-=\max\{-\nu_n,0\}.
\]
Since \(E\) is invariant under reflection across \(\{x_n=0\}\), and
\(M\) meets this hyperplane orthogonally, we have
\[
\nu_n=0
\qquad\text{on }M\cap\{x_n=0\}.
\]
We may therefore extend \(u_+\) and \(u_-\) evenly across
\(\{x_n=0\}\), retaining the same notation for the extended functions.
They are locally Lipschitz, \(n\)-symmetric, nonzero, and satisfy
\[
0\leq u_\pm\leq1,
\qquad
u_+u_-=0.
\]

Let
\[
e_n^T\coloneqq e_n-\nu_n\nu^E,
\qquad
x^T\coloneqq x-(x\cdot\nu^E)\nu^E
\]
denote the tangential projections of \(e_n\) and \(x\), respectively.
Differentiating the Euler--Lagrange equation in the direction \(e_n^T\)
gives
\[
\nabla_M\Hh_E\cdot e_n
=
D_{e_n^T}(x\cdot\nu^E)
=
x\cdot D_{e_n^T}\nu^E,
\]
because \(e_n^T\cdot\nu^E=0\). By the symmetry of the second
fundamental form,
\[
x\cdot D_{e_n^T}\nu^E
=
B_E(e_n^T,x^T)
=
B_E(x^T,e_n^T)
=
D_{x^T}\nu^E\cdot e_n
=
x\cdot\nabla_M\nu_n.
\]
Combining the well-known identity $\Delta_M\nu^E+|B_E|^2\nu^E=\nabla_M\Hh_E$
(see, e.g., \cite[formula~(10.17)]{Giusti}) with the previous ones, yields
\begin{equation}\label{eq:calculation-nu-n}
\Delta_M\nu_n-x\cdot\nabla_M\nu_n+|B_E|^2\nu_n=0 \,\text{ on }\, M.
\end{equation}

Let \(\zeta_h\) be given by Lemma~\ref{lem:gaussian_cutoff}.
Since \(u_\pm\in H^1_{\mathrm{loc}}(M^+)\) and have zero trace on
\(M\cap\{x_n=0\}\), the functions
\[
\zeta_h^2u_+,\qquad -\zeta_h^2u_-
\]
belong to \(H_0^1(M^+)\). Therefore, by density of
\(C_c^\infty(M^+)\) in \(H^1_0(M^+)\), the integration-by-parts
identity associated with \eqref{eq:L_OU} applies to these test
functions. Testing the weak formulation of
\eqref{eq:calculation-nu-n} with
\(\psi=\zeta_h^2u_+\) and \(\psi=-\zeta_h^2u_-\), respectively, we obtain
\[
\int_{M^+}\left(
 |\nabla_M(\zeta_hu_\pm)|^2
 -|B_E|^2\zeta_h^2u_\pm^2
\right)\,d\mathcal H^{n-1}_\gamma
=
\int_{M^+}u_\pm^2|\nabla_M\zeta_h|^2\,d\mathcal H^{n-1}_\gamma .
\]
All the quantities involved are supported in the compact set
\(\supp\zeta_h\subset M\), where the second fundamental form is
bounded. After even reflection across \(\{x_n=0\}\), both sides are
multiplied by \(2\), and hence
\begin{equation}\label{eq:nodal_cutoff_identity}
\int_M\left(
 |\nabla_M(\zeta_hu_\pm)|^2
 -|B_E|^2\zeta_h^2u_\pm^2
\right)\,d\mathcal H^{n-1}_\gamma
=
\int_Mu_\pm^2|\nabla_M\zeta_h|^2\,d\mathcal H^{n-1}_\gamma .
\end{equation}

Set
\[
a_h\coloneqq \int_M\zeta_hu_+\,d\mathcal H^{n-1}_\gamma,
\qquad
b_h\coloneqq \int_M\zeta_hu_-\,d\mathcal H^{n-1}_\gamma.
\]
Since \(u_+\) and \(u_-\) are nonzero, \(a_h,b_h>0\) for all
sufficiently large \(h\). Define
\[
c_h\coloneqq \frac{a_h}{b_h},
\qquad
\varphi_h\coloneqq \zeta_h(u_+-c_hu_-).
\]
Then \(\varphi_h\) is \(n\)-symmetric, has compact support in \(M\), and
\[
\int_M\varphi_h\,d\mathcal H^{n-1}_\gamma=0.
\]

By Remark \ref{rem:less regularity} we can apply the stability inequality \eqref{eq:stability} to \(\varphi_h\). 
By the Sobolev chain rule, on \(M^+\) one has
\[
\nabla_Mu_+
=
\chi_{\{\nu_n>0\}}\nabla_M\nu_n,
\qquad
\nabla_Mu_-
=
-\chi_{\{\nu_n<0\}}\nabla_M\nu_n
\quad\text{a.e.}
\]
The same orthogonality is preserved by even reflection. Consequently,
\[
u_+u_-=0,
\qquad
\nabla_Mu_+\cdot\nabla_Mu_-=0
\quad\text{a.e. on }M,
\]
and all the cross terms vanish. Using
\eqref{eq:nodal_cutoff_identity}, we obtain
\begin{equation}\label{eq:negative_nodal_test}
Q(\varphi_h)
=
\int_M\bigl(u_+^2+c_h^2u_-^2\bigr)|\nabla_M\zeta_h|^2\,d\mathcal H^{n-1}_\gamma
-\int_M\zeta_h^2\bigl(u_+^2+c_h^2u_-^2\bigr)\,d\mathcal H^{n-1}_\gamma .
\end{equation}

Set $c:=\lim c_h=\int u_+/\int u_-$.
Since \(\zeta_h\to1\) in \(L^2_\gamma(M)\) and the sequence \((c_h)\) is bounded, 
\[
\int_M
 \bigl(u_+^2+c_h^2u_-^2\bigr)|\nabla_M\zeta_h|^2\,d\mathcal H^{n-1}_\gamma
\leq
\bigl(1+c_h^2\bigr)
\int_M|\nabla_M\zeta_h|^2\,d\mathcal H^{n-1}_\gamma
\longrightarrow0,
\]
while
\[
\int_M
 \zeta_h^2\bigl(u_+^2+c_h^2u_-^2\bigr)\,d\mathcal H^{n-1}_\gamma
\longrightarrow
\int_M\bigl(u_+^2+c^2u_-^2\bigr)\,d\mathcal H^{n-1}_\gamma>0.
\]
It follows from \eqref{eq:negative_nodal_test} that
\(Q(\varphi_h)<0\) for all sufficiently large \(h\), contradicting
the stability inequality. Therefore, \(\nu_n\) cannot change sign on
\(M^+\).

\smallskip
\noindent\textbf{Step 4: Strong maximum principle.}
Let $\sigma\in\{-1,1\}$ be such that $v\coloneqq \sigma\nu_n\geq0$ on $M^+$.
Since $v$ satisfies the same linear equation
\[
\Delta_M v-x\cdot\nabla_M v+|B_E|^2v=0
\qquad\text{on }M^+,
\]
the strong maximum principle implies that either $v>0$ on $M^+$ or
$v\equiv0$ there. Thus either $\nu_n$ has a strict sign on $M^+$ or
$\nu_n\equiv0$ on $M^+$. 
De Giorgi's structure theorem (see \cite[Theorem 15.9]{Maggi}) yields
\begin{equation*}
D\chi_E=-\nu^E \mathcal{H}^{n-1}\mres M.
\end{equation*}
Therefore, in the case $\nu_n\equiv 0$, the distributional partial derivative $D_n\chi_E$ 
is zero and necessarily $E=E'\times\R$ for some set $E'$ of locally
finite perimeter in $\R^{n-1}$. 
\end{proof}

We may now proceed  to the proof of Theorem \ref{theorem reg}. 

\begin{proof}[\textbf{Proof of Theorem \ref{theorem reg}}] 
If $n \leq 7$, then Proposition~\ref{prop:partial-reg} implies that  $\partial E$ is analytic. 
So we can assume $n\geq 8$.
After relabeling the coordinates and iterating
the cylindrical alternative in Lemma~\ref{lem:graph}, exactly one of
the following occurs:
\begin{enumerate}[label=(\roman*)]
\item \emph{Cylindrical case.} There exist $k\in\{1,\ldots,n-1\}$ and
      $E'\subset\R^{n-k}$ such that
      \[
      E=E'\times\R^k;
      \]
\item \emph{Non-cylindrical case.} For every $i\in\{1,\ldots,n\}$ the component $\nu_i$ has a strict  sign on 
$\partial^*E\cap\{x\in\R^n \colon x_i >0\}$.
\end{enumerate}

We first consider the cylindrical case. The set $E'$ is a minimizer of the Gaussian isoperimetric problem 
among $(n-k)$-symmetric sets in $\R^{n-k}$. By adapting Heilman's strategy to the $n$-symmetric case, either 
$E'$ or $(E')^c$ is convex. 
Indeed, writing points of $\R^n$ as $(x',y)\in\R^{n-k}\times\R^k$,
the cylindrical construction in \cite[Lemma~5.4]{He2} uses test functions in the second variation of the form
\[
    \varphi(x',y)=f(x')(y_1^2-1),
    \qquad (x',y)\in\partial^*E'\times\R^k.
\]
The factor $y_1^2-1$ is even and has zero mean with respect to the one-dimensional Gaussian measure. 
Thus $\varphi$ satisfies the first-order volume constraint and is invariant under every coordinate
reflection whenever $f$ is $(n-k)$-symmetric. The geometric base functions used in Heilman's argument 
have precisely this invariance: the mean curvature is invariant under coordinate reflections, while
the second fundamental form transforms by orthogonal conjugation, so its eigenvalues and their regularized 
functions are invariant as well. The auxiliary nonnegative eigenfunctions, chosen on reflection-invariant domains, 
can likewise be averaged over the coordinate-reflection group without changing their eigenvalues, since the 
relevant operators commute with these reflections. The cutoff approximations can also be chosen
reflection-invariant, retaining the zero-mean condition in the cylindrical variable. 
Consequently, the destabilizing variations in \cite{He2} remain admissible for the minimizer $E$ in the
$n$-symmetric class, and the same argument yields the asserted convexity.

If an almost minimizer of the perimeter $E'$, or its complement, is convex then it is well known that the singular set 
is empty (see e.g. \cite[Lemma 3]{DelgadinoMaggi}). 

We henceforth assume that $E$ is not cylindrical in any coordinate
direction. Write $M=\partial^*E$, $\Sigma_E=\partial E\setminus M$, and
$\nu=\nu^E$ for short. Replacing $E$ by its open essential complement if necessary,
we may assume that $\nu_n>0$ on $M^+=M\cap\{x_n>0\}$.
We proceed in four steps.

\smallskip
\noindent\textbf{Step 1: Graph structure and regularity in open orthants.}
Let us define $U_n=\Pi_n(E)$, where $\Pi_n$ is the projection. Since $E$ is open, also $U_n$ is open. 
De Giorgi's structure theorem gives
\[
D_n\chi_E=-\nu_n\,\mathcal H^{n-1}\mres M\leq0
\qquad\text{in }\{x_n>0\}.
\]
This means that for a.e. $x'\in U_n$ the characteristic function $t \mapsto \chi_E(x',t)$ is nonincreasing on $[0, \infty)$.  Using the density estimates, given by the previous regularity result,  and the open precise representative, we deduce that the above monotonicity holds for every $x' \in U_n$. This implies that the set $E \cap \{ x_n>0\}$ can be written as a subgraph of a function. To be more precise,  for every $x' \in U_n$ the section $\{t\geq0:(x',t)\in E\}$ is of the form $[0,u(x'))$, where $0<u(x')\leq\infty$ (note that the case $u(x') = \infty$ is included).

In addition, we deduce from Lemma~\ref{lem:graph} that for every $i = 1, \dots, n-1$ it holds either $\nu_i <0$ or $\nu_i >0$ on $M\cap \{ x_i >0\}$. We may relabel the coordinates such that it holds $\nu_i <0$ for $i = 1, \dots, m$ and $\nu_i >0$ for $i = m+1, \dots, n-1$. Of course, if  $\nu_i >0$ for all $i$ then this relabeling is unnecessary.  Then by repeating the above argument we deduce that the height function $u$ is nondecreasing in $x_i$-directions for $i = 1, \dots, m$ for $x_i >0$ and nonincreasing in $x_i$-directions for $i >m$. 

We remark that on the part of $M^+$, where the height function $u$ has finite value, it holds 
\[
 \sqrt{1+|\nabla u|^2}=\frac1{\nu_n} \qquad \text{and} \qquad \nu_i = \frac{- \partial_{x_i} u}{\sqrt{1+|\nabla u|^2}}.
\]
Denote $O=(0,\infty)^n$ and choose a unit vector $e\in \R^n$ such that 
\[
e = \frac{1}{\sqrt{n}}\left( -\sum_{i = 1}^m e_i+ \sum_{i = m+1}^n  e_i\right)
\]
Then  on $M_O = M\cap O$ we have
\[
e\cdot\nu=\frac{1}{\sqrt n}\sum_{i=1}^n|\nu_i|\geq \frac{1}{\sqrt n} \left(\sum_{i=1}^n\nu_i^2\right)^{\frac12}
\geq\frac{1}{\sqrt n}.
\]
Using the previous argument, we deduce that we may write $\partial E \cap O$ as a graph of a height function $f$ defined on a subset of the hyperplane $e^\perp = \{ x : x \cdot e = 0 \}$. By the above inequality, and by the fact that at regular points it holds $\sqrt{1+|\nabla f|^2}=\frac1{\nu \cdot e}$, we deduce that $f$ is in fact Lipschitz continuous. Using the Euler--Lagrange equation and elliptic
regularity we deduce that in $O$ the boundary $\partial E$ is analytic, i.e., $M_O = \partial E \cap O$ is an analytic hypersurface. We therefore obtain that there is a single parametrization
\[
F:V\to M_O,\qquad F(z)=z+f(z)e,
\qquad V=\pi_e(M_O)\subset e^\perp,
\]
where $\pi_e(x)=x-(x\cdot e)e$  is the projection on $e^\perp$ and $V$ is open. We shall use the bounds
\[
|DF|_{op}\leq\sqrt{n},
\qquad 1\leq J_F=\sqrt{1+|\nabla f|^2}
=\frac1{e\cdot\nu}\leq\sqrt n,
\]
where $J_F$ denotes the Jacobian and $|\cdot|_{op}$ is the operator norm. We need to prove the regularity of $\partial E$ at the points $x \neq 0$ with $x_i \geq 0$ for all $i$ and  $x_j = 0$ at least for one index $j <n$.   
Reflection then gives regularity of $\partial E$ outside the origin. 

\smallskip
\noindent\textbf{Step 2: Energy and mean-value estimates.}
Since $\nu_n>0$ on $M^+$ we may define for $0<\varepsilon\leq1$ 
\[
w=\log\frac{1+\varepsilon}{\nu_n+\varepsilon}
\qquad\text{on }M^+.
\]
It holds $w\geq0$ and it is easy to see that we have the bound $0\leq\nu_nw\leq1$.
Applying the equation~\eqref{eq:calculation-nu-n} we obtain by a straightforward calculation that 
\begin{equation}\label{eq:thm2-21}
\Delta_{\partial E} w-x\cdot\nabla_{\partial E} w
=|B_E|^2\frac{\nu_n}{\nu_n+\varepsilon}+|\nabla_{\partial E} w|^2.
\end{equation}
By Cauchy--Schwarz inequality we obtain 
\begin{equation}\label{eq:thm2-2}
\Delta_{\partial E} w\geq-\frac{|x|^2}{2}+\frac{|\nabla_{\partial E} w|^2}{2}.
\end{equation}

Let us prove an energy estimate from \eqref{eq:thm2-21}, which will be useful later.  Let $\zeta_h$
be the cutoff function from Lemma~\ref{lem:gaussian_cutoff}, and, for $T>0$, set
$f_T=\min\{x_n,T\}$ on $M^+$. The function $\zeta_hf_T$ is supported
away from $\Sigma_E$ and vanishes  on $M\cap\{x_n=0\}$.
We multiply~\eqref{eq:thm2-21} by 
$(\zeta_hf_T)^2$   and integrate by parts to deduce 
\[
\begin{split}
\int_{M^+}(\zeta_hf_T)^2|\nabla_{\partial E}w|^2\,d\mathcal{H}_\gamma^{n-1}
&\leq-2\int_{M^+}\zeta_hf_T\,
\nabla_{\partial E}(\zeta_hf_T)\cdot\nabla_{\partial E} w\,d\mathcal{H}_\gamma^{n-1},
\end{split}
\]
which in turn implies 
\[
\int_{M^+}(\zeta_hf_T)^2|\nabla_{\partial E} w|^2\,d\mathcal{H}_\gamma^{n-1}
\leq4\int_{M^+}|\nabla_{\partial E}(\zeta_hf_T)|^2\,d\mathcal{H}_\gamma^{n-1}.
\]
For fixed $T$, the bounds $f_T\leq T$ and $|\nabla_{\partial E} f_T|\leq1$
imply $\zeta_hf_T\to f_T$ in $H^1_\gamma(M^+)$.
Using Fatou's lemma and  monotone convergence, we obtain
\begin{equation}\label{eq:thm2-23}
\int_{M^+}x_n^2|\nabla_{\partial E}w|^2\,d\mathcal{H}_\gamma^{n-1}
\leq 4 P_\gamma(E).
\end{equation}

We also obtain the following local energy estimate in a similar way. Arguing as above, we deduce that if $\psi$
is a bounded Lipschitz function with bounded support and which vanishes  on $M\cap\{x_n=0\}$, then it holds
\begin{equation}\label{eq:reg-local-energy}
\int_{M^+}\psi^2|\nabla_{\partial E} w|^2\,d\mathcal{H}^{n-1}
\leq2\int_{M^+}|x|^2\psi^2\,d\mathcal{H}^{n-1}
+16\int_{M^+}|\nabla_{\partial E}\psi|^2\,d\mathcal{H}^{n-1}.
\end{equation}
The proof of \eqref{eq:reg-local-energy} is similar to that of \eqref{eq:thm2-23}, with the difference that we multiply \eqref{eq:thm2-2} by $(\zeta_h\psi)^2$ and integrate by parts. 

\smallskip
\noindent\emph{A local mean-value estimate.}
For $y\in M^+$ and $0<r_0\leq\min\{1,y_n/4\}$, we claim that
\begin{equation}\label{eq:thm2-3}
w(y)\leq\frac{C}{r_0^{n-1}}
\int_{M\cap B_{r_0}(y)}(w+1)\,d\mathcal H^{n-1}.
\end{equation}
The constant is uniform when the balls lie in a fixed bounded subset
of $\{x_n>0\}$. Following the argument of
\cite[Theorem~13.2]{Giusti}, for $0<s<r\leq r_0$ define, with
$\rho=|x-y|$,
\[
\varphi_s(x)=
\begin{cases}
\displaystyle
\frac{s^{3-n}-r^{3-n}}{2(n-3)}
+\frac{r^{1-n}-s^{1-n}}{2(n-1)}\rho^2,
& 0\leq\rho\leq s,\\
\displaystyle
\frac{\rho^{3-n}}{(n-1)(n-3)}
+\frac{r^{1-n}\rho^2}{2(n-1)}
-\frac{r^{3-n}}{2(n-3)},
&s<\rho<r,\\
0,&\rho\geq r.
\end{cases}
\]
Since $n\geq8$, we have $0\leq\varphi_s\leq Cs^{3-n}$ and
$|\nabla\varphi_s|\leq Cs^{2-n}$. By  the Euler--Lagrange equation the mean curvature satisfies  
$\Hh_E=x\cdot\nu+\lambda$, and is locally bounded. The tangential
Laplacian therefore satisfies
\[
\Delta_{\partial E}\varphi_s
\leq r^{1-n}\chi_{B_r(y)}-s^{1-n}\chi_{B_s(y)}
+Cs^{2-n}\chi_{B_r(y)}.
\]
Using the cutoff function from Lemma~\ref{lem:gaussian_cutoff}, localized to
$B_r(y)$, and~\eqref{eq:reg-local-energy} to pass to the limit in
integration by parts, we deduce from  \eqref{eq:thm2-2} that 
\[
\int_M w\Delta_{\partial E}\varphi_s\,d\mathcal{H}^{n-1} = \int_M \Delta_{\partial E} w\varphi_s\,d\mathcal{H}^{n-1}
\geq-C\int_M\varphi_s\,d\mathcal{H}^{n-1}.
\]
The estimate for the perimeter $\mathcal{H}^{n-1}(M\cap B_r(y))\leq Cr^{n-1}$ then  gives
\begin{equation*}
\begin{split}
\frac1{s^{n-1}}\int_{M\cap B_s(y)}w\,d\mathcal{H}^{n-1}
&\leq\frac1{r^{n-1}}\int_{M\cap B_r(y)}w\,d\mathcal{H}^{n-1}\\
&\quad+\frac{C}{s^{n-2}}\int_{M\cap B_r(y)}w\,d\mathcal{H}^{n-1}
+C\frac{r^{n-1}}{s^{n-3}}.
\end{split}
\end{equation*}
For $r_j=2^{-j}r_0$ and
$A_j=r_j^{1-n}\int_{M\cap B_{r_j}(y)}w\,d\mathcal{H}^{n-1}$,
this implies
\[
A_{j+1}\leq(1+Cr_j)A_j+Cr_j^2.
\]
Then for $\tilde A_j = \max\{ A_j, 1\}$ it holds $\tilde A_{j+1}\leq(1+Cr_j)\tilde A_j$, after possibly increasing the constant. By  iterating this we obtain the bound 
\[
A_j \leq C(A_0 + 1)
\]
for all $j$. Letting $j\to\infty$ at the regular point $y$ proves
\eqref{eq:thm2-3}.

\smallskip
\noindent\textbf{Step 3: Propagation of singularities along coordinate lines.}
We prove the propagation property stated in~\eqref{eq:reg-singular-line} below. Roughly speaking it states that if a point $x$ is on the singular set then the whole line $\{x + te_n\}$ belongs to the singular set. 

For $x'\in\widetilde U_n=\Pi_n(\partial E)$ we define
\[
u_n(x')=\sup\{t\geq0:(x',t)\in\partial E\},
\qquad
v_n(x')=\inf\{t\geq0:(x',t)\in\partial E\}.
\]
We note that again we allow $u_n(x') = \infty$. As $\partial E$ is regular in $O$ and $\nu_n >0$ there, then it holds $v_n(x') = u_n(x')$ for $x' \in (0,\infty)^{n-1}\cap \widetilde U_n$. At other points it may happen that $v_n(x')< u_n(x')$ (although we are going to eventually prove that this is not the case) but in that case the whole interval $\{x'\} \times [v_n(x'), u_n(x')]$ belongs to the boundary $\partial E$.  Indeed, an interior point of either open phase between two boundary points would contradict the corresponding vertical monotonicity. Similarly we deduce that $u_n$ is upper semicontinuous and $v_n$ is lower
semicontinuous, relative to $\widetilde U_n$. 

We proceed by proving that a point $y=(y',y_n)\in\partial E$, with
$y_n>0$, is regular whenever $u_n(y')<\infty$ or $v_n(y')>0$.
Fix $y=(y',y_n)\in\partial E$ with $y_n>0$, and suppose that
$u_n(y')<\infty$ or $v_n(y')>0$. For $\eta>0$ set
\[
U_\eta=\{x'\in\widetilde U_n:u_n(x')<1/\eta\},
\qquad
U_\eta'=\{x'\in\widetilde U_n:v_n(x')>\eta\}.
\]
Choose $\eta<y_n$ so that $y'$ belongs to either  $U_\eta$ or $U_\eta'$.
By semicontinuity $U_\eta$ and $U_\eta'$ are relatively open sets.  Therefore we may choose $0<\delta<\min\{\eta/8,y_n/8,1/8\}$ such that
$B_\delta(y')\cap\widetilde U_n$ is contained in $U_\eta$ in the
first case, or in $U_\eta'$ in the second case.
Let us fix a cutoff function  $\zeta\in C_0^\infty(B_{\delta/2}(y'))$ such that 
$0\leq\zeta \leq 1$, $\zeta=1$ on $B_{\delta/4}(y')$, and
$|\nabla \zeta |\leq C/\delta$. We define also the cylinder
$C_t(y)=\{x:|x'-y'|<t\}$.

In the case $y' \in U_\eta$  we denote  $D=M^+$ and choose $h(x_n)=x_n$.
Note that then in the set  $x \in D \cap \supp\zeta$ it holds $x_n \leq \frac{1}{\eta}$.
In the case $y' \in U_\eta'$  we denote  $D=M^+\cap\{x_n<2y_n\}$ and choose  $h(x_n)=(2y_n-x_n)_+$.
In this case for  $x \in D \cap \supp\zeta$ it holds $\eta \leq x_n \leq 2y_n$. In both cases, we will prove
\begin{equation}\label{eq:thm2-4}
\int_{D\cap C_{\delta/4}(y)}w\,d\mathcal{H}^{n-1}\leq C.
\end{equation}

Consider first $h(x_n)=x_n$. Using the  cutoff functions  as above, we obtain by integration by parts
\[
\int_{M^+}\nabla_{\partial E}(h\zeta  w)\cdot e_n\,d\mathcal{H}^{n-1}
=\int_{M^+}\Hh_E h \zeta \,  w\nu_n\,d\mathcal{H}^{n-1}.
\]
Since $\nabla_{\partial E} x_n\cdot e_n=1-\nu_n^2$ and
$\nabla_{\partial E}\zeta\cdot e_n=-(\nabla\zeta \cdot\nu)\nu_n$, we obtain
\begin{equation}\label{eq:thm2-5}
\begin{split}
\int_D\zeta w\,d\mathcal{H}^{n-1}
&=-\int_D h\zeta\,\nabla_{\partial E} w\cdot e_n\,d\mathcal{H}^{n-1}\\
&\quad+\int_D h(\nabla\zeta \cdot\nu)w\nu_n\,d\mathcal{H}^{n-1}
+\int_D (\Hh_E h+\nu_n)\zeta w\nu_n\,d\mathcal{H}^{n-1}.
\end{split}
\end{equation}
Since  $\Hh_E$ is locally bounded, and recalling the bound $x_n \leq \frac{1}{\eta}$ on $x \in D \cap \supp\zeta$  and $\nu_nw\leq1$, we have 
\[
\left|\int_Dh(\nabla\zeta \cdot\nu)w\nu_n\,d\mathcal{H}^{n-1}\right|
+\left|\int_D(\Hh_E h+\nu_n)\zeta  w\nu_n\,d\mathcal{H}^{n-1}\right|
\leq C.
\]
On the other hand, choosing  $\psi=h\zeta $ in the estimate~\eqref{eq:reg-local-energy} yields 
\begin{equation}\label{eq:thm2-7}
\int_{M^+}\psi^2|\nabla_{\partial E} w|^2\,d\mathcal{H}^{n-1}
\leq C+16\int_{M^+}|\nabla_{\partial E}\psi|^2\,d\mathcal{H}^{n-1}
\leq C.
\end{equation}
By Cauchy--Schwarz we then deduce by the above inequalities that the left hand side of 
\eqref{eq:thm2-5} is bounded. This proves~\eqref{eq:thm2-4} in the case $y' \in U_\eta$.

In  the case  $y' \in U_\eta'$ the argument is similar.  In this case, since $h(x_n)=(2y_n-x_n)_+$, it holds  $h'=-1$ on $0<x_n<2y_n$. The same argument implies, with the signs reversed accordingly,
\[
\begin{split}
\int_D\zeta w\,d\mathcal{H}^{n-1}
&=\int_Dh\zeta \,\nabla_{\partial E} w\cdot e_n\,d\mathcal{H}^{n-1}\\
&\quad-\int_Dh(\nabla\zeta\cdot\nu)w\nu_n\,d\mathcal{H}^{n-1}
+\int_D(\nu_n-\Hh_E h)\zeta w\nu_n\,d\mathcal{H}^{n-1}.
\end{split}
\]
Again there is no contribution from $x_n=0$, since for 
$x \in D \cap \supp\zeta$ it holds  $x_n > \eta$, and $h$ vanishes at $x_n=2y_n$.
Using the same bounds as above, including~\eqref{eq:thm2-7}, we obtain  
\eqref{eq:thm2-4}.

We may apply~\eqref{eq:thm2-3} with radius $\delta/16$ at every regular point
$z\in M\cap B_{\delta/16}(y)$ and deduce that  $w(z)\leq C$
uniformly in $z$ and $\varepsilon$. Letting $\varepsilon\to 0$
implies 
\[
\nu_n(z)\geq c>0
\qquad z\in M\cap B_{\delta/16}(y).
\]
This estimate also controls the boundary at potentially singular
points. Indeed, for $v \in e_n^\perp$ with $|v|< c/2$
\[
D_{e_n +v} \chi_E=-(\nu_n + \nu \cdot v) \,\mathcal{H}^{n-1}\mres M\leq0
\qquad\text{in} \, \, B_{\delta/16}(y) .
\]

The two open phases consequently satisfy a uniform local cone condition.
 Thus $\partial E$ is a Lipschitz graph near $y$, and the Euler--Lagrange equation and elliptic regularity imply that  it is analytic.

It follows that at a singular point $y=(y',y_n)$ with $y_n>0$ it must hold both 
$v_n(y')=0$ and $u_n(y')=\infty$. Since the whole interval $y' \times [v_n(y'), u_n(y')]$ belongs to the boundary $\partial E$, then by symmetry  the entire line
$l_y = \{ y+ t e_n: t \in \R\}$ lies in $\partial E$. Moreover, every point with
nonzero $n$th coordinate on this line  is singular. Indeed, if such a point  $x$ were regular then it would hold   $\nu_n(x)=0$, which contradicts 
Lemma~\ref{lem:graph}. Also the point on the line with zero $n$th coordinate has to be singular since  $\Sigma_E$ is closed. Applying the argument in each coordinate direction, we may rephrase the previous observation in a compact way as follows 
\begin{equation}\label{eq:reg-singular-line}
x\in\Sigma_E,\quad x_i\neq0
\quad\Longrightarrow\quad l_x = \{ x+ t e_i: t \in \R \} \subset\Sigma_E.
\end{equation}

\smallskip
\noindent\textbf{Step 4: Exclusion of nonzero singular points.}
We need to prove that $\Sigma_E \cap \{x_n >0\}$ is empty.  Suppose, for contradiction  that it is not and choose $y \in\Sigma_E \cap \{x_n >0\}$ with the smallest possible number of
zero coordinates. Denote 
\[
J=\{i =1, \dots, n-1:y_i=0\},\qquad k=|J|,
\qquad H=\{x\in\R^n:x_i=0\text{ for every }i\in J  \}.
\]
Regularity in the set $O$ implies  $1\leq k\leq n-1$.
We may assume that $y_j>0$ for $j\notin J$.  Iterating~\eqref{eq:reg-singular-line} in the directions outside $J$ implies  $H\subset\Sigma_E$. Since the Hausdorff dimension of $\Sigma_E$ is small, we have   $k \geq 2$. 

Recall that we choose the coordinates such that either the height function $u$ is nondecreasing in $x_i$-directions in $\{x_i>0\}$  for $i = 1, \dots, m$ and nonincreasing for others (or it is nonincreasing  in every direction). Let us then define the set of indices 
\[
J_- = \{ i \in J : \nu_i <0 \quad \text{on }\, M_O\} \quad \text{and} \quad J_+ = \{ i \in J : \nu_i >0 \quad \text{on }\, M_O\}.
\]
We claim that in this case both $J_-$ and $J_+$ are nonempty.  Indeed, if $J_-$ is empty then the height function would be nonincreasing in every direction in $i \in J$. Then if $x \in O$ is  an interior point of $E$, the point $x - \sum_{i \in J} x_i e_i$ would also be an interior point of $E$, which contradicts  the fact that  $H\subset\partial E$. If $J_+$ were empty, then if  $x \in O$ is  an exterior point of $E$, the point $x - \sum_{i \in J} x_i e_i$ would also be an exterior point of $E$, which is again a contradiction. 

Choose
\[
0<r<\frac14\min\big\{1,\min_{j\notin J}y_j\big\}
\qquad \text{and} \qquad 
Q=\{x\in\R^n:x_j\geq y_j/2\text{ for every }j\notin J\}.
\]
Every point in $Q\setminus H$ has fewer than $k$ zero coordinates
and is nonzero. By the choice of $y$, it holds 
\[
\Sigma_E\cap Q=H\cap Q.
\]
Set $\Sigma_1=M_O\cap B_r(y)$ and $V_1=\pi_e(\Sigma_1)$.
We translate $V_1$ simultaneously in all directions $e_i$ for $i =1, \dots,  m$, i.e., in direction in which the height function $u$ is nondecreasing. To this aim we recall the choice of the unit vector $e$ from the beginning of the proof and  define
\[
b=\sum_{i=1}^m e_i \qquad \text{and}
\qquad \hat b= \pi_{e}(b) = b-(b\cdot e)e\in e^\perp.
\]

Let us fix  $z\in V_1$ and define the curve $X_z(s)=F(z+s\hat b)$, wherever it is defined, i.e., when $z+s\hat b \in V$. By differentiating we obtain for $\hat \sigma =  -\sum_{i=1}^m e_i + \sum_{i=m+1}^n e_i$ that 
\[
X_z'(s)
=b-\frac{\nu(X_z(s))\cdot b}{\nu(X_z(s))\cdot e}\,e
=b+\theta(s)\hat \sigma,
\qquad
\theta(s)=\frac{\sum_{i=1}^m|\nu_i(X_z(s))|}
{\sum_{i=1}^n|\nu_i(X_z(s))|}\in(0,1).
\]
In particular, by the choice of $b$ it holds 
\begin{equation}\label{eq:thm2-curve}
(X_z')_i=
\begin{cases}
1-\theta(s),&\text{for} \,\, i = 1, \dots, m,\\
\theta(s),&\text{for} \,\, i = m+1, \dots, n.
\end{cases}
\end{equation}
This means that every coordinate of  $X_z(s)$ is nondecreasing and $|X_z'|\leq\sqrt n$.
For each $z \in V_1$, it holds $X_z(s) \cdot e_i \geq X_z(0) \cdot e_i > 0  $ and therefore $X_z(s)  \in O$ for all $s>0$ when it is defined. The curve  $s \mapsto X_z(s)$ thus  stays in a compact subset of $M_O$ and  is smooth. Thus it is well defined for all  $0\leq s\leq1$ and  we have 
\[
V_2=V_1+\hat b\subset V \qquad \text{and} 
\qquad
\hat V= \{x + s \hat b : x \in V_1, s \in [0,1] \} \subset V.
\]

Let us define  $\Sigma_2=F(V_2)$ and $\hat\Sigma=F(\hat V)$. By the choice of  $V_1$ it holds $F(z)\cdot e_n \geq \frac{y_n}{2}$ for $z \in V_1$, and thus by the above monotonicity it holds $X_z(s)\cdot e_n \geq \frac{y_n}{2}$ for all $s \in [0,1]$. 
For  $R=|y|+r+\sqrt n$, these bounds imply
\[
\hat\Sigma\subset Q\cap B_R(0) \qquad \text{and} 
\qquad x_n\geq y_n/2\quad\text{on }\hat\Sigma.
\]

Fix $i\in J_-$ and $j\in J_+$. The  formula \eqref{eq:thm2-curve} for $(X_z')_i$  gives the
 identity
\[
\frac{d}{ds}\big((X_z)_{i}+(X_z)_{j}\big)=1.
\]
Since both initial coordinates $i$ and  $j$   are positive, we have 
\[
\text{dist}(X_z(1),H)\geq\frac1{\sqrt2}.
\]
Thus $\overline{\Sigma}_2$ is contained in the compact set
\[
K=\partial E\cap Q\cap\overline{B_R(0)}
\cap\{x:\text{dist}(x,H)\geq1/\sqrt2\}\subset M^+.
\]
Lemma~\ref{lem:graph} together with  compactness imply $\nu_n\geq c_K>0$ on $K$.
Since  $K$ is bounded, this yields
\[
\int_{\Sigma_2}w\,d\mathcal{H}^{n-1}\leq C,
\]
where $w$ is the function defined at the start of Step 2.

Write $W=w\circ F$ and recall that $F$ is Lipschitz regular with $|DF|_{op} \leq \sqrt{n}$. Since,  for each $s\in[0,1]$, the translated domain
$V_1+s\hat b$ is contained in $\hat V$, the fundamental theorem of calculus along the
translated segments implies 
\[
\int_{V_1}W\,dz
\leq\int_{V_2}W\,dz
+\sqrt n\int_{\hat V}|\nabla_{\partial E} w|(F(z))\,dz.
\]
The area formula and the bounds on the Jacobian  $J_F$
then imply
\[
\int_{\Sigma_1}w\,d\mathcal{H}^{n-1}
\leq C\int_{\Sigma_2}w\,d\mathcal{H}^{n-1}
+C\int_{\hat\Sigma}|\nabla_{\partial E} w|\,d\mathcal{H}^{n-1}.
\]
Note that the Gaussian and the Euclidean surface measures are comparable on $B_R(0)$.
Using also the estimate $x_n\geq y_n/2$ on $\hat\Sigma$ and~\eqref{eq:thm2-23}, we obtain
\[
\begin{split}
\int_{\hat\Sigma}|\nabla_{\partial E} w|\,d\mathcal{H}^{n-1}
&\leq\frac{2C_{n,R}}{y_n}
\int_{\hat\Sigma}x_n|\nabla_{\partial E} w|\,d\mathcal{H}_\gamma^{n-1}\\
&\leq\frac{2C_{n,R}}{y_n}
\left(\int_{\hat\Sigma}x_n^2|\nabla_{\partial E} w|^2
\,d\mathcal{H}_\gamma^{n-1}\right)^{1/2}P_\gamma(E)^{1/2}
\leq C.
\end{split}
\]
By recalling the definition of $\Sigma_1$, we have
\begin{equation}\label{thm2-st3-a}
\int_{M_O\cap B_r(y)}w\,d\mathcal{H}^{n-1}\leq C.
\end{equation}

Reflection in the coordinate hyperplanes indexed by $J$ preserves
$B_r(y)$, $\nu_n$ and $w$, and covers all neighboring orthants.
At a regular point with $x_i=0$, symmetry gives $\nu_i=0$. Therefore
$\partial^* E$ intersects  that coordinate hyperplane orthogonally  and the
intersection has zero $\mathcal{H}^{n-1}$-measure. Since the singular
set is also $\mathcal{H}^{n-1}$-negligible, reflection of
\eqref{thm2-st3-a} implies
\[
\int_{M \cap B_r(y)}w\,d\mathcal{H}^{n-1}\leq C.
\]
Applying~\eqref{eq:thm2-3} with radius $r/4$ at every
$x \in M \cap B_{r/2}(y)$ we obtain as before that $w(x)\leq C$ uniformly in
$x$ and $\varepsilon$, and therefore $\nu_n\geq c>0$ there. By the Euler--Lagrange equation and elliptic regularity the boundary is analytic in a neighborhood of $y$. This contradicts the fact that $y\in\Sigma_E$ and proves that $\Sigma_E \cap \{x_n>0\}$ is empty. Repeating the argument in every coordinate direction and using reflection yields  
$\Sigma_E\subset\{0\}$.
\end{proof}

%%%%%%%%%%%%%%%%%%%%%%%%%%%%%%%%%%%%%%%%%%%%%%%%%%%%%%%%%%%%%%%%%%%%%%%%%%%%%%%%%%%%%%%%%

\section{Rigidity}\label{sect rig}

In this section we resume the notation $(x,y)\in\R^2\times\R^{n-2}$, and $E$ is a minimizer of the 
Gaussian isoperimetric problem among $n$-symmetric sets with prescribed Gaussian measure. Recall that 
$E$ is identified with its precise representative \eqref{eq: precise representative}, which is an open set, 
and that by Theorem~\ref{theorem reg} the set $\partial E\setminus\{0\}=\partial^*E\setminus\{0\}$ is a 
real analytic hypersurface. Since $E^c$ is also a minimizer (for the complementary volume $1-\gamma(E)$), 
Theorem~\ref{theorem reg} applies to $E^c$ as well: its precise representative 
$\{u_{E^c}=1\}=\{u_E=0\}=\R^n\setminus\overline E$ is open.

Define the ``arc-projection'' $\Pi:\R^2\times\R^{n-2}\to[0,\infty)\times\R^{n-2}$ as $\Pi(x,y)\coloneqq(|x|,y)$. 
If $B\subset(0,\infty)\times\R^{n-2}$, then $\Pi^{-1}(B)=\Phi(B\times\Ss^1)$, where $\Phi$ is the 
diffeomorphism defined in \eqref{eq:diffeo}. Set $\R^2_+\coloneqq\{x\in\R^2\colon x_1>0,\ x_2>0\}$ and 
$\omega_\theta\coloneqq(\cos\theta,\sin\theta)$. Recalling \eqref{def:mu} and \eqref{def:beta}, for every 
$(r,y)\in(0,\infty)\times\R^{n-2}$
\begin{equation}\label{eq:beta-angular}
\beta_\mu(r,y)=\frac{1}{4}\int_0^{2\pi}\chi_E(r\omega_\theta,y)\,d\theta.
\end{equation}

\smallskip

\begin{lemma}\label{lemma:connected-slices}
Let $E\subset\R^n$ be a minimizer of the Gaussian isoperimetric problem among $n$-symmetric sets 
with prescribed Gaussian measure. 
Then, for every $(r,y)\in(0,\infty)\times\R^{n-2}$,
\[
\widetilde E_{(r,y)}\coloneqq E_{(r,y)}\cap\{x_1\geq0,\ x_2\geq0\}
\]
is connected. Moreover, if $\widetilde E_{(r,y)}$ is non-empty, then it contains at least one of the 
two points $(r,0)$ and $(0,r)$.
\end{lemma}

\begin{proof}
Set $G\coloneqq\R^n\setminus\overline E$, which is open. Every $p\in\partial E$ is a limit of points of $G$: otherwise some ball 
$B_\rho(p)$ would be contained in $E\cup\partial E$, hence in $E$ up to the $\mathcal{H}^n$-null set 
$\partial E$, and then in $E=E^{(1)}$, contradicting $p\notin E$.

\smallskip
\noindent\emph{Almost-everywhere structure of the fibers.}
By minimality, \eqref{per ineq} holds with equality for every Borel set $B\subset(0,\infty)\times\R^{n-2}$. 
By Theorem~\ref{theorem ineq}(i), for $\mathcal{H}^{n-1}$-a.e. $(\rho,\eta)\in(0,\infty)\times\R^{n-2}$ the angular section
\[
I(\rho,\eta)\coloneqq\left\{\theta\in[0,\pi/2]:(\rho\omega_\theta,\eta)\in E\right\}
\]
agrees, up to $\mathcal{H}^1$-null sets, with the empty set or with an interval attached to one of the 
endpoints of $[0,\pi/2]$.

Consequently, \emph{no circular fiber can contain three points, in strictly increasing angular order in 
$[0,\pi/2]$, belonging respectively to $E, G, E$ or to $G, E, G$}. Indeed, since 
$E$ and $G$ are open, either configuration persists for all parameters $(\rho,\eta)$ in a 
non-empty open neighborhood of the given fiber, and each of its three points has a relative angular 
neighborhood of positive length in the corresponding phase. This contradicts the preceding 
almost-everywhere description.

\smallskip
\noindent\emph{Connectedness.}
Suppose that $\widetilde E_{(r,y)}$ is disconnected. Since $E$ is open, there are angles
$0<\theta_1<\theta_2<\theta_3<\pi/2$ such that
\[
(r\omega_{\theta_1},y),\ (r\omega_{\theta_3},y)\in E,
\qquad
p\coloneqq(r\omega_{\theta_2},y)\notin E.
\]
If $p\in G$, this is already a forbidden configuration. Otherwise $p\in\partial E$, and $p$ can be 
approached by points of $ G$. Choose $(\rho\omega_\theta,\eta)\in G$ so close to $p$ that 
$\theta_1<\theta<\theta_3$ and $(\rho\omega_{\theta_1},\eta),(\rho\omega_{\theta_3},\eta)\in E$. 
This gives the forbidden configuration $E, G,E$ on the fiber $(\rho,\eta)$. 
Thus $\widetilde E_{(r,y)}$ is connected.

\smallskip
\noindent\emph{The endpoint assertion.}
Write $\nu_i\coloneqq\nu^E\cdot e_i$, $i=1,2$. By Lemma~\ref{lem:graph}, set $\sigma_i\coloneqq0$ in the 
cylindrical alternative, and otherwise let $\sigma_i\in\{-1,1\}$ be such that $\sigma_i\nu_i>0$ on 
$\partial^*E\cap\{x_i>0\}$. By \eqref{eq:de_giorgi_structure}, $D_i\chi_E=-\nu_i\,\mathcal{H}^{n-1}\mres\partial^*E$, 
so that
\[
\sigma_i D_i\chi_E\le0\ \text{ in }\{x_i>0\}\quad\text{if }\sigma_i\in\{-1,1\},
\qquad
D_i\chi_E=0\ \text{ in }\R^n\quad\text{if }\sigma_i=0 .
\]
By the one-dimensional slicing theorem for $BV$ functions (see \cite{AFP}), writing $x=(x',x_i)$ up to a 
permutation of the coordinates, for a.e. $x'\in\R^{n-1}$ the function $t\mapsto\chi_E(x',t)$ is a.e. equal on 
$(0,\infty)$ to a nonincreasing function if $\sigma_i=1$, to a nondecreasing function if $\sigma_i=-1$, 
and to a constant if $\sigma_i=0$. We claim that the corresponding statements hold pointwise for $E=E^{(1)}$. 
First, membership in $E$ is preserved when $|x_i|$ is decreased 
if $\sigma_i=1$, when $|x_i|$ is increased if $\sigma_i=-1$, and is independent of $x_i$ if $\sigma_i=0$. 
Indeed, let $q\in E$ with $q_i\geq0$ (the case $q_i<0$ follows by symmetry) and let 
$U'\times(q_i-\delta,q_i+\delta)\subset E$ be an open product neighborhood of $q$. If $\sigma_i=1$, 
the a.e. monotonicity of the sections gives $|(U'\times(0,q_i+\delta))\setminus E|=0$, and the 
reflection symmetry of $E$ gives $|(U'\times(-q_i-\delta,q_i+\delta))\setminus E|=0$. Every point of 
this open set has density one, so it is contained in $E=E^{(1)}$. The cases $\sigma_i=-1$ 
(yielding $U'\times(q_i-\delta,\infty)\subset E$) and $\sigma_i=0$ (yielding $U'\times\R\subset E$) 
are analogous.

Suppose now that $\widetilde E_{(r,y)}$ is non-empty but that
\[
p_1\coloneqq(r,0,y)\notin E,\qquad p_2\coloneqq(0,r,y)\notin E.
\]
Choose $q=(x_1,x_2,y)\in E$ on this fiber with $x_1,x_2>0$. If $\sigma_1\geq0\geq\sigma_2$, 
first decreasing $x_1$ to zero and then increasing $x_2$ to $r$ gives $p_2\in E$. If 
$\sigma_2\geq0\geq\sigma_1$, the analogous changes give $p_1\in E$. Both are contradictions, and these 
cases include all cylindrical alternatives and all opposite-sign alternatives. We are therefore left with 
$\sigma_1=\sigma_2=\sigma\in\{-1,1\}$.

For $i=1,2$, if $p_i\in G$, then $p_i+\sigma te_i\in G$ for all sufficiently small $t>0$, 
by openness. If instead $p_i\in\partial E$, then $p_i\neq0$ and $(p_i)_i=r>0$; hence 
$p_i\in\partial^*E\cap\{x_i>0\}$ by Theorem~\ref{theorem reg}, and Lemma~\ref{lem:graph}, whose 
strict-sign conclusion holds on the whole set $\partial^*E\cap\{x_i>0\}$ (including points where other 
coordinates vanish), gives
\[
\nu^E(p_i)\cdot(\sigma e_i)=\sigma\nu_i(p_i)>0 .
\]
Since $\partial E$ is a $C^1$ hypersurface near $p_i$, with $E$ on the side of $-\nu^E$, the same 
conclusion $p_i+\sigma te_i\in G$ follows by transversality. Consequently, for all sufficiently 
small $t>0$, with $r_t\coloneqq r+\sigma t>0$,
\[
(r_t,0,y),\ (0,r_t,y)\in G,
\qquad
q_t\coloneqq\Big(\frac{r_t}{r}x_1,\frac{r_t}{r}x_2,y\Big)\in E,
\]
where the last inclusion follows from the openness of $E$. These three points lie on the fiber $(r_t,y)$, 
at the angles $0<\arctan(x_2/x_1)<\pi/2$, and give the forbidden configuration $ G,E, G$. 
This contradiction proves the endpoint assertion.
\end{proof}

\begin{proof}[Proof of Theorem~\ref{theorem rig}]
Recall that $E=E^{(1)}$ and $F=F_\mu^{(1)}$. By Theorem~\ref{theorem ineq} and minimality, $F$ is also a minimizer. 
Hence Theorem~\ref{theorem reg}, Proposition~\ref{prop:components_structure}, Lemma~\ref{lem:graph} and 
Lemma~\ref{lemma:connected-slices} apply to both $E$ and $F$; in particular $F$ is open and 
$\partial F\setminus\partial^*F\subset\{0\}$. 
If $E$ is cylindrical in both directions $x_1$ and $x_2$ (Lemma~\ref{lem:graph}(i)), then $E=\R^2\times A$ 
for the precise representative $A\subset\R^{n-2}$ of the corresponding section, every circular fiber of $E$ 
is empty or full, $F=E$, $\Omega_\mu=\emptyset$, and (i) holds. We therefore assume henceforth that 
at most one of $x_1,x_2$ is a cylindrical direction for $E$.

\smallskip
\noindent\textbf{Step 1: the precise representative has the same pointwise circular distribution.}
We claim that
\begin{equation}\label{eq:rigidity-pointwise-distribution}
\mathcal H^1_\gamma(F_{(r,y)})=\mu(r,y)
\qquad\text{for every }(r,y)\in(0,\infty)\times\R^{n-2}.
\end{equation}
Fix $(r,y)$ and write $C_{r,y}\coloneqq\partial B(r)\times\{y\}$. The circle $C_{r,y}$ stays away from the 
origin, so $\partial E$ is analytic in a neighborhood of it by Theorem~\ref{theorem reg}; hence either 
$C_{r,y}\cap\partial E$ is finite or $C_{r,y}\subset\partial E$. Indeed, if $C_{r,y}\cap\partial E$ has an 
accumulation point $p$, a local analytic defining function of $\partial E$ near $p$ vanishes at infinitely 
many points of the analytic curve $\theta\mapsto(r\omega_\theta,y)$ accumulating at $p$, hence identically 
along an arc; therefore the set of $\theta\in\Ss^1$ such that an arc of $C_{r,y}$ around $(r\omega_\theta,y)$ 
is contained in $\partial E$ is non-empty, open, and closed by the same argument, so it is all of $\Ss^1$.

\smallskip
\noindent\emph{Case 1: $C_{r,y}\cap\partial E$ is finite.} 
Then, for all but finitely many $\theta\in[0,2\pi)$, the point $(r\omega_\theta,y)$ belongs to one of the two 
open phases $E$ and $\R^n\setminus\overline E$, so that $\chi_E(\rho\omega_\theta,\eta)\to\chi_E(r\omega_\theta,y)$ 
as $(\rho,\eta)\to(r,y)$. By \eqref{eq:beta-angular} and dominated convergence, $\beta_\mu$ is continuous 
at $(r,y)$. Therefore, for every $\theta\in(0,\pi/2)$ with $\theta\neq\beta_\mu(r,y)$, by \eqref{def:F} the 
point $(r\omega_\theta,y)$ has a neighborhood contained in $F_\mu$ if $\theta<\beta_\mu(r,y)$, and a 
neighborhood disjoint from $F_\mu$ if $\theta>\beta_\mu(r,y)$: in the first case it belongs to $F$, in the 
second it does not. By symmetry, $F_{(r,y)}$ and $(F_\mu)_{(r,y)}$ differ at most at the four angular 
endpoints $\pm\beta_\mu(r,y)$, $\pi\pm\beta_\mu(r,y)$ and at the four axis points, and 
\eqref{eq:rigidity-pointwise-distribution} follows since $F_\mu$ is $\mu$-distributed.

\smallskip
\noindent\emph{Case 2: $C_{r,y}\subset\partial E$.} 
Then $\mu(r,y)=0$, since $E$ is open. The circle is tangent to $\partial E$, so
\[
\nu_x^E(p)=a(p)\,\hat x,\qquad p=(x,y)\in C_{r,y},
\]
for the continuous function $a\coloneqq\nu^E_x\cdot\hat x$. Neither $x_1$ nor $x_2$ is a cylindrical 
direction for $E$: if, say, $\nu_1\equiv0$ on $\partial^*E$, then $a=0$ at every point of $C_{r,y}$ with 
$x_1\neq0$, hence $\nu_2=a\,\hat x_2=0$ at a point of $\partial^*E\cap\{x_2>0\}$, and 
Lemma~\ref{lem:graph} makes $x_2$ cylindrical too, contrary to our reduction. Let then 
$\sigma_1,\sigma_2\in\{-1,1\}$ be the strict signs of $\nu_1$ on $\partial^*E\cap\{x_1>0\}$ and of $\nu_2$ on 
$\partial^*E\cap\{x_2>0\}$ given by Lemma~\ref{lem:graph}. At an open-quadrant point of $C_{r,y}$ both 
$\nu_1=a\hat x_1$ and $\nu_2=a\hat x_2$ have the sign of $a\neq0$, so $\sigma_1=\sigma_2=:\sigma$; 
at the axis points $(\pm r,0,y)$, $(0,\pm r,y)$ one has $\sigma a>0$ by the same lemma; and 
$a(A_jp)=a(p)$ for $j=1,2$ by the reflection symmetry of $E$. Hence, by continuity and compactness,
\[
\inf_{p\in C_{r,y}}\sigma a(p)>0 .
\]
Since $\partial E$ is $C^1$ in a neighborhood of the compact set $C_{r,y}$, with $E$ on the side of $-\nu^E$, 
there exists $t_0>0$ such that for every $t\in(0,t_0)$
\[
C_{r-\sigma t,y}\subset E,
\qquad
C_{r+\sigma t,y}\subset\R^n\setminus\overline E .
\]
Each inclusion persists on a neighborhood $N^\mp_t$ of $(r\mp\sigma t,y)$ in $(0,\infty)\times\R^{n-2}$, 
on which $\beta_\mu$ is identically $\pi/2$, respectively $0$. Thus the open set $\Phi(N^-_t\times\Ss^1)$ is 
contained in $F_\mu$ up to an $\mathcal H^n$-null set, hence in $F$, while the open set 
$\Phi(N^+_t\times\Ss^1)$ is disjoint from $F_\mu$, hence from $\overline F$. Letting $t\downarrow0$, every 
point of $C_{r,y}$ is approached by both $F$ and $\R^n\setminus\overline F$, so $C_{r,y}\subset\partial F$. 
Since $F$ is open, $F_{(r,y)}=\emptyset$, proving \eqref{eq:rigidity-pointwise-distribution} also in this case.

In particular, $F$ is $\mu$-distributed, and $\Omega_\mu$ is the same set whether computed from $E$ 
or from~$F$.

\smallskip
\noindent\textbf{Step 2: the rotational Jacobi field.}
Set
\[
M\coloneqq\partial^*F,\qquad M^+\coloneqq M\cap(\R^2_+\times\R^{n-2}),\qquad X(x,y)\coloneqq(-x_2,x_1,0).
\]
Since $\partial F\setminus M\subset\{0\}$, we have $M^+=\partial F\cap(\R^2_+\times\R^{n-2})$.

\smallskip
\noindent\emph{Slab cases.} If $M$ is disconnected, Proposition~\ref{prop:components_structure} implies 
that $F=\{|z_j|<a\}$ or $F=\{|z_j|>a\}$ for some $a>0$ and $j\in\{1,\dots,n\}$, where $z=(x_1,x_2,y)$. 
If $j\geq3$, then $F$ is rotationally invariant in the $(x_1,x_2)$-plane, every fiber of $F$ is empty or 
full, and $\Omega_\mu=\emptyset$ by \eqref{eq:rigidity-pointwise-distribution}. If $j\in\{1,2\}$ and $F=\{|z_j|<a\}$,
the fiber $F_{(r,y)}$ is full (up to two points) for $r\leq a$ and a non-empty proper union of arcs for $r>a$, so 
\eqref{eq:rigidity-pointwise-distribution} gives $\Omega_\mu=(a,\infty)\times\R^{n-2}$, which is open and 
connected; moreover $M^+=\{z_j=a\}\cap(\R^2_+\times\R^{n-2})$, so that 
$\Pi(M^+)=(a,\infty)\times\R^{n-2}=\Omega_\mu$, and $\nu^F=\pm e_j$ gives 
$\nu^F_{x\parallel}=\pm\big(e_j-(e_j\cdot\hat x)\hat x\big)\neq0$ on $M^+$, since 
$\hat x\neq\pm e_j$ there. Similarly if $F=\{|z_j|>a\}$. Thus, all the conclusions of the theorem hold for $F$ 
in place of $E$.

\smallskip
\noindent\emph{Connectedness of $M^+$.} Assume henceforth that $M$ is connected. The continuous folding map 
$(x_1,x_2,y)\mapsto(|x_1|,|x_2|,y)$ maps $M$ onto $M\cap\{x_1\geq0,\,x_2\geq0\}$, so this latter set is 
connected. For $i=1,2$ and $p\in M\cap\{x_i=0\}$, the reflection symmetry $\nu^F(A_ip)=A_i\nu^F(p)$ gives 
$\nu_i^F(p)=0$, hence
\[
\nabla_Mx_i(p)=e_i-\nu_i^F(p)\,\nu^F(p)=e_i .
\]
Thus $x_1$ and $x_2$ are submersions on $M$ near $\{x_1=0\}$ and $\{x_2=0\}$ respectively, and their 
gradients are linearly independent at the points of $M\cap\{x_1=x_2=0\}$. Therefore 
$M\cap\{x_1\geq0,\,x_2\geq0\}$ is a manifold with corners, whose interior is $M^+$. As explained in Step~2 of Lemma~\ref{lem:graph}, the interior of a 
connected manifold with corners is connected. Hence $M^+$ is connected, and non-empty because $M\neq\emptyset$.

\smallskip
\noindent\emph{Sign of $u$.} Define
\[
u\coloneqq X\cdot\nu^F\quad\text{on }M .
\]
In angular coordinates $x=|x|\,\omega_\theta$, $\theta\in(0,\pi/2)$, on the open quadrant 
$\R^2_+\times\R^{n-2}$, one has $X=\partial_\theta$ and, by \eqref{def:F}, $\chi_{F_\mu}$ is nonincreasing 
in $\theta$. Since $\diver X=0$ and $F=_{\mathcal H^n}F_\mu$, for every $\varphi\in C^\infty_c(\R^2_+\times\R^{n-2})$ 
with $\varphi\geq0$
\[
\langle X\cdot D\chi_F,\varphi\rangle
=-\int_{\R^n}\chi_{F_\mu}\,X\cdot\nabla\varphi\,dz
=-\int_{(0,\infty)\times\R^{n-2}}\int_0^{\pi/2}\chi_{F_\mu}(r\omega_\theta,y)\,
\partial_\theta\big(\varphi(r\omega_\theta,y)\big)\,d\theta\; r\,dr\,dy\leq0 ,
\]
i.e. $X\cdot D\chi_F\leq0$ as a Radon measure in $\R^2_+\times\R^{n-2}$. On the other hand 
$X\cdot D\chi_F=-u\,\mathcal H^{n-1}\mres M$. Therefore $u\geq0$ $\mathcal H^{n-1}$-a.e. on $M^+$, and hence 
everywhere on $M^+$ by continuity.

\smallskip
\noindent\emph{The Jacobi equation.} Let $J$ be the skew-symmetric endomorphism of $\R^n$ with $Je_1=e_2$, 
$Je_2=-e_1$, $Je_k=0$ for $k\geq3$, so that $X(z)=Jz$, and for $z\in M$ write 
$z^T\coloneqq z-(z\cdot\nu^F)\nu^F$, $X^T\coloneqq X-u\,\nu^F$. By Proposition~\ref{prop:partial-reg}, 
$\Hh_F-z\cdot\nu^F=\lambda$ on $M$. For $\tau\in T_zM$, since $D_\tau X=J\tau$ and 
$D_\tau\nu^F=B_F(\tau,\cdot)$,
\[
D_\tau u=J\tau\cdot\nu^F+B_F(\tau,X^T),
\qquad
D_\tau\Hh_F=D_\tau(z\cdot\nu^F)=B_F(\tau,z^T).
\]
Choosing $\tau=z^T$ in the first identity, where 
$Jz^T\cdot\nu^F=Jz\cdot\nu^F-(z\cdot\nu^F)\,J\nu^F\cdot\nu^F=u$, and $\tau=X^T$ in the second, the symmetry 
of $B_F$ gives
\[
z\cdot\nabla_Mu=u+B_F(z^T,X^T)=u+X^T\cdot\nabla_M\Hh_F .
\]
Moreover $\Delta_M(Jz)=J\Delta_Mz=-\Hh_F\,J\nu^F$ is orthogonal to $\nu^F$, and, for an orthonormal tangent 
frame $\{\tau_i\}$, $\sum_iD_{\tau_i}(Jz)\cdot D_{\tau_i}\nu^F=\sum_{i,k}B_F(\tau_i,\tau_k)\,J\tau_i\cdot\tau_k=0$ 
by the symmetry of $B_F$ and the skew-symmetry of $J$. Hence, by the identity 
$\Delta_M\nu^F+|B_F|^2\nu^F=\nabla_M\Hh_F$ used in the proof of Lemma~\ref{lem:graph},
\[
\Delta_Mu=X\cdot\Delta_M\nu^F=X^T\cdot\nabla_M\Hh_F-|B_F|^2u
=z\cdot\nabla_Mu-\big(|B_F|^2+1\big)u,
\]
i.e.,
\begin{equation}\label{eq:rigidity-rotational-Jacobi}
\Delta_Mu-z\cdot\nabla_Mu+\bigl(|B_F|^2+1\bigr)u=0
\qquad\text{on }M .
\end{equation}
In other words, \eqref{eq:rigidity-rotational-Jacobi} is the linearization of the Gaussian 
constant-mean-curvature equation along the rotations $R_t\coloneqq e^{tJ}$, which preserve the Gaussian weight 
and hence the equation and its constant right-hand side; no admissibility of the rotated sets in the 
symmetric variational problem is used. The term $u$, absent in \eqref{eq:calculation-nu-n}, comes from 
$D_\tau X=J\tau\neq0$.

\smallskip
\noindent\emph{Minimum principle.} By \eqref{eq:rigidity-rotational-Jacobi} and $u\geq0$,
\[
\Delta_Mu-z\cdot\nabla_Mu=-\bigl(|B_F|^2+1\bigr)u\leq0
\qquad\text{on }M^+ .
\]
The operator $\Delta_M-z\cdot\nabla_M$ has no zeroth-order term, so the strong minimum principle applies 
locally to the nonnegative supersolution $u$: an interior zero forces $u$ to vanish in a neighborhood, and 
the connectedness of $M^+$ yields
\[
u\equiv0\quad\text{on }M^+
\qquad\text{or}\qquad
u>0\quad\text{on }M^+ .
\]

\smallskip
\noindent\textbf{Step 3: openness and connectedness of $\Omega_\mu$.}
Suppose first that $u\equiv0$ on $M^+$. Since $u(A_ip)=-u(p)$ for $i=1,2$, and $M\cap\{x_1x_2=0\}$ is 
$\mathcal H^{n-1}$-negligible (it is a proper submanifold of $M$ by the transversality proved in Step~2), 
we get $u=0$ $\mathcal H^{n-1}$-a.e. on $M$, hence $X\cdot D\chi_F=-u\,\mathcal H^{n-1}\mres M=0$ in $\R^n$. 
Let $R_t=e^{tJ}$ be the flow of $X$. For $\varphi\in C^\infty_c(\R^n)$, since $R_t$ commutes with $J$ 
and $\diver X=0$,
\[
\frac{d}{dt}\int_{\R^n}\chi_F(z)\,\varphi(R_tz)\,dz
=\int_{\R^n}\chi_F(z)\,X(z)\cdot\nabla(\varphi\circ R_t)(z)\,dz
=-\langle X\cdot D\chi_F,\varphi\circ R_t\rangle=0 .
\]
Hence $R_tF=_{\mathcal H^n}F$ for every $t$. Orthogonal maps commute with the passage to density-one 
representatives, and $F=F^{(1)}$; thus $R_tF=F$ pointwise for every $t$, i.e. $F$ is rotationally 
invariant in the $(x_1,x_2)$-plane. Every circular fiber of $F$ is then empty or full, and 
\eqref{eq:rigidity-pointwise-distribution} gives $\Omega_\mu=\emptyset$.

Suppose instead that $u>0$ on $M^+$. For every $p\in M^+$, the circular fiber through $p$ is tangent to 
$X(p)$, hence transversal to $M$ at $p$, and thus contains positive-length portions of both open phases $F$ 
and $\R^n\setminus\overline F$. By \eqref{eq:rigidity-pointwise-distribution}, 
$0<\mu(\Pi(p))<\mathcal H^1_\gamma(\partial B(|x|))$, i.e. $\Pi(M^+)\subset\Omega_\mu$.
Conversely, let $(r,y)\in\Omega_\mu$. By \eqref{eq:rigidity-pointwise-distribution}, 
$0<\mathcal H^1(F_{(r,y)})<2\pi r$; by Lemma~\ref{lemma:connected-slices} applied to the minimizer $F$, and 
by symmetry, the quadrant section $\widetilde F_{(r,y)}$ is a non-empty, connected, relatively open proper 
subset of the closed quarter-circle containing one of its endpoints, hence it is $\{\theta\in[0,\theta^*)\}$ 
or $\{\theta\in(\theta^*,\pi/2]\}$ for some $\theta^*\in(0,\pi/2)$. The point $(r\omega_{\theta^*},y)$ lies 
in $\overline F\setminus F=\partial F$ and in the open quadrant, hence in $M^+$ since $r>0$. Consequently
\begin{equation}\label{eq:rigidity-projection-identity}
\Omega_\mu=\Pi(M^+).
\end{equation}
In particular $\Omega_\mu\neq\emptyset$ is connected, as the continuous image of the connected set $M^+$. 
Moreover, for $p=(x,y)\in M^+$,
\[
X(p)\cdot\nu^F(p)=u(p)>0,
\]
so $\Pi|_{M^+}$ is a local diffeomorphism onto its image; therefore 
$\Omega_\mu$ is open. Finally, $u(p)=|x|\,\nu^F_x(p)\cdot x_\parallel$ and 
$\nu^F_{x\parallel}=(\nu^F_x\cdot x_\parallel)\,x_\parallel$ show that $\nu^F_{x\parallel}(p)\neq0$ for every 
$p\in M^+$.

Together with the slab cases, we have proved: either $F$ is rotationally invariant in the $(x_1,x_2)$-plane 
and $\Omega_\mu=\emptyset$, or $\Omega_\mu\neq\emptyset$, \eqref{eq:rigidity-projection-identity} holds, and 
$\nu^F_{x\parallel}\neq0$ on $M^+=\partial F\cap(\R^2_+\times\R^{n-2})$. In both cases $\Omega_\mu$ is open 
and connected. In the first case we clearly have also $E=F$ since the circular sections of $F$ are empty or full.

\smallskip
\noindent\textbf{Step 4: rigidity of the original minimizer} in the case $\Omega_\mu\neq\emptyset$.
Set
\[
\Omega_\mu^1\coloneqq\{(r,y)\in(0,\infty)\times\R^{n-2}:(r,0,y)\in E\},
\qquad
\Omega_\mu^2\coloneqq\{(r,y)\in(0,\infty)\times\R^{n-2}:(0,r,y)\in E\}.
\]
These sets are open, being inverse images of the open set $E$ under continuous maps. Let $(r,y)\in\Omega_\mu$. 
Then $\widetilde E_{(r,y)}$ is non-empty and is not the whole closed quarter-circle, so by 
Lemma~\ref{lemma:connected-slices} it contains exactly one of the two axis points (containing both would 
force, by connectedness, the whole quarter-circle). Therefore
\[
\Omega_\mu=(\Omega_\mu^1\cap\Omega_\mu)\,\dot\cup\,(\Omega_\mu^2\cap\Omega_\mu).
\]
This is a covering of the connected set $\Omega_\mu$ by two disjoint relatively open sets, so one of them 
is empty.
In the first case $E=F$, while in the second case $E=R(F)$.
\end{proof}

\begin{corollary}\label{cor:no radial}
Let $E\subset\R^n$ be a minimizer of the Gaussian perimeter among $n$-symmetric
sets with prescribed Gaussian measure. For $i\neq j$, let $\mathcal{R}_{i,j}$
denote the double-arc rearrangement in the coordinate plane $(x_i,x_j)$, with
arcs centered on the $x_i$-axis. 
When applied to a minimizer, $\mathcal{R}_{i,j}(E)$ is understood as the density-one representative 
of the corresponding double-arc rearrangement, as in Theorem~\ref{theorem rig}.
Then, after a permutation of the coordinate
axes, one has
\begin{equation}
\label{eq:all_pairwise_rearranged}
E=\mathcal{R}_{i,j}(E)
\qquad\text{for every } i<j .
\end{equation}
Moreover, if there exists a point $p\in\partial E\cap\R^n_+$ such that
\[
\nu^E(p)=\frac{p}{|p|}
\qquad
\left(\text{respectively, }\nu^E(p)=-\frac{p}{|p|}\right),
\]
then $E$ (respectively, $\R^n\setminus\overline E$) is the open
$n$-dimensional ball of radius $|p|$ centred at the origin.
\end{corollary}

\begin{proof}

Let $M\coloneqq \partial^* E$ and $M^+\coloneqq M\cap \R^n_+$. 
For $i\neq j$, consider the infinitesimal rotation in the plane $(x_i,x_j)$,
\[
X_{i,j}(x)\coloneqq -x_j e_i+x_i e_j,
\]
and set
\[
u_{i,j}\coloneqq X_{i,j}\cdot\nu^E=x_i\nu^E_j-x_j\nu^E_i
\qquad\text{on } M .
\]

\smallskip
\noindent\textbf{Step 1.} 
By Theorem~\ref{theorem rig}, applied in the plane $(x_i,x_j)$, either $E=\mathcal{R}_{i,j}(E)$ or $E=\mathcal{R}_{j,i}(E)$. 
Moreover, the sign argument in the proof of Theorem~\ref{theorem rig} gives the following alternative on the positive
orthant of the plane $(x_i,x_j)$:
\[
E=\mathcal{R}_{i,j}(E) \quad\Longrightarrow\quad u_{i,j}\geq 0,
\]
whereas
\[
E=\mathcal{R}_{j,i}(E) \quad\Longrightarrow\quad u_{i,j}\leq 0.
\]
If $u_{i,j}\equiv0$, then $E$ is rotationally invariant in the plane
$(x_i,x_j)$, and the two choices are equivalent.

Given $p\in M^+$, for every $i$ define
\[
\rho_i\coloneqq \frac{\nu^E_i(p)}{p_i}.
\]
Since $p_i>0$ for all $i$, we have
\[
u_{i,j}(p)
=
p_i p_j\left(\rho_j-\rho_i\right).
\]

Hence the choice of rearrangement axis in the plane $(x_i,x_j)$ is
determined by the ordering of $\rho_i$ and $\rho_j$.
Indeed, if $\rho_i<\rho_j$, then $u_{i,j}(p)>0$, so the preceding
sign implications rule out $E=\mathcal{R}_{j,i}(E)$ and yield
$E=\mathcal{R}_{i,j}(E)$. Similarly, if $\rho_j<\rho_i$, then
$E=\mathcal{R}_{j,i}(E)$.

If $\rho_i=\rho_j$, then $u_{i,j}(p)=0$. Since $p_i,p_j>0$,
this means that the angular component of $\nu^E(p)$ in the
$(x_i,x_j)$-plane vanishes, contradicting alternative~(ii) of
Theorem~\ref{theorem rig}, applied in that plane.
Thus alternative~(i) holds: $E$ is rotationally invariant in the
$(x_i,x_j)$-plane, and consequently
\[
E=\mathcal{R}_{i,j}(E)=\mathcal{R}_{j,i}(E).
\]
We have therefore proved that
\[
\rho_i\leq\rho_j \ \Longrightarrow\ E=\mathcal{R}_{i,j}(E),
\qquad
\rho_j\leq\rho_i \ \Longrightarrow\ E=\mathcal{R}_{j,i}(E),
\]
with both rearrangements leaving $E$ unchanged when $\rho_i=\rho_j$.

After relabeling the coordinate axes so that
\[
\rho_1\leq \rho_2\leq\cdots\leq \rho_n,
\]
we obtain $E=\mathcal{R}_{i,j}(E)$ for every $i<j$, which proves
\eqref{eq:all_pairwise_rearranged}.

\smallskip
\noindent\textbf{Step 2.} 
If $M$ is not connected, then by Proposition~\ref{prop:components_structure}
the set $E$ is a symmetric slab or its complement. This is incompatible with
the existence of a point $p\in M^+$ such that $\nu^E(p)=p/|p|$. 
Hence we may assume that $M$ is connected. As in the proof of
Theorem~\ref{theorem rig}, the folded positive orthant part $M^+$ is connected.

By the previous step we can assume that $E=\mathcal{R}_{i,j}(E)$ for $i<j$.
Since $u_{i,j}(p)=0$, using the Jacobi equation and  the strong maximum principle as in the proof of Theorem \ref{theorem rig}, 
we get that $u_{i,j}\equiv 0$ on $M^+$.
By analyticity and symmetry, this identity extends to the whole of $M$.
This means that the infinitesimal rotation field $X_{i,j}$ is tangent to $M$. 

Since this holds true for any pair $\{i,j\}$ with $i< j$, $M$ has to be a hypersphere. 
\end{proof}

%%%%%%%%%%%%%%%%%% Appendix
\appendix

\section{The problem in one dimension}\label{app A}

The conjecture about the shape of the symmetric Gaussian isoperimetric problem can be easily verified in one dimension.

\begin{proposition}[Isoperimetric problem on the real line with symmetry] \label{prop:line_sym}
Let $m\in(0,1)$. On the real line $\R$, the Gaussian perimeter has a minimum among the sets of 
Gaussian length $m$, which are symmetric with respect to the origin. 
More precisely, 
\begin{itemize}
\item If $m < 1/2$, the unique perimeter minimizer is the symmetric tails $(-\infty,-a)\cup(a,\infty)$ for a suitable $a>0$.
\item If $m > 1/2$, the unique perimeter minimizer is the centered interval $(-a,a)$ for a suitable $a>0$.
\item If $m = 1/2$, both configurations have equal perimeter.
\end{itemize}
\end{proposition}

\begin{proof}
The existence of minimizers is standard and follows by the direct method in the calculus of variations in the $BV$ setting.
Let $E\subset\R$ be a minimizer of the Gaussian perimeter. Since $\chi_E$ belongs to $SBV_{\textnormal{loc}}(\R)$, the set $E$ is equivalent,
up to a set of measure zero, to a countable and locally finite union of intervals $(a_h,b_h)$ with disjoint closures. We have two cases.

\smallskip
\noindent\textit{The origin does not belong to any of the intervals $(a_h,b_h)$}. If we write $E^-\coloneqq (-\infty,0)\cap E$ and $E^+\coloneqq (0,\infty)\cap E$,
we have $P_\gamma(E)=P_\gamma(E^-)+P_\gamma(E^+)$. Necessarily, $E^-$ must be the interval $(-\infty,-a)$, with $a>0$ such that 
$\gamma((-\infty,-a))=m/2$, because this is the minimal configuration of the Gaussian isoperimetric problem. By symmetry, $E^+$ must be the interval $(a,\infty)$.

\smallskip
\noindent\textit{The origin belongs to one of the intervals $(a_h,b_h)$}. Instead of $E$ we consider its complement, reducing to the previous case (with mass $(1-m)$).

\smallskip
Having established that the only candidates are the symmetric tails and the centered interval, a direct computation completes the classification.
\end{proof}

Because of this dichotomy, we cannot rearrange an $n$-symmetric set $E$ by substituting its one-dimensional slices with the minimal one-dimensional configuration. 
The following example in dimension 2 illustrates this failure.

\begin{example} \label{ex:counterexample}
Let $a > 0$ be the threshold such that the centered interval $(-a,a)$ has
Gaussian measure $1/2$. Consider the $2$-symmetric set $E \subset \R^2$
defined by
\begin{equation*}
E \coloneqq
\left((-a-\e,a+\e)\times[-1,1]\right)
\cup
\left((-a+\e,a-\e)\times
\bigl((-\infty,-1)\cup(1,\infty)\bigr)\right),
\end{equation*}
where $\e>0$ is sufficiently small.

For $|y|\leq1$, the horizontal slice is
$E_y=(-a-\e,a+\e)$, which has measure $m_+>1/2$.
According to Proposition \ref{prop:line_sym}, this is the unique minimizer.

For $|y|>1$, the horizontal slice is
$E_y=(-a+\e,a-\e)$, which has measure $m_-<1/2$.
According to Proposition \ref{prop:line_sym}, the unique minimizer for this
mass is the union of symmetric tails
$(-\infty,-b)\cup(b,\infty)$, with $b>a+\e$.

If we construct a rearranged set $F$ by replacing each slice with its unique
$1$-dimensional minimizer, we obtain
\begin{equation*}
F_y=
\begin{cases}
(-a-\e,a+\e) & \text{for } |y|\leq1,\\
(-\infty,-b)\cup(b,\infty) & \text{for } |y|>1.
\end{cases}
\end{equation*}
At the interfaces $y=\pm1$, the set $F$ undergoes a topological
change, effectively turning ``inside out'', see Figure
\ref{fig:obstruction}.
This discontinuity introduces horizontal boundary components with
significant Gaussian weight, whereas the original set $E$ has only small
horizontal shoulders of total length $4\e$ at each interface.
Moreover, as $\e\to0$, we have $b\to a$, so the difference between the
vertical contributions to the perimeters tends to zero, while the excess
horizontal contribution tends to $2e^{-1/2}$.
Consequently, $P_\gamma(F)>P_\gamma(E)$ for $\e>0$ sufficiently small,
showing that optimizing slices individually does not minimize the total
perimeter.
\end{example}

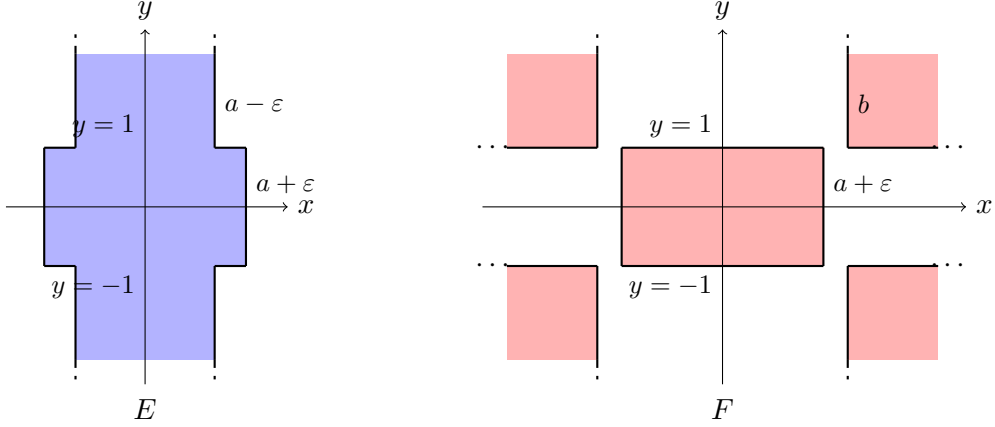
\begin{figure}[ht]
\centering
\begin{tikzpicture}[scale=0.92]
% Schematic values: a-e = 1.00, a+e = 1.45, b = 1.80,
% and the levels y=+/-1 are drawn at y=+/-0.85.

% Original set E
\begin{scope}[xshift=-4.15cm]
  \fill[blue!30] (-1.45,-0.85) rectangle (1.45,0.85);
  \fill[blue!30] (-1.00,0.85) rectangle (1.00,2.20);
  \fill[blue!30] (-1.00,-2.20) rectangle (1.00,-0.85);

  \draw[thick] (-1.45,-0.85) -- (-1.45,0.85);
  \draw[thick] ( 1.45,-0.85) -- ( 1.45,0.85);
  \draw[thick] (-1.00,0.85) -- (-1.00,2.20);
  \draw[thick] ( 1.00,0.85) -- ( 1.00,2.20);
  \draw[thick] (-1.00,-2.20) -- (-1.00,-0.85);
  \draw[thick] ( 1.00,-2.20) -- ( 1.00,-0.85);
  \draw[thick] (-1.45, 0.85) -- (-1.00, 0.85);
  \draw[thick] ( 1.00, 0.85) -- ( 1.45, 0.85);
  \draw[thick] (-1.45,-0.85) -- (-1.00,-0.85);
  \draw[thick] ( 1.00,-0.85) -- ( 1.45,-0.85);

  \draw[thick,dashed] (-1.00,2.20) -- (-1.00,2.48);
  \draw[thick,dashed] ( 1.00,2.20) -- ( 1.00,2.48);
  \draw[thick,dashed] (-1.00,-2.20) -- (-1.00,-2.48);
  \draw[thick,dashed] ( 1.00,-2.20) -- ( 1.00,-2.48);

  \draw[->] (-2.00,0) -- (2.05,0) node[right] {$x$};
  \draw[->] (0,-2.55) -- (0,2.55) node[above] {$y$};
  \node at (0,-2.90) {$E$};
  \node[above left] at (0,0.85) {\small $y=1$};
  \node[below left] at (0,-0.85) {\small $y=-1$};
  \node[right] at (1.00,1.48) {\small $a-\e$};
  \node[right] at (1.45,0.32) {\small $a+\e$};
\end{scope}

% Rearranged set F
\begin{scope}[xshift=4.15cm]
  \fill[red!30] (-1.45,-0.85) rectangle (1.45,0.85);
  \fill[red!30] (-3.10,0.85) rectangle (-1.80,2.20);
  \fill[red!30] ( 1.80,0.85) rectangle ( 3.10,2.20);
  \fill[red!30] (-3.10,-2.20) rectangle (-1.80,-0.85);
  \fill[red!30] ( 1.80,-2.20) rectangle ( 3.10,-0.85);

  \draw[thick] (-1.45,-0.85) -- (-1.45,0.85);
  \draw[thick] ( 1.45,-0.85) -- ( 1.45,0.85);
  \draw[thick] (-1.80,0.85) -- (-1.80,2.20);
  \draw[thick] ( 1.80,0.85) -- ( 1.80,2.20);
  \draw[thick] (-1.80,-2.20) -- (-1.80,-0.85);
  \draw[thick] ( 1.80,-2.20) -- ( 1.80,-0.85);
  \draw[thick,dashed] (-1.80,2.20) -- (-1.80,2.48);
  \draw[thick,dashed] ( 1.80,2.20) -- ( 1.80,2.48);
  \draw[thick,dashed] (-1.80,-2.20) -- (-1.80,-2.48);
  \draw[thick,dashed] ( 1.80,-2.20) -- ( 1.80,-2.48);

  \draw[thick] (-1.45, 0.85) -- ( 1.45, 0.85);
  \draw[thick] (-3.10, 0.85) -- (-1.80, 0.85);
  \draw[thick] ( 1.80, 0.85) -- ( 3.10, 0.85);
  \draw[thick] (-1.45,-0.85) -- ( 1.45,-0.85);
  \draw[thick] (-3.10,-0.85) -- (-1.80,-0.85);
  \draw[thick] ( 1.80,-0.85) -- ( 3.10,-0.85);

  \node at (-3.27,0.85) {$\cdots$};
  \node at ( 3.27,0.85) {$\cdots$};
  \node at (-3.27,-0.85) {$\cdots$};
  \node at ( 3.27,-0.85) {$\cdots$};

  \draw[->] (-3.45,0) -- (3.50,0) node[right] {$x$};
  \draw[->] (0,-2.55) -- (0,2.55) node[above] {$y$};
  \node at (0,-2.90) {$F$};
  \node[above left] at (0,0.85) {\small $y=1$};
  \node[below left] at (0,-0.85) {\small $y=-1$};
  \node[right] at (1.45,0.32) {\small $a+\e$};
  \node[right] at (1.80,1.48) {\small $b$};
\end{scope}
\end{tikzpicture}
\caption{The set $E$ (left) has only short horizontal boundary at $y=\pm1$. Slice-by-slice
minimization produces $F$ (right), which switches from a centered interval to symmetric tails 
at both interfaces and therefore creates horizontal boundary of order one in Gaussian weight.}
\label{fig:obstruction}
\end{figure}

\begin{remark}
This obstruction arises because the optimal configuration in 
Proposition~\ref{prop:line_sym} depends discontinuously on the mass $m$ at the 
threshold $m = 1/2$. In contrast, the double-arc rearrangement of 
Theorem~\ref{theorem ineq} uses circular slices, where the optimal symmetric 
configuration (two antipodal arcs) is independent of the mass, cf.\ 
Lemma~\ref{lem:isop 2 symm}.
\end{remark}

%%%%%%%%%%%%%%%%%%%%%%%%%%%%%%%%%%%%%%%%%%%%%%%%%%%%%%%%%%%%%%%%%%%%%%%%%%%%%%%%%%%%%%%%%%%%

\medskip

\section{Obstruction}\label{app B}

It would be useful to generalize the double-arc rearrangement, by foliating
an $n$-symmetric set $E$ using spheres, or more generally using $k$-dimensional spheres, instead of circles.
More precisely, for every point in $\R^n$,  $n \geq 3$, we write $(x, y)$, with 
$x = (x_1, x_2, x_3) \in \R^3$ and $y = (y_1, \ldots, y_{n-3}) \in \R^{n-3}$.
For every $r > 0$, we set \(\partial B (r) \coloneqq \{ x \in \R^3 \colon |x| = r \}\).

Given a Borel set $E$, for every $(r, y) \in (0, \infty) \times \R^{n-3}$ we define the 
spherical slice $E_{(r, y)} $ as the subset of $\R^3$ given by
\begin{equation*}
E_{(r, y)} \coloneqq \{ x \in \R^3: x \in \partial B (r) \text{ and } (x, y) \in E\},
\end{equation*}
and define the \emph{spherical distribution} of $E$ as the function
\begin{equation*} 
\mu (r, y)\coloneqq \mathcal{H}^2_{\gamma} ( E_{(r, y)}), \qquad \text{ for every } (r, y) \in (0, \infty) \times \R^{n-3}.
\end{equation*}
The idea would be to substitute $E_{(r,y)}$ with the minimizer of the symmetric isoperimetric problem on the sphere $\partial B(r)$, 
under the constraint of having $\mathcal{H}^2_{\gamma}$ measure equal to $\mu$.

Since the Gaussian density is constant on spheres centered at the origin, it is enough to
analyze the problem on the unit sphere $\Ss^2$, endowed with the standard Hausdorff measure.

The needed classification follows from the work of Viana~\cite{Viana} on the isoperimetric
problem in real projective spaces $\mathbb{RP}^{n+1}$. Equivalently, Viana's result implies
that, among antipodal-invariant regions of $\Ss^{n+1}$ with fixed volume, the least perimeter
is attained by sets of the form
\[
\left\{x\in \Ss^{n+1}:\sum_{i=1}^r x_i^2>a\right\},
\]
for some $r\in\{1,\ldots,n+1\}$ and $a\in\R_+$. The class of sets which are symmetric with
respect to the coordinate hyperplanes is contained in the class of antipodal-invariant sets. 
Moreover, the minimizers above are themselves symmetric with respect to all coordinate hyperplanes. 
Hence Viana's classification immediately yields the corresponding classification in the smaller
$n$-symmetric class.

We record the statement in the form needed below.

\smallskip

\begin{proposition}[Isoperimetric problem on the sphere with symmetry] \label{prop:sphere_sym}
Let $m\in(0,4\pi)$. On the sphere $\Ss^2$, the perimeter has a minimum among the sets of area $\mathcal{H}^2(E)=m$,
which are $3$-symmetric (symmetric with respect to the three coordinate planes). 
More precisely, 
\begin{itemize}
\item If $m < 2\pi$, the unique perimeter minimizer is the union of two antipodal geodesic caps.
\item If $m > 2\pi$, the unique perimeter minimizer is a symmetric equatorial belt (the complement of two antipodal caps).
\item If $m = 2\pi$, both configurations have equal perimeter.
\end{itemize}
\end{proposition}

\medskip
The fact that the shape of the symmetric isoperimetric problem on the sphere depends on the mass provides an obstruction to our idea:
the rearrangement proposed above is not going to decrease the global Gaussian perimeter in general. 

\begin{example} \label{ex:counterexample2}
Let $E\coloneqq (-a,a)\times\R^2$ be a symmetric slab. Taking $a$ large enough, $E$ is the only minimizer of the $3$-symmetric Gaussian isoperimetric problem (see \cite[Corollary 1]{BJ}).
If we foliate $E$ by spheres $\partial B (r)$, the spherical slice $E_r$ is a full sphere for $r<a$
and an equatorial belt for $r>a$ (see Figure \ref{fig:obstruction2}). 
However, for $r$ large enough, the equatorial belt is no longer the minimizer on $\partial B (r)$ 
under the constraint to have measure $\mu(r)$. If we substitute the slice $E_r$ with the two antipodal caps (or any other shape), we change the global shape of $E$
and we strictly increase the Gaussian perimeter. 
\end{example}

\begin{figure}
\begin{center}
\includegraphics[width=12cm]{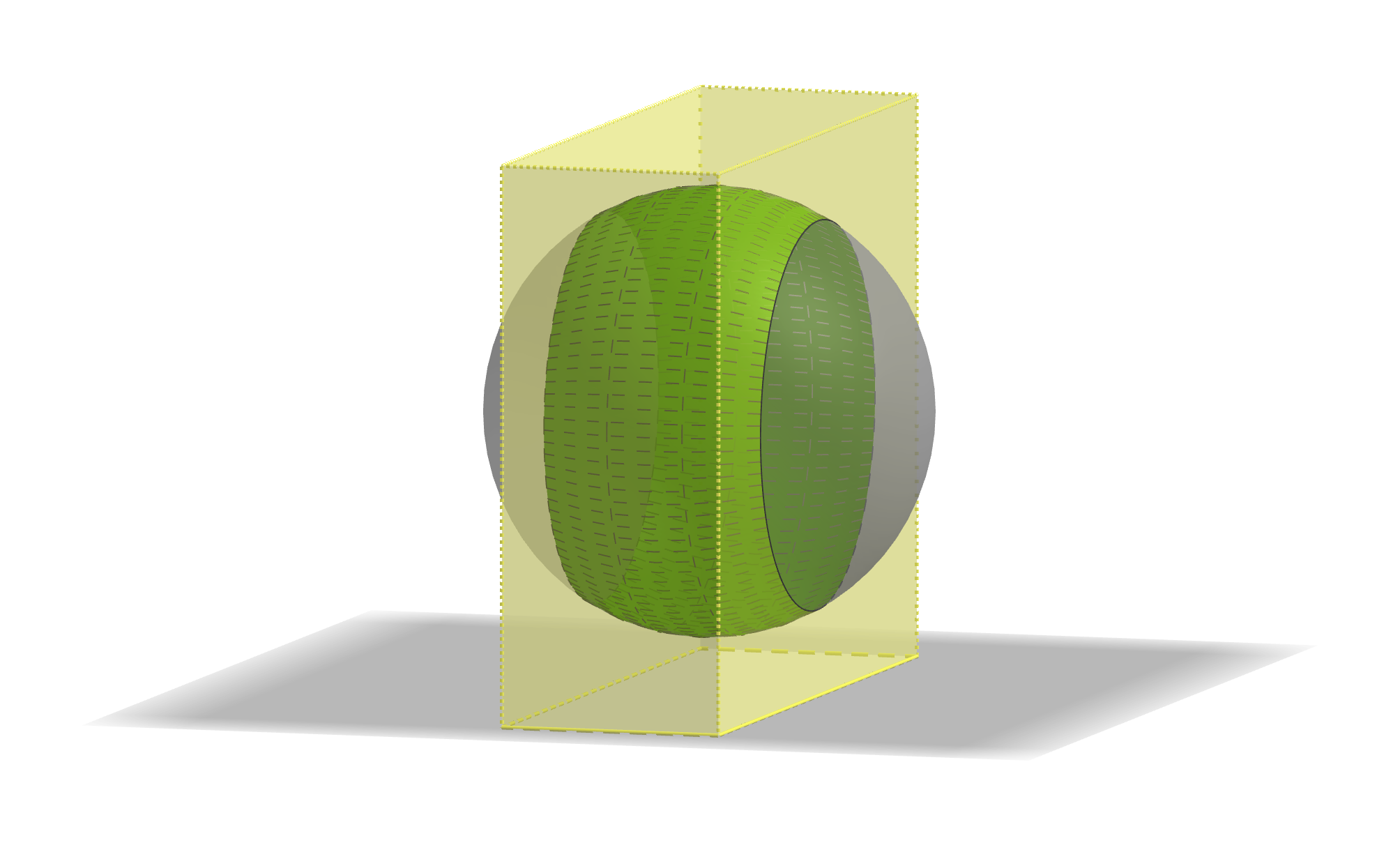}
\end{center}
\caption{The intersection of a symmetric slab with a sphere is an equatorial belt.}
\label{fig:obstruction2}
\end{figure}

\begin{remark}
Viana's classification also bears on a different strategy: rather than substituting the
spherical slices of a fixed $\R^n$ as above, one lifts the problem to spheres of growing
dimension, in the spirit of the classical spherical derivation of the Gaussian isoperimetric
inequality (see, e.g., \cite{Ledoux1994}). Let $\Ss^N_{\sqrt N}\subset\R^{N+1}$ be the sphere of radius $\sqrt N$
and let $\Pi_{N+1,n}\colon\R^{N+1}\to\R^n$ be the projection onto the first $n$
coordinates. An $n$-symmetric (hence centrally symmetric) set $E \subset \R^n$ lifts to the antipodally
invariant set $\Pi_{N+1,n}^{-1}(E)\cap\Ss^N_{\sqrt N}$, whose normalized measure converges
to $\gamma(E)$ as $N \to \infty$. Since the $n$-symmetric class lies inside the antipodally invariant one and
Viana's minimizers $\{x:\sum_{i\le r}x_i^2>a\}$ are themselves $n$-symmetric, one could hope to
recover the symmetric Gaussian minimizers in the Poincar\'e limit $N\to\infty$. 

A possible difficulty in this limiting argument is the lack of an a priori bound on the dimension needed to 
describe the spherical minimizers. Since complementation exchanges the ranks $r_N$ and $N+1-r_N$, the relevant quantity is
$\min\{r_N,N+1-r_N\}$.

For $\gamma(E)\in(0,1/2)$, boundedness of this quantity would be consistent with Heilman's conjectural picture of finite-dimensional attainment~\cite{He2}, which motivates the cylindrical hypothesis $E=E'\times\R$.

At mass $1/2$, the minimizing ranks themselves can always be chosen to diverge. Indeed, if $r$ is a minimizing rank on 
$\Ss_{\sqrt N}^{N}\subset\R^{N+1}$, then so is $N+1-r$, since the complementary region has the same mass and perimeter. 
Choosing the larger rank in each complementary pair gives $r_N\ge(N+1)/2$, and hence $r_N\to\infty$.

This observation alone does not establish an obstruction to the Gaussian limit: if $N+1-r_N$ remains bounded, these regions 
still admit descriptions involving a bounded number of coordinates. The possible obstruction is instead that both $r_N$ and 
$N+1-r_N$ might diverge. 
\end{remark}

%%%%%%%%%%%%%%%%%%%%%%%%%%%%%%%%%%%%%%%%%%%%%%%%%%%%%%%%%%%%%%%%%%%%%%%

\medskip

\section{Structural reduction in dimension two}\label{app C}

In this section, we address the symmetric Gaussian isoperimetric problem in the planar case $n=2$.
Unlike the one-dimensional setting, where the classification is straightforward, the two-dimensional problem presents a richer geometric structure
among the possible competitors. A set $E$ that is stationary for the Gaussian perimeter subject to a volume constraint must have a boundary $\partial E$ 
satisfying the constant mean curvature equation:
\begin{equation} \label{eq:GMC_intro}
\kappa(x) - x \cdot \nu^E(x) = \lambda,
\end{equation}
where $\nu^E$ is the exterior unit normal and
$\kappa=\diver_{\partial E}\nu^E$ is the corresponding signed curvature.
While circles centered at the origin and straight lines (forming strips) are the most immediate solutions to \eqref{eq:GMC_intro}, 
they are not the only ones. A notable class of non-trivial solutions is given by the so-called \emph{Chang curves}~\cite{Chang}.
These are compact curves that satisfy \eqref{eq:GMC_intro} with $\lambda < 0$ and exhibit non-trivial rotational 
symmetry (see Figure~\ref{fig:chang_curves}).
The support function $h$ satisfies the ODE 
\begin{equation} \label{eq:support_ode}
h'' + h = \frac{1}{h + \lambda}.
\end{equation}
Chang curves with $l$-fold symmetry arise as solutions bifurcating from the trivial circle solution $h \equiv r_0$ at critical values of the radius. Linearizing \eqref{eq:support_ode} leads to the equation $\phi'' + (1+r_0^2)\phi = 0$. The kernel of the linearized operator is non-trivial if and only if $\sqrt{1+r_0^2} = k$ 
for some integer $k \ge 2$; for $l$-fold symmetry, one considers the branch corresponding to the eigenfunction $\cos(l\theta)$.

It is known that Chang curves $\Gamma_l$ with $l$-fold symmetry, $l \ge 6$, are unstable critical points for 
the Gaussian isoperimetric problem. However, instability in the general case does not automatically rule them out 
for the \emph{symmetric} problem, as the directions of instability might break the symmetry constraint. 
We can prove via Theorem \ref{theorem ineq} that for the specific class of 2-symmetric sets, a Gaussian isoperimetric set 
cannot have a Chang curve of parameter $l \ge 4$ as boundary.

\begin{lemma}\label{lem:chang_circle_intersection}
Let $l\geq4$ be even, let $\Gamma_l\subset\R^2$ be a
noncircular Chang curve with $l$-fold rotational symmetry, and let
$E_l$ be the convex region enclosed by $\Gamma_l$, positioned so
that $E_l$ is $2$-symmetric. Set
\[
    r_{\min}\coloneqq\min_{x\in\Gamma_l}|x|,
    \qquad
    r_{\max}\coloneqq\max_{x\in\Gamma_l}|x|.
\]
For every $r\in(r_{\min},r_{\max})$, the circular section
$E_l\cap\partial B(r)$ has at least $l$ components of positive
length, separated by gaps of positive length. Since $l\geq4$,
these sections cannot coincide, even up to $\mathcal H^1$-null
sets, with two antipodal arcs. Hence the equality condition
\textup{(i)} in Theorem~\ref{theorem ineq} fails on an interval
of radii of positive length, and therefore
\[
    P_\gamma(F_\mu)<P_\gamma(E_l),
\]
where $\mu$ is the circular distribution of $E_l$.
In particular, $E_l$ cannot be a minimizer of the Gaussian
perimeter among $2$-symmetric sets with the same Gaussian measure.
\end{lemma}

\begin{figure}
\begin{center}
\includegraphics[width=15.5cm]{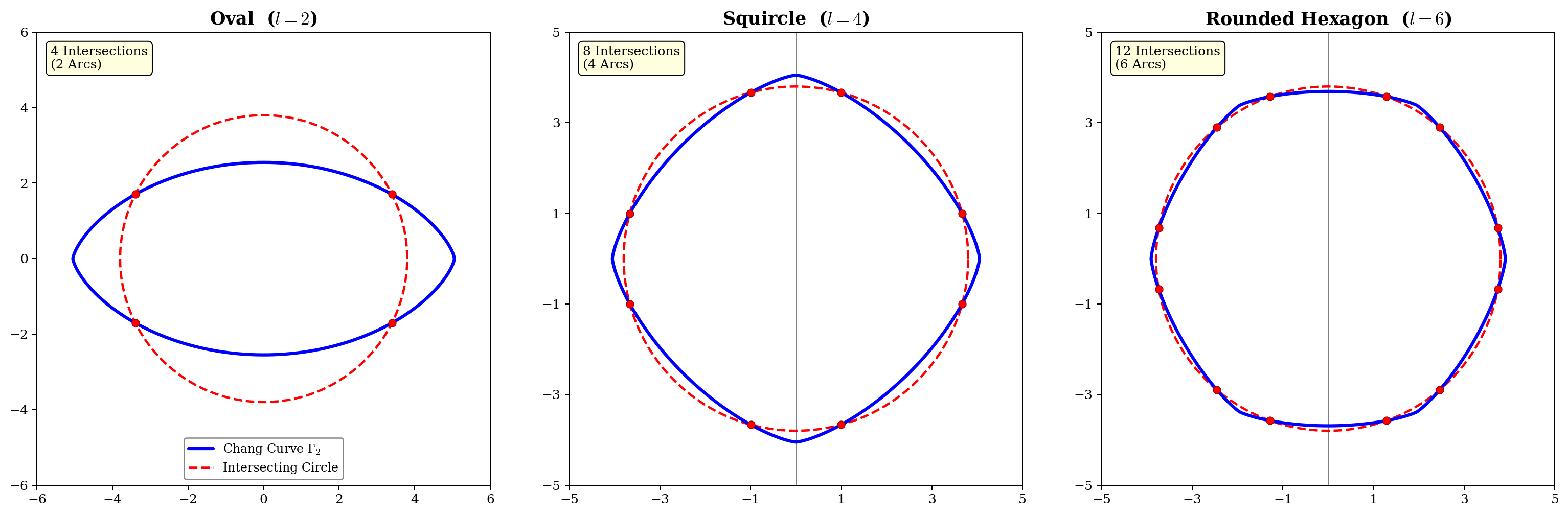}
\end{center}
\caption{A schematic illustration of Chang curves $\Gamma_l$ for $l = 2, 4, 6$. 
A generic circle $\partial B(r)$ with $r_{\min} < r < r_{\max}$ intersects $\Gamma_l$ 
in exactly $2l$ points (marked in red). For $l \geq 4$, this yields at least $8$ intersection points.}
\label{fig:chang_curves}
\end{figure}

\smallskip

\begin{proposition}
Let \(E\subset\R^2\) be a \(2\)-symmetric minimizer of the Gaussian
perimeter under fixed Gaussian area. Then, up to replacing \(E\) by \(E^c\), $E$ has one of the following configurations:
\begin{enumerate}[label=(\roman*)]
\item a symmetric strip; 
\item a disk centered at the origin;
\item a bounded open convex set such that $\nu^E(p)\neq p/|p|$ for any $p\in\partial E\cap\R^2_+$.
\end{enumerate}
\end{proposition}

\begin{proof}
Let \(M\coloneqq\partial E\). 
By Proposition~\ref{prop:components_structure}, if \(M\) is
disconnected, then \(E\) is a symmetric strip or its complement.
We may therefore assume that \(M\) is connected.

We first prove that \(M\) is compact. Let \(\nu^E\) be the exterior
unit normal and let \(\kappa=\operatorname{div}_M\nu^E\).
By \eqref{eq:GMC_intro},
\[
    \kappa-x\cdot\nu^E=\lambda
    \qquad\text{on }M.
\]
Along any local arclength parametrization \(x=x(s)\), with
\(\tau=\partial_sx\), we have
\[
    \partial_s\nu^E=\kappa\tau.
\]
Differentiating the Euler--Lagrange equation gives
\[
    \partial_s\kappa
    =
    \partial_s(x\cdot\nu^E)
    =
    \kappa\,x\cdot\tau,
\]
and hence
\[
    \partial_s\left(\kappa e^{-\frac{|x|^2}{2}}\right)
    =
    e^{-\frac{|x|^2}{2}}
    \left(\partial_s\kappa-\kappa x\cdot\tau\right)
    =
    0.
\]
Since \(M\) is connected, there exists \(c\in\R\) such that
\[
    \kappa e^{-\frac{|x|^2}{2}}=c
    \qquad\text{on }M.
\]

We claim that \(c\neq0\). Otherwise \(\kappa\equiv0\), so
\(\partial_s\nu^E=0\), and connectedness implies that \(\nu^E\)
is constant on \(M\). On the other hand, the central symmetry
\(E=-E\) gives
\[
    \nu^E(-x)=-\nu^E(x)
    \qquad\text{for every }x\in M,
\]
a contradiction.

Combining the first integral with the Euler--Lagrange equation,
we obtain
\[
    |c|e^{\frac{|x|^2}{2}}
    =
    |\kappa(x)|
    =
    |x\cdot\nu^E(x)+\lambda|
    \leq |x|+|\lambda|
    \qquad\text{for every }x\in M.
\]
Since \(c\neq0\), this inequality forces \(|x|\) to remain bounded
on \(M\). As \(M\) is closed, it is compact.

Thus \(M\) is a simple closed real analytic curve. By the Jordan
curve theorem, \(E\) is either the bounded domain enclosed by \(M\)
or its exterior. After replacing \(E\) by the density-one
representative of \(E^c\), if necessary, we may assume that \(E\)
is the bounded domain. Complementation preserves minimality for
the complementary Gaussian area and reverses the signs of
\(\nu^E,\kappa,\lambda\), and \(c\); we retain the same notation
for the resulting quantities.

Orient \(M\) counterclockwise. The turning-number theorem gives
\[
    2\pi
    =
    \int_M\kappa\,ds
    =
    c\int_M e^{\frac{|x|^2}{2}}\,ds,
\]
and hence \(c>0\). Therefore \(\kappa>0\) everywhere on \(M\),
so \(E\) is strictly convex, with compact closure.

Set \(M_+\coloneqq M\cap\R^2_+\). If there exists
\(p\in M_+\) such that
\[
    \nu^E(p)=\frac{p}{|p|},
\]
then Corollary~\ref{cor:no radial} implies that \(E\) is the
disk centered at the origin with radius \(|p|\), giving
alternative~\textup{(ii)}. Otherwise,
\[
    \nu^E(p)\neq\frac{p}{|p|}
    \qquad\text{for every }p\in M_+,
\]
and alternative~\textup{(iii)} holds.
\end{proof}

\smallskip

\begin{remark}
Let $\Gamma_2 \subset \R^2$ be the Chang oval, i.e., the Chang 
curve with $2$-fold rotational symmetry bifurcating from the 
circle of radius $r_0 = \sqrt{3}$.
In \cite{Chang}, the existence of the Chang oval as a solution of the constant mean 
curvature equation \eqref{eq:GMC_intro} has been established for every 
$\lambda\in(\lambda_{\mathrm{bif}},0)$, where $\lambda_{\mathrm{bif}} = -2/\sqrt{3}$.
As a possible minimizer, it falls in the third case of the previous proposition.
\end{remark}

%\section*{Declarations}

%\subsection*{Data availability}
%Data sharing is not applicable to this article as no datasets
%were generated or analysed during the current study.

%\subsection*{Conflict of interest}
%The authors have no relevant financial or non-financial
%interests to disclose.

%\clearpage

\medskip
% Include bibliography
\bibliographystyle{abbrv}
%\nocite{*}
\bibliography{biblio2arc}

@book{AFP,
  author    = {Ambrosio, Luigi and Fusco, Nicola and Pallara, Diego},
  title     = {Functions of Bounded Variation and Free Discontinuity Problems},
  series    = {Oxford Mathematical Monographs},
  publisher = {Oxford University Press},
  address   = {Oxford},
  year      = {2000}
}

@article{BCF,
  author  = {Barchiesi, Marco and Cagnetti, Filippo and Fusco, Nicola},
  title   = {Stability of the {Steiner} symmetrization of convex sets},
  journal = {J. Eur. Math. Soc. (JEMS)},
  volume  = {15},
  number  = {4},
  year    = {2013},
  pages   = {1245--1278},
  doi     = {10.4171/JEMS/391}
}

@article{BBJ,
  author  = {Barchiesi, Marco and Brancolini, Alessio and Julin, Vesa},
  title   = {Sharp dimension free quantitative estimates for the Gaussian isoperimetric inequality},
  journal = {Ann. Probab.},
  volume  = {45},
  number  = {2},
  year    = {2017},
  pages   = {668--697},
  doi     = {10.1214/15-AOP1072}
}

@article{BJ,
  author  = {Barchiesi, Marco and Julin, Vesa},
  title   = {Symmetry of minimizers of a Gaussian isoperimetric problem},
  journal = {Probab. Theory Related Fields},
  volume  = {177},
  year    = {2020},
  pages   = {217--256},
  doi     = {10.1007/s00440-019-00947-9}
}

@article{Barthe,
  author  = {Barthe, Franck},
  title   = {An isoperimetric result for the Gaussian measure and unconditional sets},
  journal = {Bull. Lond. Math. Soc.},
  volume  = {33},
  number  = {4},
  year    = {2001},
  pages   = {408--416},
  doi     = {10.1017/S0024609301008141}
}

@article{BDGM,
  author  = {Bombieri, E. and De Giorgi, E. and Miranda, M.},
  title   = {Una maggiorazione a priori relativa alle ipersuperfici minimali non parametriche},
  journal = {Archive for Rational Mechanics and Analysis},
  volume  = {32},
  year    = {1969},
  pages   = {255--267},
}

@book{Bon,
  author    = {Bonnesen, T.},
  title     = {Les Probl{\`e}mes des isop{\'e}rim{\`e}tres et des is{\'e}piphanes},
  publisher = {Gautier-Villars},
  address   = {Paris},
  year      = {1929}
}

@article {CCDPM,
    AUTHOR = {Cagnetti, Filippo and Colombo, Maria and De Philippis, Guido
              and Maggi, Francesco},
     TITLE = {Rigidity of equality cases in {S}teiner's perimeter
              inequality},
   JOURNAL = {Anal. PDE},
  FJOURNAL = {Analysis \& PDE},
    VOLUME = {7},
      YEAR = {2014},
    NUMBER = {7},
     PAGES = {1535--1593},
      ISSN = {2157-5045,1948-206X},
   MRCLASS = {49K21},
  MRNUMBER = {3293444},
MRREVIEWER = {Sandro\ Zagatti},
       DOI = {10.2140/apde.2014.7.1535},
       URL = {https://doi.org/10.2140/apde.2014.7.1535},
}

@article{CCDPM2,
  author  = {Cagnetti, Filippo and Colombo, Maria and De Philippis, Guido and Maggi, Francesco},
  title   = {Essential connectedness and the rigidity problem for Gaussian symmetrization},
  journal = {J. Eur. Math. Soc. (JEMS)},
  volume  = {19},
  number  = {2},
  year    = {2017},
  pages   = {395--439},
  doi     = {10.4171/JEMS/669}
}

@article{CPS,
  author    = {Cagnetti, Filippo and Perugini, Matteo and St{\"o}ger, Daniel},
  title     = {Rigidity for perimeter inequality under spherical symmetrisation},
  journal   = {Calc. Var. Partial Differential Equations},
  volume    = {59},
  number    = {4},
  year      = {2020},
  eid       = {139},
  numpages  = {53},
  doi       = {10.1007/s00526-020-01786-6}
}

@article{Chang,
  author  = {Chang, Jui-En},
  title   = {{$1$}-dimensional solutions of the {$\lambda$}-self shrinkers},
  journal = {Geometriae Dedicata},
  volume  = {189},
  year    = {2017},
  pages   = {97--112},
  doi     = {10.1007/s10711-017-0219-z}
}

@article{CCF,
  author  = {Chleb{\'\i}k, Miroslav and Cianchi, Andrea and Fusco, Nicola},
  title   = {The perimeter inequality under {Steiner} symmetrization: cases of equality},
  journal = {Ann. of Math. (2)},
  volume  = {162},
  number  = {1},
  year    = {2005},
  pages   = {525--555},
  doi     = {10.4007/annals.2005.162.525}
}

@article{CL,
  author  = {Cicalese, Marco and Leonardi, Gian Paolo},
  title   = {Best constants for the isoperimetric inequality in quantitative form},
  journal = {J. Eur. Math. Soc. (JEMS)},
  volume  = {15},
  number  = {3},
  year    = {2013},
  pages   = {1101--1129},
  doi     = {10.4171/JEMS/387}
}

@article{DelgadinoMaggi,
  author  = {Delgadino, Matias Gonzalo and Maggi, Francesco},
  title   = {Alexandrov's theorem revisited},
  journal = {Analysis \& PDE},
  volume  = {12},
  number  = {6},
  year    = {2019},
  pages   = {1613--1642},
  doi     = {10.2140/apde.2019.12.1613},
  mrnumber = {3921314}
}

@article{Gam,
  author  = {Gambicchia, Chiara},
  title   = {The double spherical cap rearrangement of planar sets},
  journal = {Rend. Sem. Mat. Univ. Padova},
  year    = {2026},
  note    = {Published online first},
  doi     = {10.4171/RSMUP/196}
}

@book{Giusti,
  author    = {Giusti, Enrico},
  title     = {Minimal Surfaces and Functions of Bounded Variation},
  series    = {Monographs in Mathematics},
  volume    = {80},
  publisher = {Birkh{\"a}user},
  address   = {Boston, MA},
  year      = {1984},
  doi       = {10.1007/978-1-4684-9486-0}
}

@article{HPRR,
  title   = {The periodic isoperimetric problem},
  author  = {Hauswirth, Laurent and P{\'e}rez, Joaqu{\'i}n and Romon, Pascal and Ros, Antonio},
  journal = {Transactions of the American Mathematical Society},
  volume  = {356},
  number  = {5},
  pages   = {2025--2047},
  year    = {2004},
  doi     = {10.1090/S0002-9947-03-03362-2}
}

@article{He1,
  author  = {Heilman, Steven},
  title   = {Symmetric convex sets with minimal Gaussian surface area},
  journal = {Amer. J. Math.},
  volume  = {143},
  number  = {1},
  year    = {2021},
  pages   = {53--94},
  doi     = {10.1353/ajm.2021.0000}
}

@unpublished{He2,
  author = {Heilman, Steven},
  title  = {Convex Cylinders and the Symmetric Gaussian Isoperimetric Problem},
  note   = {Preprint, arXiv:2204.12003},
  year   = {2022}
}

@incollection{Ledoux1994,
  author    = {Ledoux, Michel},
  title     = {Isoperimetry and {G}aussian Analysis},
  booktitle = {Lectures on Probability Theory and Statistics ({S}aint-{F}lour, 1994)},         
  series    = {Lecture Notes in Math.},     
  volume    = {1648},                       
  pages     = {165--294},                   
  publisher = {Springer, Berlin},           
  year      = {1996},                       
}

@book{Maggi,
  author    = {Maggi, Francesco},
  title     = {Sets of Finite Perimeter and Geometric Variational Problems},
  subtitle  = {An Introduction to Geometric Measure Theory},
  series    = {Cambridge Studies in Advanced Mathematics},
  volume    = {135},
  publisher = {Cambridge University Press},
  address   = {Cambridge},
  year      = {2012}
}

@article{Milman2025,
  author  = {Milman, Emanuel},
  title   = {Isoperimetric Inequalities on Slabs with Applications to Cubes and Gaussian Slabs},
  journal = {Communications on Pure and Applied Mathematics},
  year    = {2025},
  note    = {to appear},
}

@article{Morgan,
  author  = {Morgan, Frank},
  title   = {Regularity of area-minimizing surfaces in {3D} polytopes
             and of invariant hypersurfaces in {$\mathbb{R}^n$}},
  journal = {J. Geom. Anal.},
  volume  = {15},
  number  = {2},
  pages   = {321--341},
  year    = {2005},
  doi     = {10.1007/BF02922198},
}

@book{Morrey,
  author    = {Morrey, Charles B.},
  title     = {Multiple Integrals in the Calculus of Variations},
  series    = {Grundlehren der mathematischen Wissenschaften},
  volume    = {130},
  publisher = {Springer},
  address   = {Berlin and Heidelberg and New York},
  year      = {1966}
}

@article{Palais,
  author  = {Palais, Richard S.},
  title   = {The Principle of Symmetric Criticality},
  journal = {Communications in Mathematical Physics},
  volume  = {69},
  number  = {1},
  pages   = {19--30},
  year    = {1979},
  doi     = {10.1007/BF01941322}
}

@article{Viana,
  author  = {Viana, Celso},
  title   = {Isoperimetry and volume preserving stability in real projective spaces},
  journal = {Journal of Differential Geometry},
  volume  = {125},
  number  = {1},
  pages   = {187--205},
  year    = {2023}
}

\end{document}